\documentclass[12pt]{article}
\usepackage{setspace}
\usepackage{amssymb,amsmath,amsthm}
\usepackage{graphicx}
\usepackage[english]{babel}
\newtheorem{Lem}{Lemma}
\newtheorem{Thm}[Lem]{Theorem}
\newtheorem{Cor}[Lem]{Corollary}

\newtheorem{Prop}[Lem]{Proposition}

\begin{document}

\begin{center}
\textbf{AN OBSTRUCTION TO DECOMPOSABLE EXACT LAGRANGIAN FILLINGS}
\end{center}

\vspace{.01cm}
\begin{center}
\small{WATCHAREEPAN ATIPONRAT}
\end{center}

\vspace{.5cm}
\noindent
    A\scriptsize{BSTRACT.}\normalsize~ We study some properties of decomposable exact Lagrangian cobordisms between Legendrian links in $\mathbb{R}^3$ with the standard contact structure. In particular, for any decomposable exact Lagrangian filling $L$ of a Legendrian link $K$, we may obtain a normal ruling of $K$ associated with $L$. We prove that the associated normal rulings must have even number of clasps. As a result, we give a particular Legendrian $(4,-(2n+5))$-torus knot, for each $n \geq 0$, which does not have a decomposable exact Lagrangian filling because it has only 1 normal ruling and this normal ruling has odd number of clasps.
    
\vspace{1.3cm}
\begin{center}
\normalsize{1. I}\footnotesize{NTRODUCTION}
\end{center}

The standard contact structure on $\mathbb{R}^3$ is the kernel of the smooth 1-form $dz-ydx$ where the coordinate of $\mathbb{R}^3$ is $(x,y,z)$. Alternatively, this is an everywhere non-integrable plane field of the 3-dimensional Euclidean space. In our work, we will focus on Legendrian links which are smooth links in $\mathbb{R}^3$ everywhere tangent to the standard contact structure. For an interested reader, a good introduction for this subject could be found in \cite{Etnyre} and \cite{Geiges}.

In the case of smooth links in $\mathbb{R}^3$, we can consider smooth links as boundaries of surfaces. Similarly, we may consider Legendrian links as boundaries of particular surfaces in $\mathbb{R}^4$. To start with, we look at the symplectization $(\mathbb{R}^4 = \mathbb{R}_{t} \times \mathbb{R}^3 ,d(e^t(dz-ydx)))$ of $\mathbb{R}^3$ (here $t$ is the coordinate coming from the first $\mathbb{R}$ factor), i.e. a symplectic manifold derived from $\mathbb{R}^3$ with the standard contact structure. Next, a surface $L$ in $\mathbb{R}^4$ is called \textbf{Lagrangian} if $d(e^t(dz-ydx))|_L = 0$. In particular, it is \textbf{exact Lagrangian} if the smooth 1-form $e^t(dz-ydx)|_L$ is exact.

Suppose we have two Legendrian links $K_+$ and $K_-$ in $\mathbb{R}^3$. An exact Lagrangian surface $L$ in $\mathbb{R}^4$ is an \textbf{exact Lagrangian cobordism from $K_-$ to $K_+$} if there exist $T > 0$ such that the following holds:

(1) $L \cap ((-\infty,-T]\times \mathbb{R}^3)=(-\infty,-T]\times K_-$;

(2) $L \cap ([T,\infty)\times \mathbb{R}^3)=[T,\infty)\times K_+$;

(3) $L \cap ([-T,T]\times \mathbb{R}^3)$ is compact with boundary $K_{+} \cup K_{-}$; and

(4) $f|_{L \cap ((-\infty,-T)\times \mathbb{R}^3)}$ and $f|_{L \cap ((T,\infty)\times \mathbb{R}^3)}$ are constant functions if $df = e^t(dz-ydx)|_L$.

Additionally, $L$ is said to be an \textbf{exact Lagrangian filling of $K_+$} if, in particular, $K_- = \emptyset$. Exact Lagrangian fillings of Legendrian links are of interest in a many aspects. They are studied in several papers, for example, see \cite{Chantraine1}, \cite{EHK} and \cite{Hayden}. As mentioned in \cite{EHK}, the existence of an exact Lagrangian filling of a Legendrian link provide an augmentation of the Legendrian Contact Homology DGA of the link itself. So, it is natural to ask whether or not each Legendrian link $K$ is a boundary of an exact Lagrangian filling. The complete answer to this problem is not obvious. However, we do have some partial results. Hayden and Sabloff show in \cite{Hayden} that all positive knots possess Legendrian representatives with exact Lagrangian fillings. Furthermore, they give a conjecture that a smooth knot type has a Legendrian representative with exact Lagrangian fillings if and only if the knot type is quasi-positive and its HOMFLY bound is sharp. An example of a Legendrian knot which does not have exact Lagrangian filling is given in their paper as well.

In many cases, the main tool which is used to construct exact Lagrangian cobordisms, and hence exact Lagrangian fillings, is the following theorem.

\begin{Thm}
\label{movie}
\emph{(See \cite{Bourgeois}, \cite{Chantraine1}, \cite{EHK}, \cite{Rizell}).}
Let $K_+$ and $K_-$ be Legendrian links. Suppose that the front diagram of $K_+$ is obtained from the front diagram of $K_-$ via one of the following three moves:

(1) Legendrian isotopy, including regular isotopy and Legendrian Reidemeister moves;

(2) 0-handle, represented by the first row of Figure \ref{handlep};

(3) 1-handle, represented by the second row of Figure \ref{handlep}.

Then there exists an embedded exact Lagrangian cobordism from $K_-$ to $K_+$.
\end{Thm}

\begin{figure}
\begin{center}
\includegraphics[height=2in]{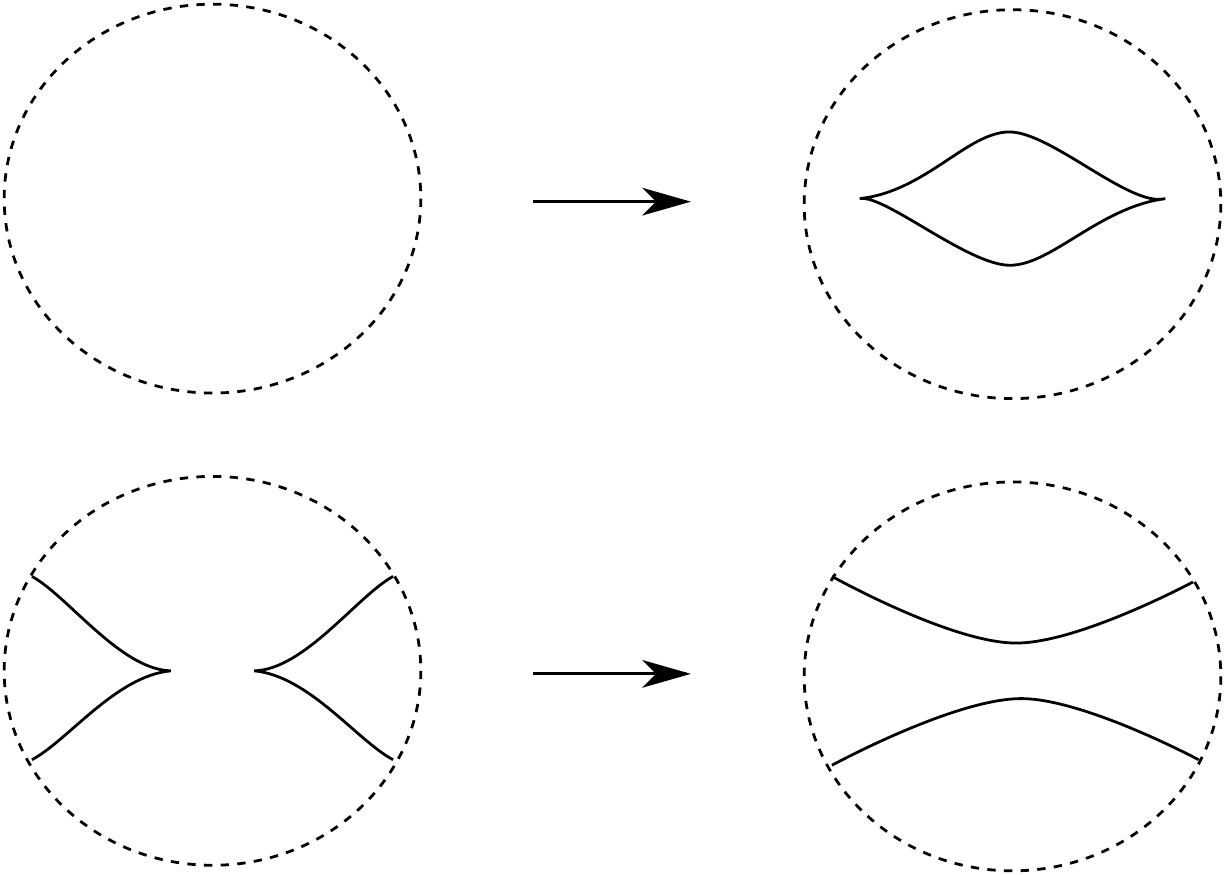}
\end{center}
\caption{Local modifications of front diagrams corresponding to 0-handle attachment and 1-handle attachment.}
\label{handlep}
\end{figure}

In the language of Theorem \ref{movie}, if there is a finite sequence\footnote{We will only consider sequences with length $\geq 1$.} of moves taking the front diagram of $K_-$ to the front diagram of $K_+$, then there exists an embedded exact Lagrangian cobordism from $K_-$ to $K_+$ arising from the composition of cobordisms associated with the moves. We say that this exact Lagrangian cobordism is \textbf{decomposable}. Throughout this paper, we will mainly study decomposable exact Lagrangian fillings since everything can be visualized as a finite sequence of moves in Theorem \ref{movie}.

Next, normal rulings are combinatorial objects which provide an invariant of Legendrian links, see $\cite{Chekanov Pushkar'}$. They can be easily obtained by considering front diagrams of Legendrian links. It will be very useful if we can establish the existence of decomposable exact Lagrangian fillings by just looking at normal rulings of Legendrian links because they are a lot easier to deal with.

The very first glimpse of the connection between exact Lagrangian fillings and normal rulings of Legendrian links has been revealed by \cite{EHK} and \cite{Sabloff}. As stated above, if a Legendrian link has an exact Lagrangian filling, then its contact homology DGA has an augmentation. Sabloff shows that if the contact homology DGA of a Legendrian link has an augmentation, then the link has a normal ruling (see \cite{Sabloff}). However, the converse is not guaranteed. That is, it is an open question whether a normal ruling implies the existence of a possibly non-orientable exact Lagrangian filling. In this work, we would like to partially answer this question by restricting ourselves to the case of decomposable exact Lagrangian fillings. In particular, we give a proof of the following lemma at the end of Section 2.

\begin{Lem}
\label{l5}
Given a decomposable exact Lagrangian filling of a Legendrian link $K$, there is a canonical normal ruling of $K$ associated with it.
\end{Lem}

By Lemma \ref{l5}, we can say that there is a specific normal ruling associated with every Legendrian link with a decomposable exact Lagrangian filling. Now, we can investigate some important characteristics of these associated normal rulings.

In order to do this, we introduce the notion of clasps of normal rulings in Section 3. This object allows us to prove the following theorem in Section 3.5.

\begin{Thm}
\label{key}
Given a decomposable exact Lagrangian filling $L$ of a Legendrian link $K$, the normal ruling associated with $L$ must have even number of clasps.
\end{Thm}

One of important consequences of Theorem \ref{key} is that every Legendrian link with a decomposable exact Lagrangian filling must have at least one normal ruling with an even number of clasps. Thus, if we want to show that there is a Legendrian link with no decomposable exact Lagrangian filling, if there is any, we just need to show that the link does not have a normal rulings with an even number of clasps. We employ this idea to answer the open problem in the following theorem.

\begin{Thm}
	\label{counter}
	For each $n \geq 0$, there is a particular Legendrian $(4,-(2n+5))$-torus knot which has a normal ruling but none of decomposable exact Lagrangian filling.
\end{Thm}

Finally, we note here that all results in this paper are coming from \cite{Met}.

\vspace{0.3cm}
\noindent
\textbf{Acknowledgments.} The author would like to acknowledge William Menasco for his support, which makes this work possible. In addition, the author would like to thank Lenhard Ng for an introduction to this topic. Finally, Chiang Mai University research funding provides financial support for this paper.

\vspace{0.5cm}
\begin{center}
\normalsize{2. N}\footnotesize{ORMAL RULINGS}
\end{center}

Suppose we have a front diagram $K$. By regular isotopy, we may assume from now on that its cusps and crossings have distinct $x$-coordinates. We consider a subset $\rho$ of the set of all crossings of $K$. Then we perform resolution, see Figure \ref{resolution}, at each crossing in $\rho$ so that we obtain a resulting front diagram $K'$. We call $\rho$ a \textbf{normal ruling} if the followings hold:

(1) each component of $K'$ has one left cusp, one right cusp and no self-intersections;

(2) horizontal strands at each resolution belong to different components in $K'$; and

(3) vertical slice at each resolution must appear in $K'$ as one of Figure \ref{normality}.

\begin{figure}
\begin{center}
\includegraphics[width=2.5in]{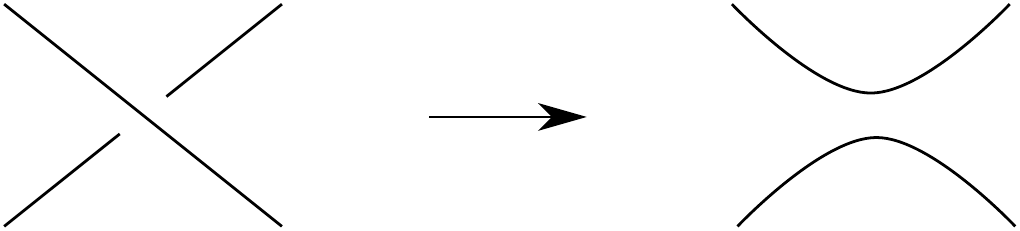}
\end{center}
\caption{Resolution at a crossing.}
\label{resolution}
\end{figure}

\begin{figure}
\begin{center}
\vspace{1cm}
\includegraphics[width=3.2in]{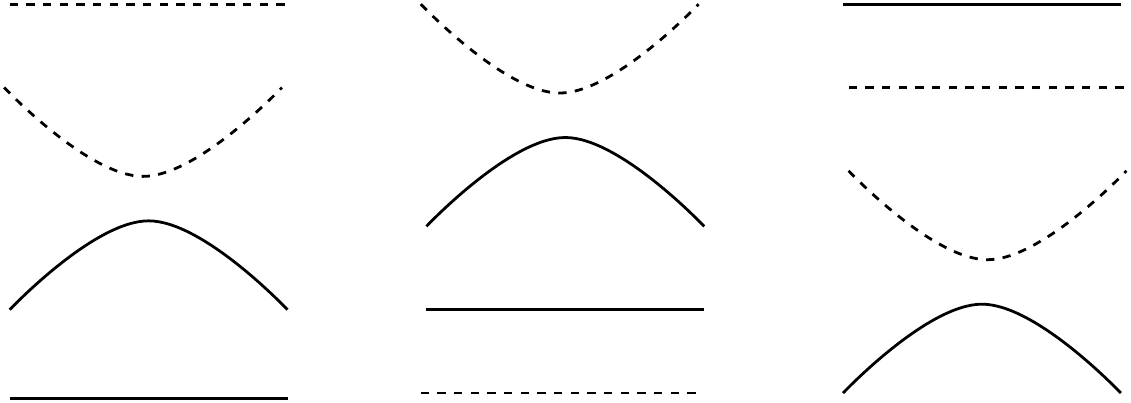}
\end{center}
\caption{Possible vertical slices at resolution when considering only two resulting components involved.}
\label{normality}
\end{figure}

If $\rho$ is a normal ruling, then all crossings in $\rho$ are called \textbf{switches} and $K'$ is the \textbf{resolution of $\rho$} while each component of $K'$ is named an \textbf{eye}. Moreover, (3) is the \textbf{normality condition}, and we say a Legendrian link has a normal ruling if its front diagram admitting a normal ruling.

For example, the knot $3_1$, as in Figure \ref{31}, has a normal ruling with all 3 crossings are switches. This is easy to verify by looking at its resolution in the right side of Figure \ref{31}. 

\begin{figure}
\begin{center}
\includegraphics[height=0.8in]{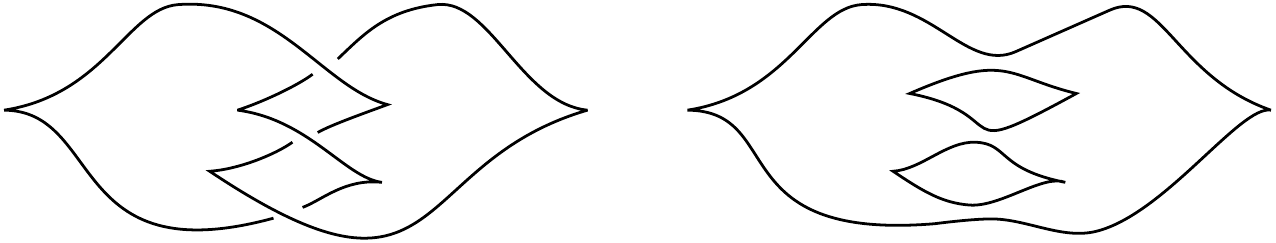}
\end{center}
\caption{The knot $3_1$ (left) and the resolution of its normal ruling (right).}
\label{31}
\end{figure}

Next, we state the fact that normal rulings offer us an invariant for Legendrian links.

\begin{Thm}
\label{Chekanov}
\emph{(See \cite{Chekanov Pushkar'}).}
Let $K$, $\widetilde{K}$ be Legendrian links. If they are Legendrian isotopic, then there is a one-to-one correspondence between their normal rulings. In particular, the number of normal rulings is invariant under Legendrian isotopy.
\end{Thm}

\begin{proof}
By induction, we only need to consider the case that $\widetilde{K}$ is obtained from $K$ by a single move of regular isotopy, R1, R2, or R3. 

First, any regular isotopy that could change a ruling must interchange $x$-coordinates of two crossings (see Figure \ref{reg}) because this might violate the normality condition. We may assume further that there is no other crossings/cusps in the vertical strips of Figure \ref{reg}. Then the correspondence would depend on the following two cases (we consider only two crossings appeared in Figure \ref{reg}. The rest of crossings are the same under the correspondence).
\vspace{0.2cm}

(1) Exactly two eyes are involved: $\emptyset \leftrightarrow \emptyset$, $\{a, b\} \leftrightarrow \{a', b'\}$, $\{a\} \leftrightarrow \{b'\}$, $\{b\} \leftrightarrow \{a'\}$.

(2) More than two eyes are involved: $\emptyset \leftrightarrow \emptyset$, $\{a, b\} \leftrightarrow \{a', b'\}$, $\{a\} \leftrightarrow \{a'\}$,  $\{b\} \leftrightarrow \{b'\}$. 
\vspace{-0.2cm}

Next, for Reidemeister moves, see Figure \ref{Rmoveid}, we have the following identifications under an assumption that there is no other crossings in every vertical strip, i.e. all the moves occur in very thin vertical strip (again, we consider only crossings appeared in Figure \ref{reg}. The rest of crossings are the same under the correspondence).

R1: $\emptyset \leftrightarrow \{a\}$;

R2: $\emptyset \leftrightarrow \emptyset$;

R3: $\emptyset \leftrightarrow \emptyset$, $\{a\} \leftrightarrow \{a'\}$, $\{b\} \leftrightarrow \{b'\}$, $\{c\} \leftrightarrow \{c'\}$, $\{a, b, c\} \leftrightarrow \{a', b', c'\}$, and

2-switches case A: $\{a, b\} \leftrightarrow \{b', c'\}$ as in Figure \ref{R3a},
 
2-switches case B: $\{a, c\} \leftrightarrow \{b', c'\}$ as in Figure \ref{R3b}, 

2-switches case C: $\{b, c\} \leftrightarrow \{a', b'\}$ as in Figure \ref{R3c}, 

2-switches case D: $\{b, c\} \leftrightarrow \{a', c'\}$ as in Figure \ref{R3d}.

Notice that case A and B cannot happen for the same Legendrian link with the rest switches identical. Similarly, case C and D cannot happen for the same Legendrian link with the rest switches identical. Hence the correspondence is bijective.
\end{proof}

\begin{figure}
\begin{center}
\vspace{0.7cm}
\includegraphics[width=4.8in]{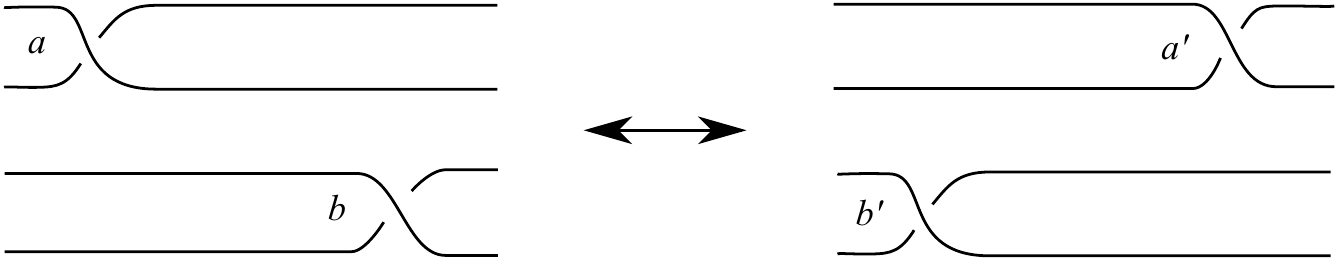}
\end{center}
\caption{Regular isotopy interchanging crossings.}
\label{reg}
\end{figure}

\begin{figure}[h]
\begin{center}
\includegraphics[width=4.9in]{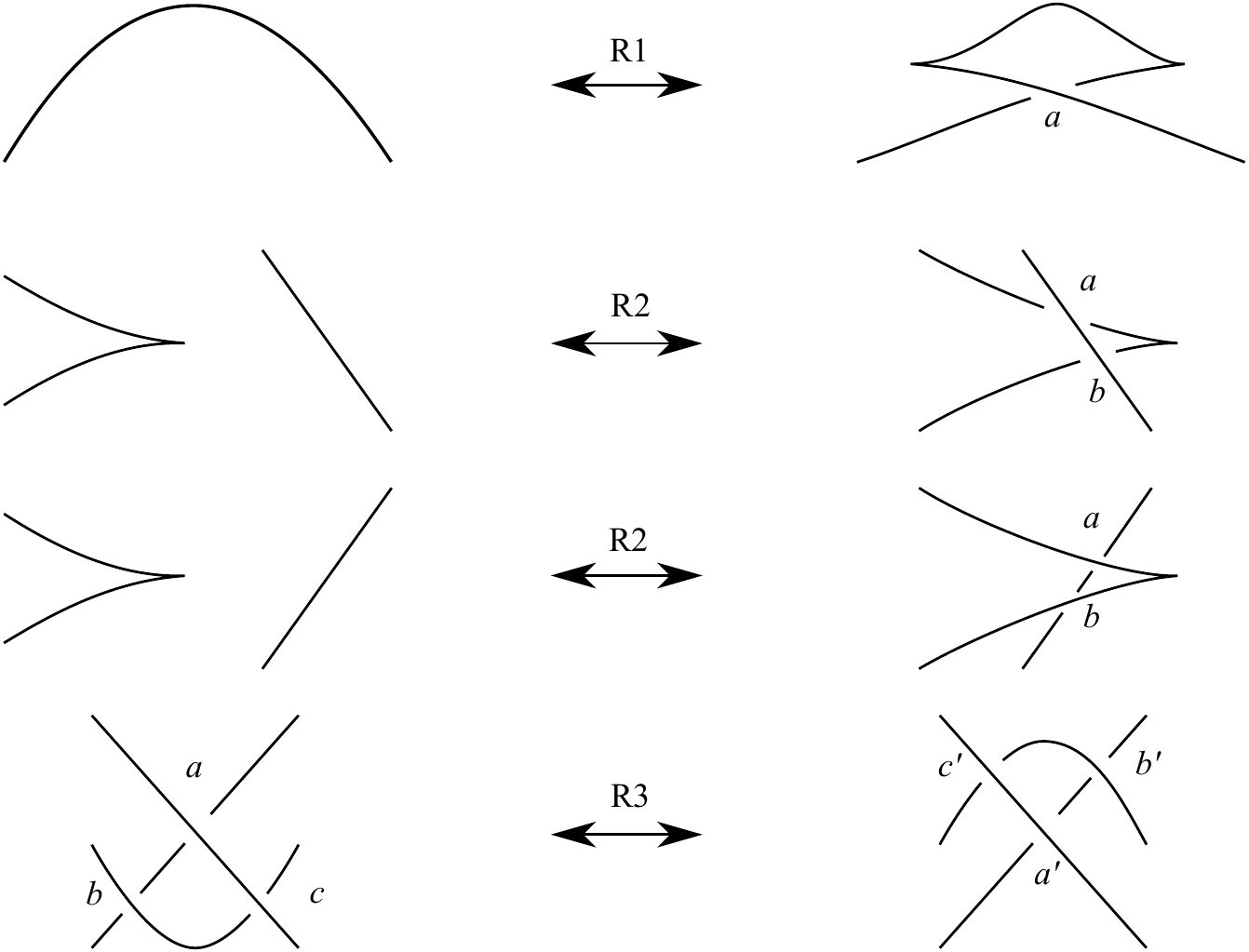}
\end{center}
\caption{Legendrian Reidemeister moves including the rotation of R1 and R2 by 180$^\circ$.}
\label{Rmoveid}
\end{figure}

\begin{figure}
\begin{center}
\includegraphics{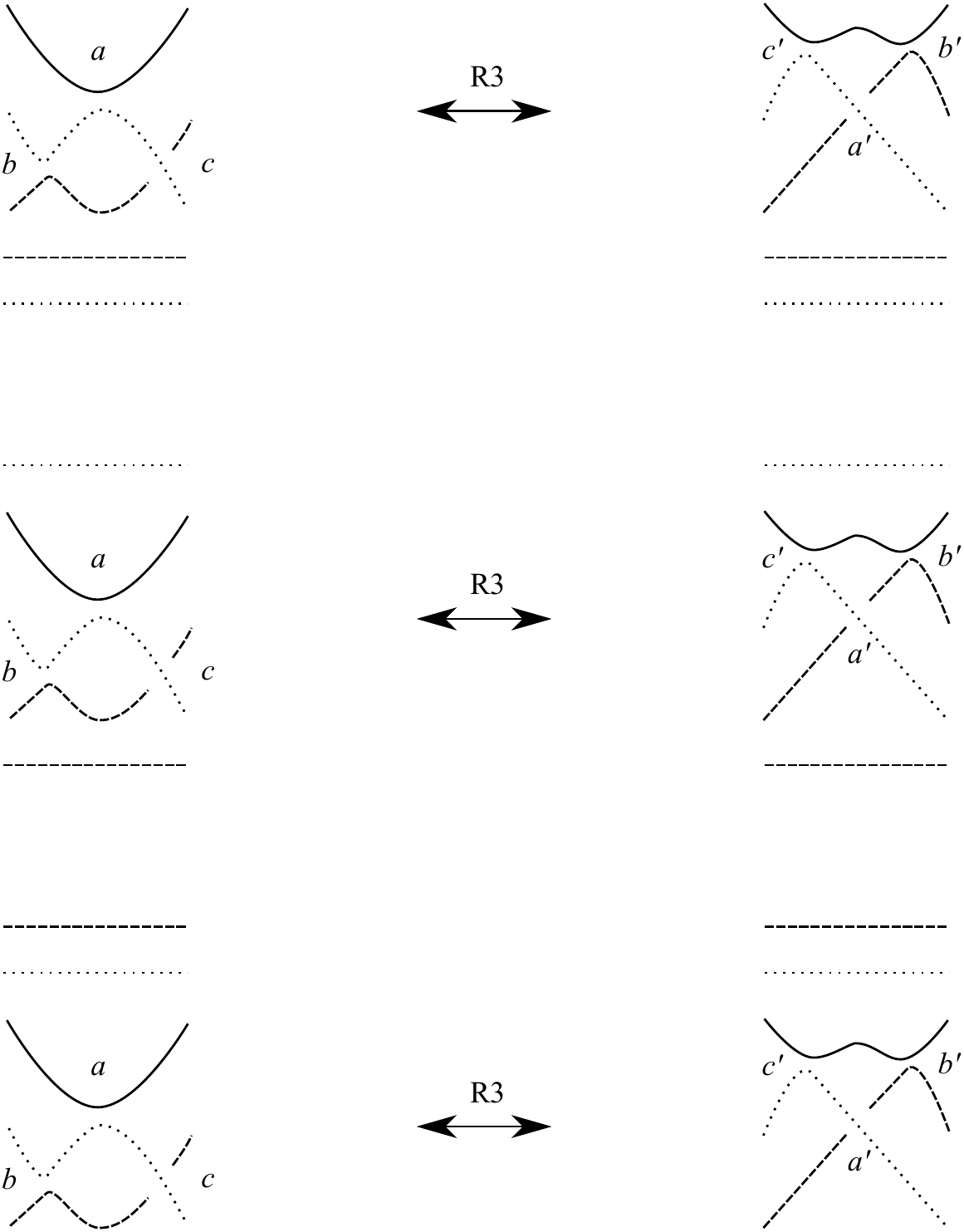}
\end{center}
\caption{Case A.}
\label{R3a}
\end{figure}

\begin{figure}
\begin{center}
\includegraphics{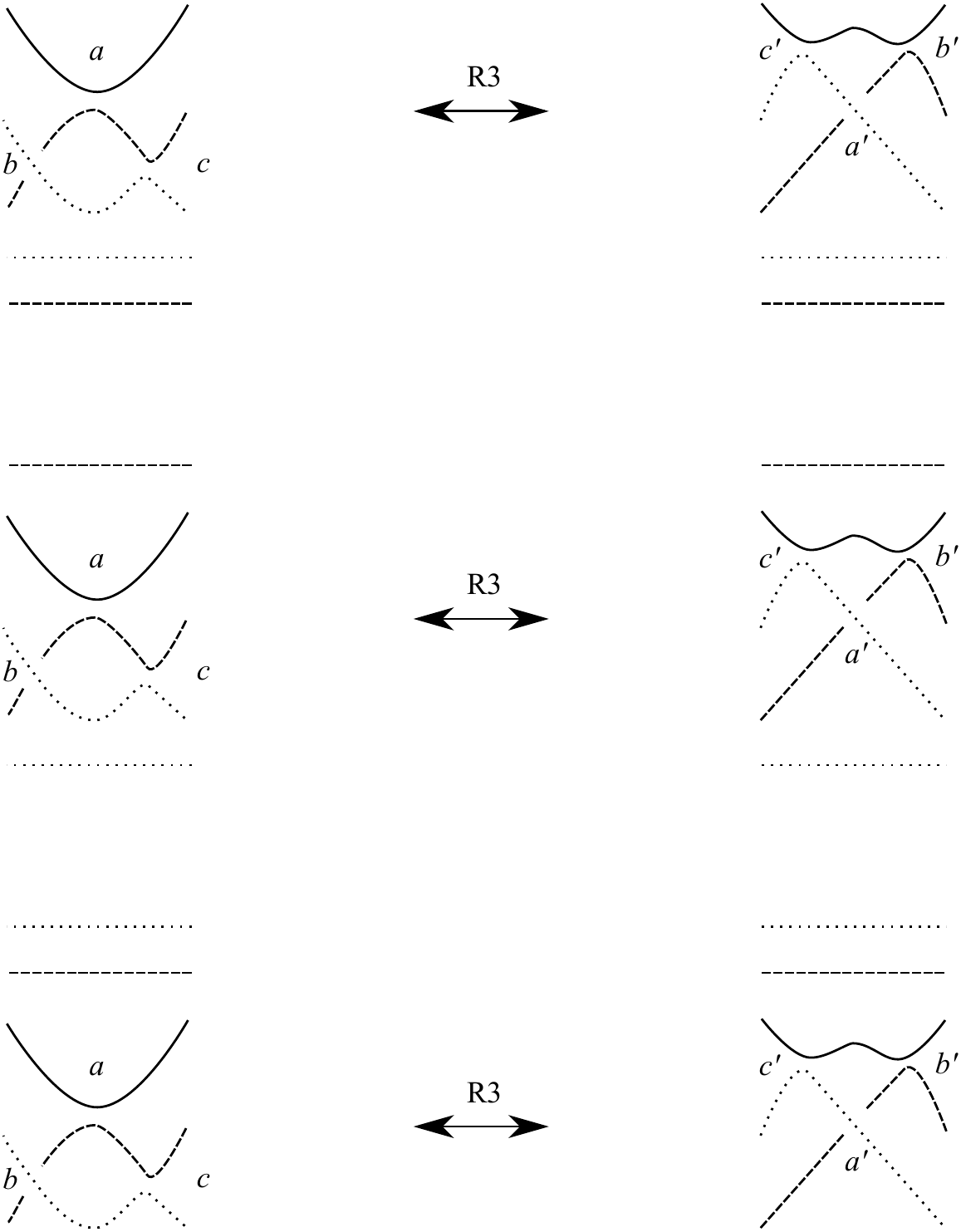}
\end{center}
\caption{Case B.}
\label{R3b}
\end{figure}

\begin{figure}
\begin{center}
\includegraphics{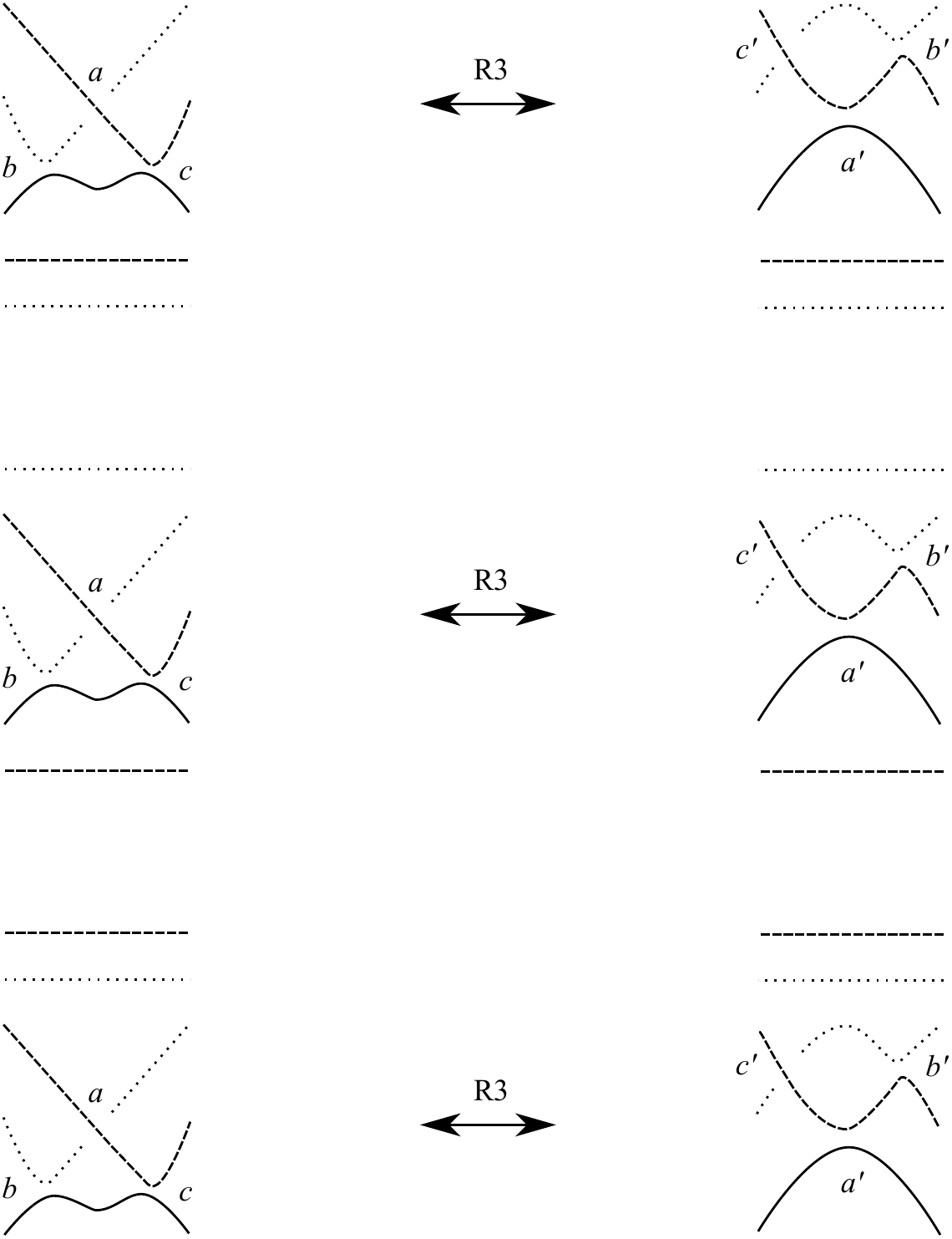}
\end{center}
\caption{Case C.}
\label{R3c}
\end{figure}

\begin{figure}
\begin{center}
\includegraphics{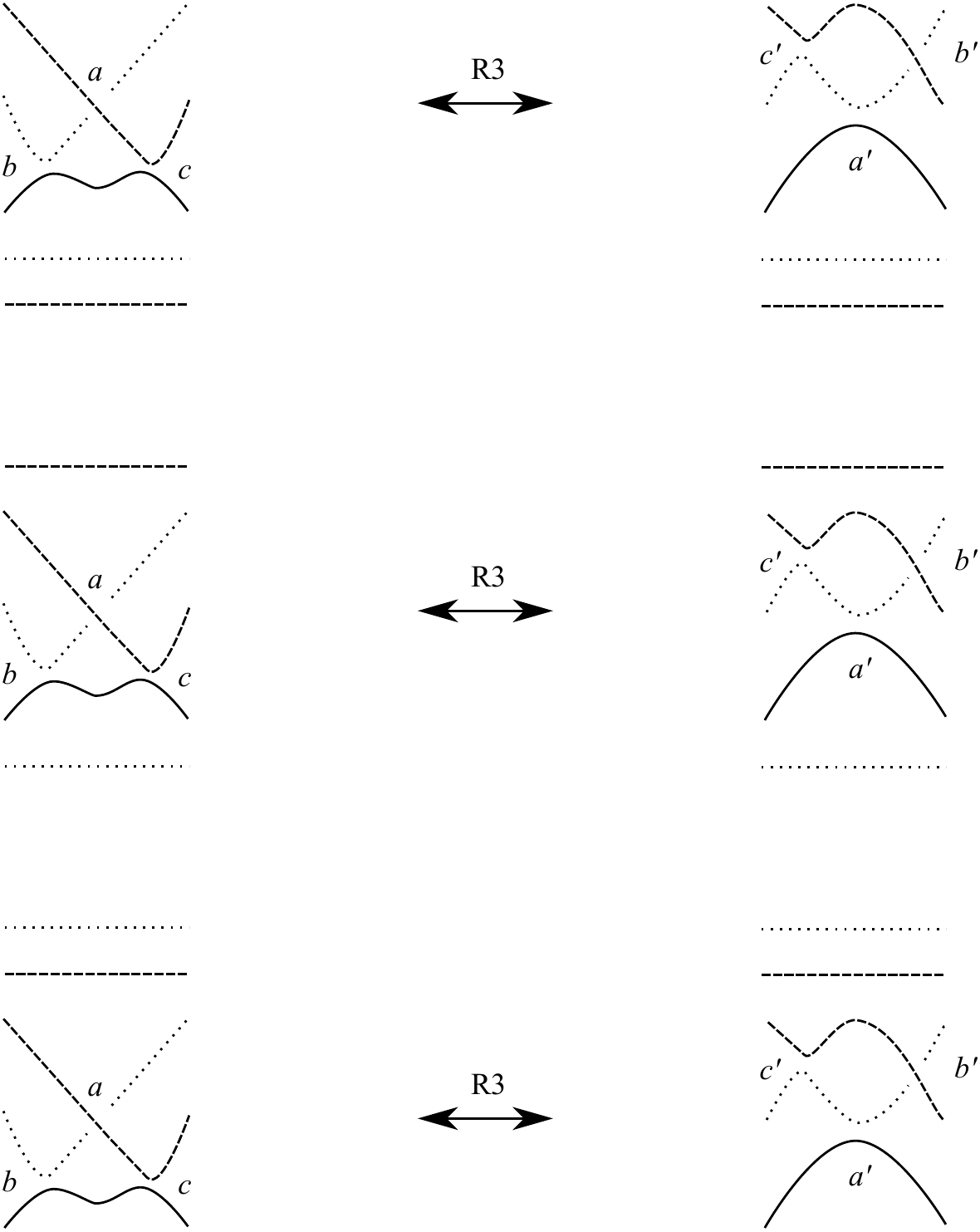}
\end{center}
\caption{Case D.}
\label{R3d}
\end{figure}

Next, we use Theorem \ref{Chekanov} to prove Lemma \ref{l5}.

\begin{proof}[Proof of Lemma \ref{l5}]
We prove by induction on the number of moves applied to $\emptyset$ to obtain $K$. For base case, the first move that could be applied to $\emptyset$ is 0-handle, which is obviously giving us a normal ruling (the standard Legendrian unknot has 1 normal ruling). Next, suppose that the statement is true for any sequence of moves with length $N$ or less. If the $(N+1)^{th}$ move is Legendrian isotopy, then by the proof of Theorem \ref{Chekanov} we obtain an associated normal ruling via the one-to-one correspondence. If the $(N+1)^{th}$ move is 0-handle or 1-handle, then the move can give a normal ruling by preserving the set of switches.
\end{proof}
\newpage
\begin{center}
\normalsize{3. C}\footnotesize{LASPS AND THEIR APPLICATIONS}
\end{center}
\noindent
\textbf{3.1 Blocks and Clasps.}
Notice that if the resolution of a normal ruling has 2 components, then it could be built (may need regular isotopy) from blocks by patching the ends of blocks at overlapping areas. For example, a front diagram in Figure \ref{be1} is built from 3 blocks illustrated at the bottom.

\begin{figure}[h]
	\begin{center}
		\vspace{0.8cm}
		\includegraphics[width=4.5in]{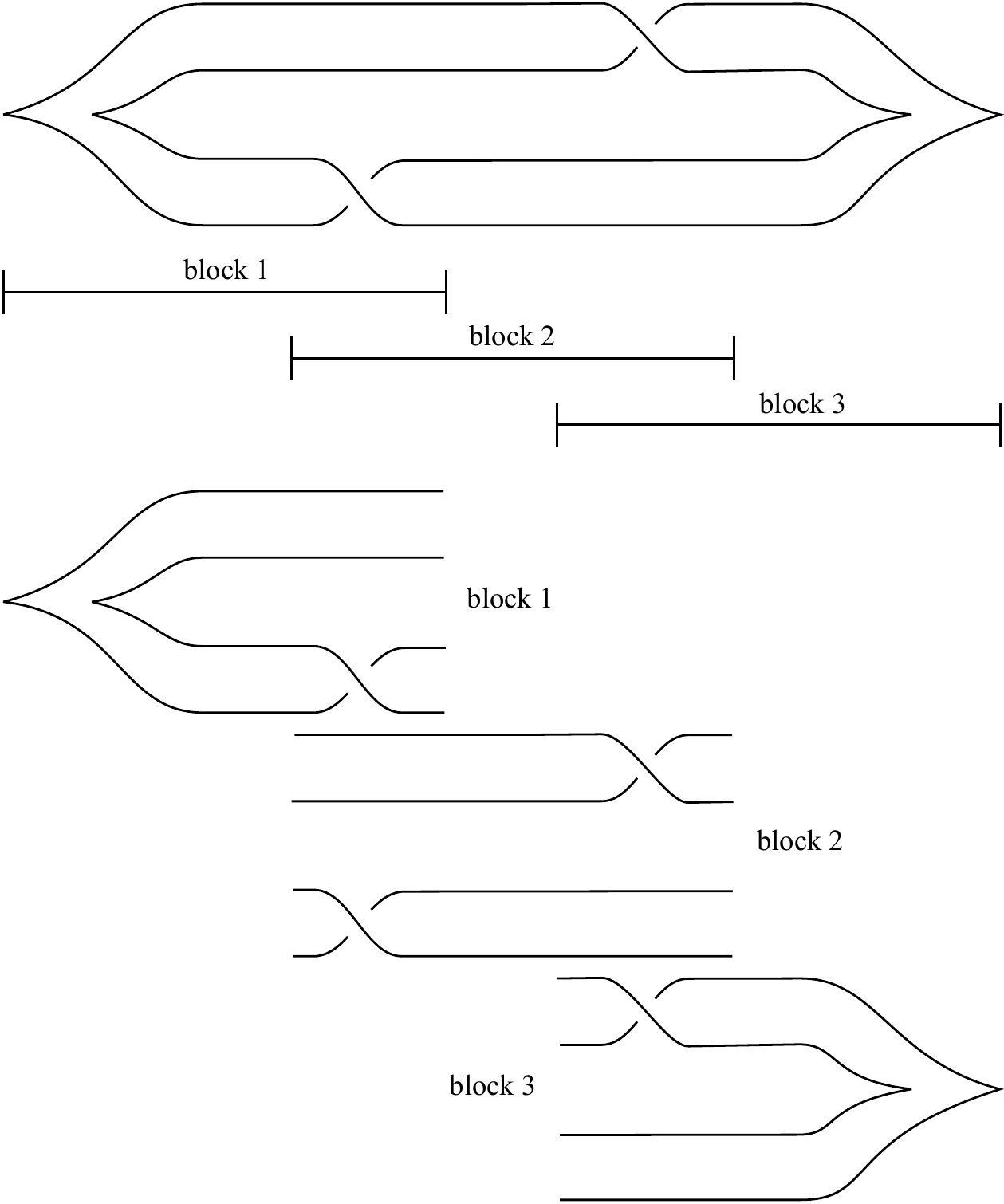}
	\end{center}
	\caption{A front diagram built from 3 blocks.}
	\label{be1}
\end{figure}

\begin{figure}
	\begin{center}
		\includegraphics[width=6in]{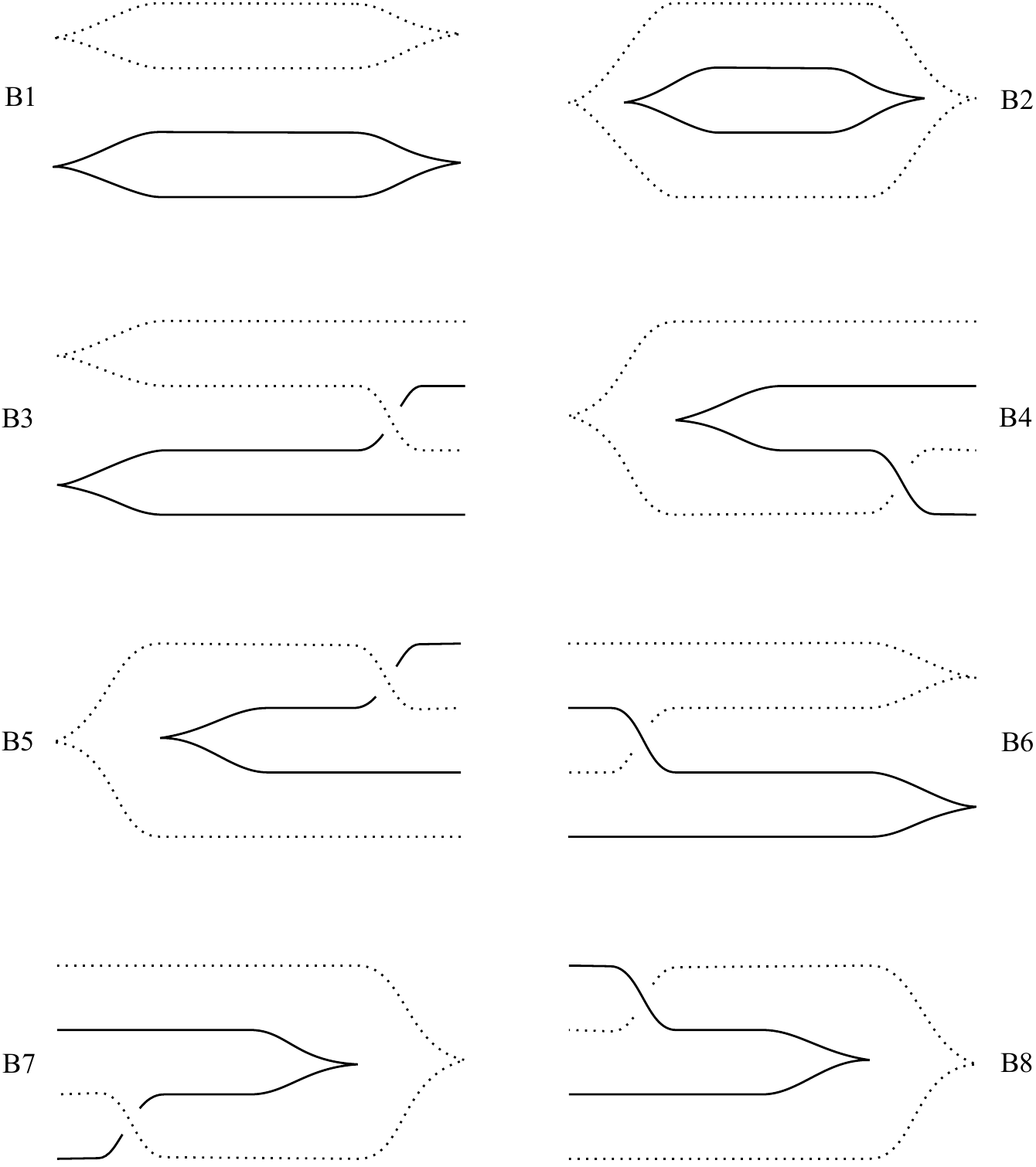}
	\end{center}
	\caption{Blocks of 2-component resolution.}
	\label{b1}
\end{figure}

\begin{figure}
	\begin{center}
		\includegraphics[width=6in]{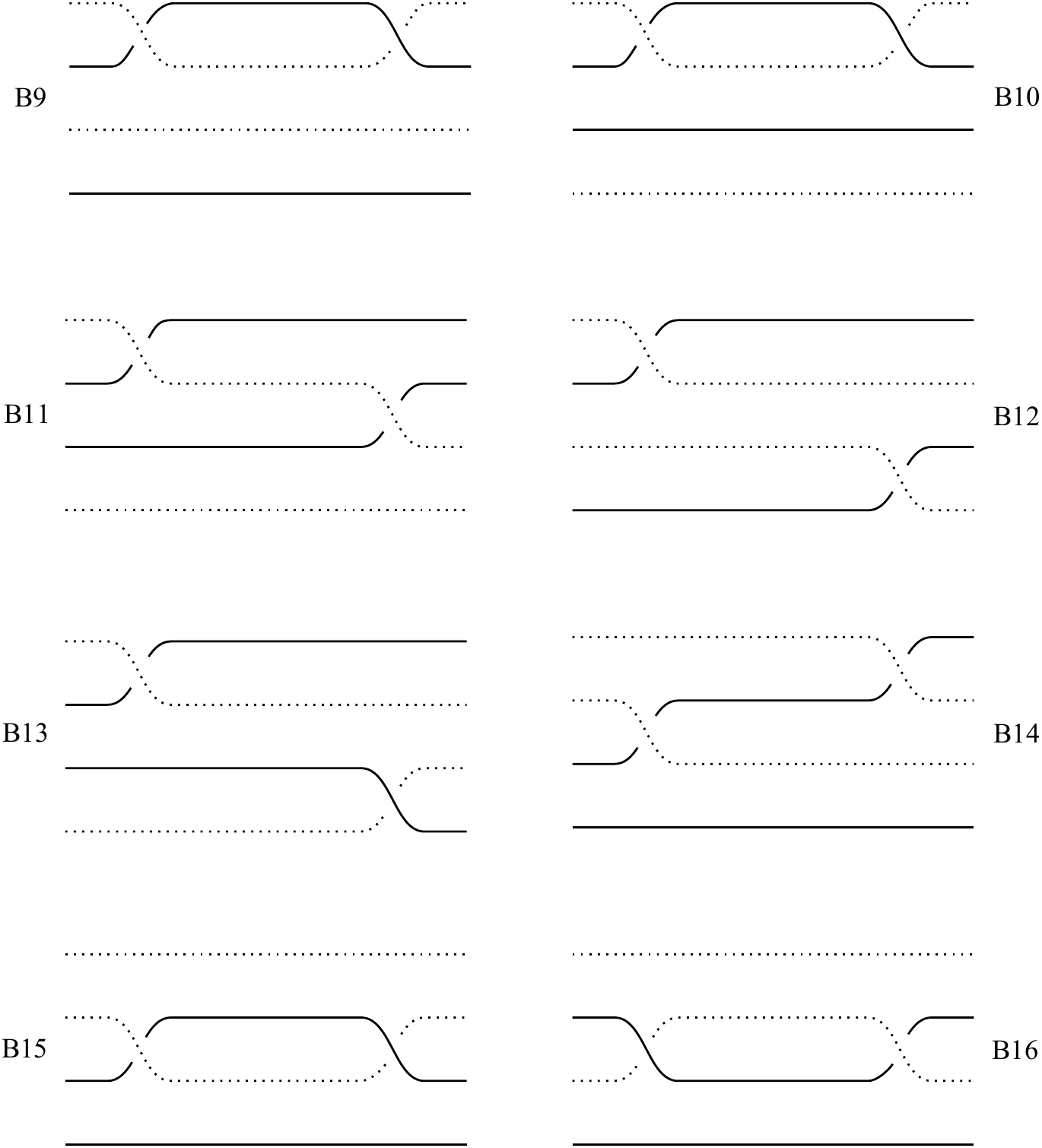}
	\end{center}
	\caption{Blocks of 2-component resolution (continue).}
	\label{b2}
\end{figure}

\begin{figure}
	\begin{center}
		\includegraphics[width=6in]{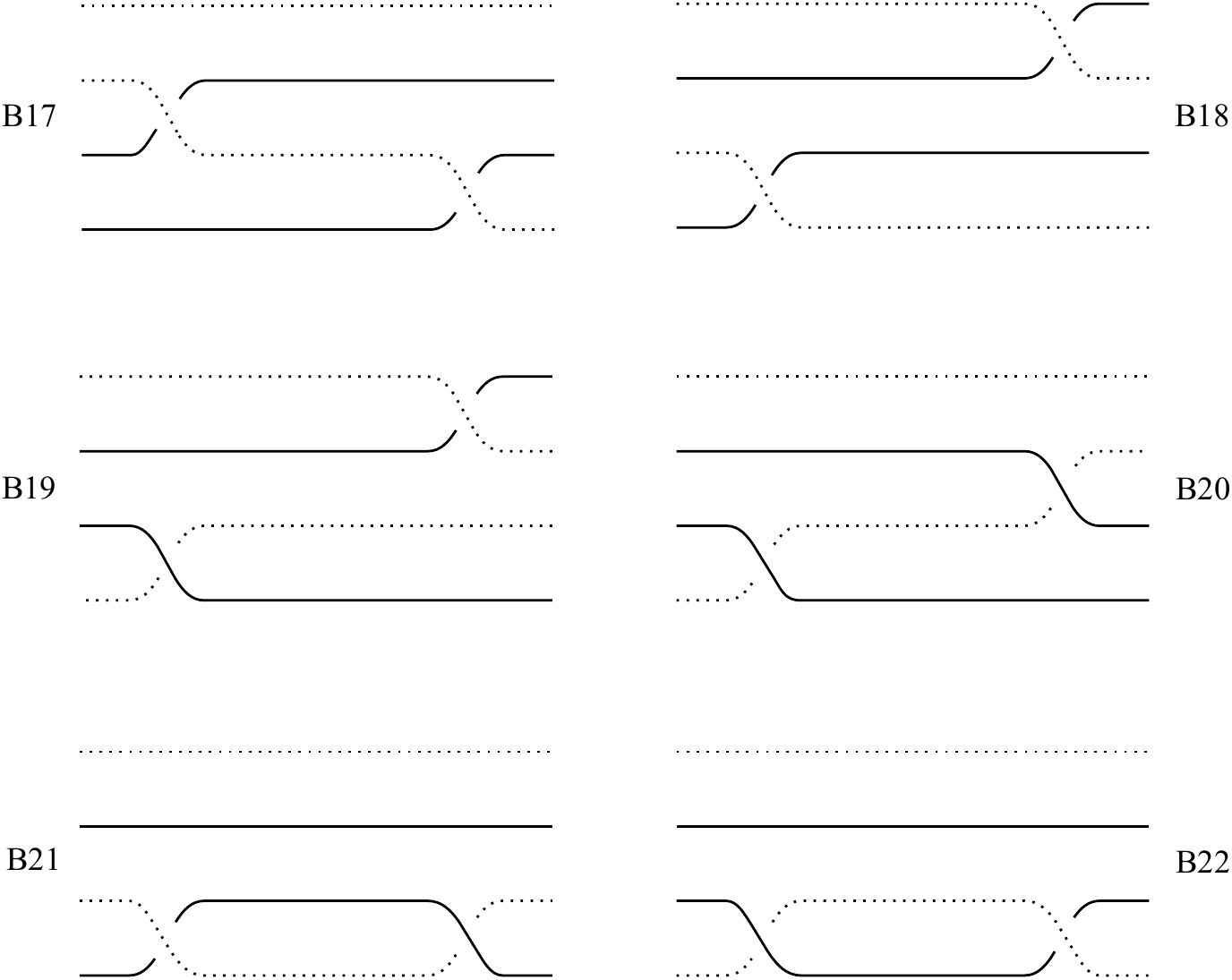}
	\end{center}
	\caption{Blocks of 2-component resolution (continue).}
	\label{b3}
\end{figure}

To be precise, we define a \textbf{block} to be a part of a 2-component resolution (up to regular isotopy) as in Figure \ref{b1} - \ref{b3}. Dotted lines and solid lines indicate how two components of eyes are positioned.

Now, some blocks are special for us. All the blocks in Figure \ref{clasp} are said to have a \textbf{clasp}. We also use a vertical dash line segment to represent a clasp. In a sense, the clasps can represent clasp intersections of disks bounded by eyes  (considered in $\mathbb{R}^3$) of the resolution of a normal ruling. In addition, a graph representing relationship between blocks as in Figure \ref{graph}.

\begin{figure}[h]
	\begin{center}
		\vspace{2.5cm}
		\includegraphics[width=6in]{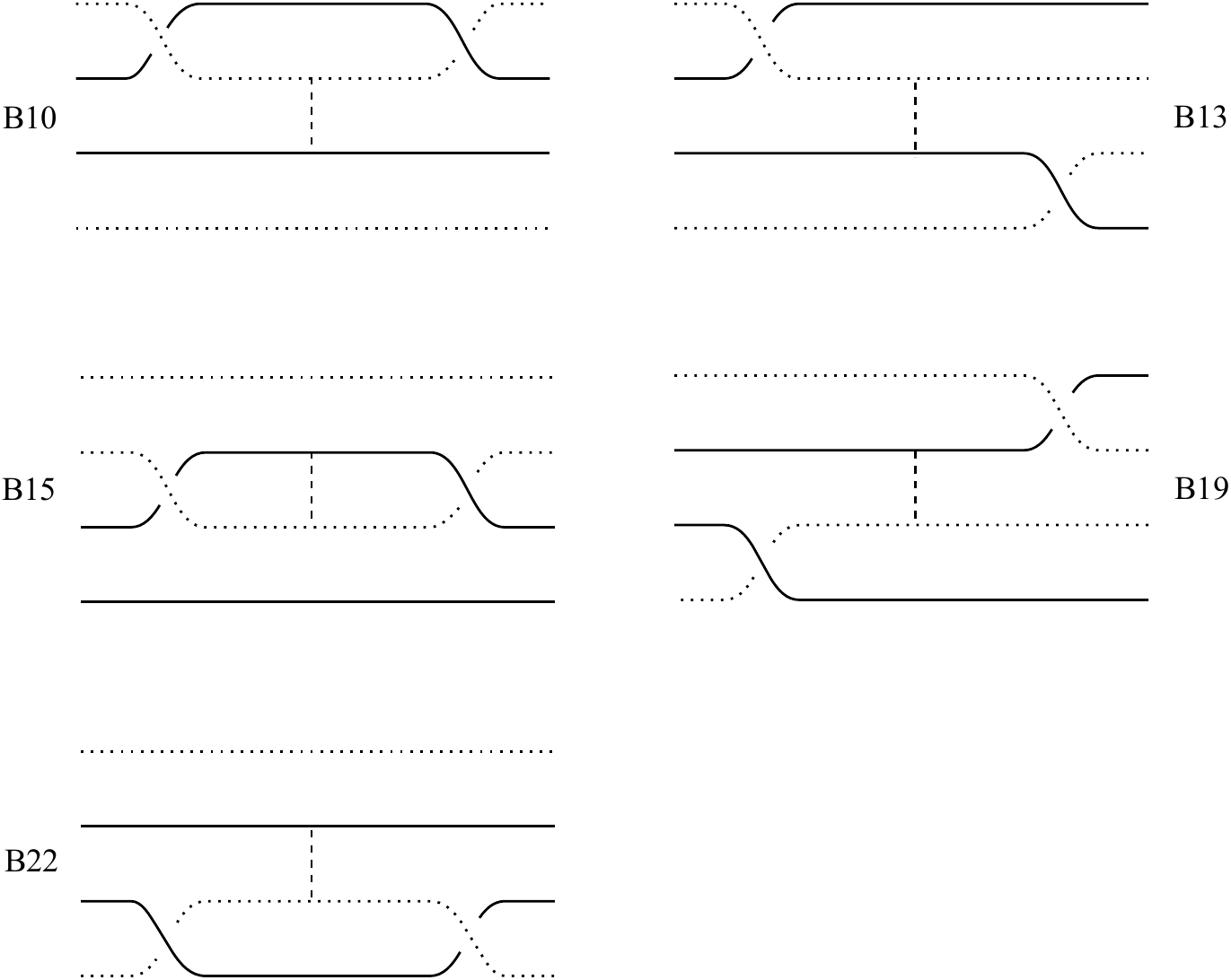}
	\end{center}
	\caption{Blocks with a clasp.}
	\label{clasp}
\end{figure}

\begin{figure}
	\begin{center}
		\includegraphics[height=6in]{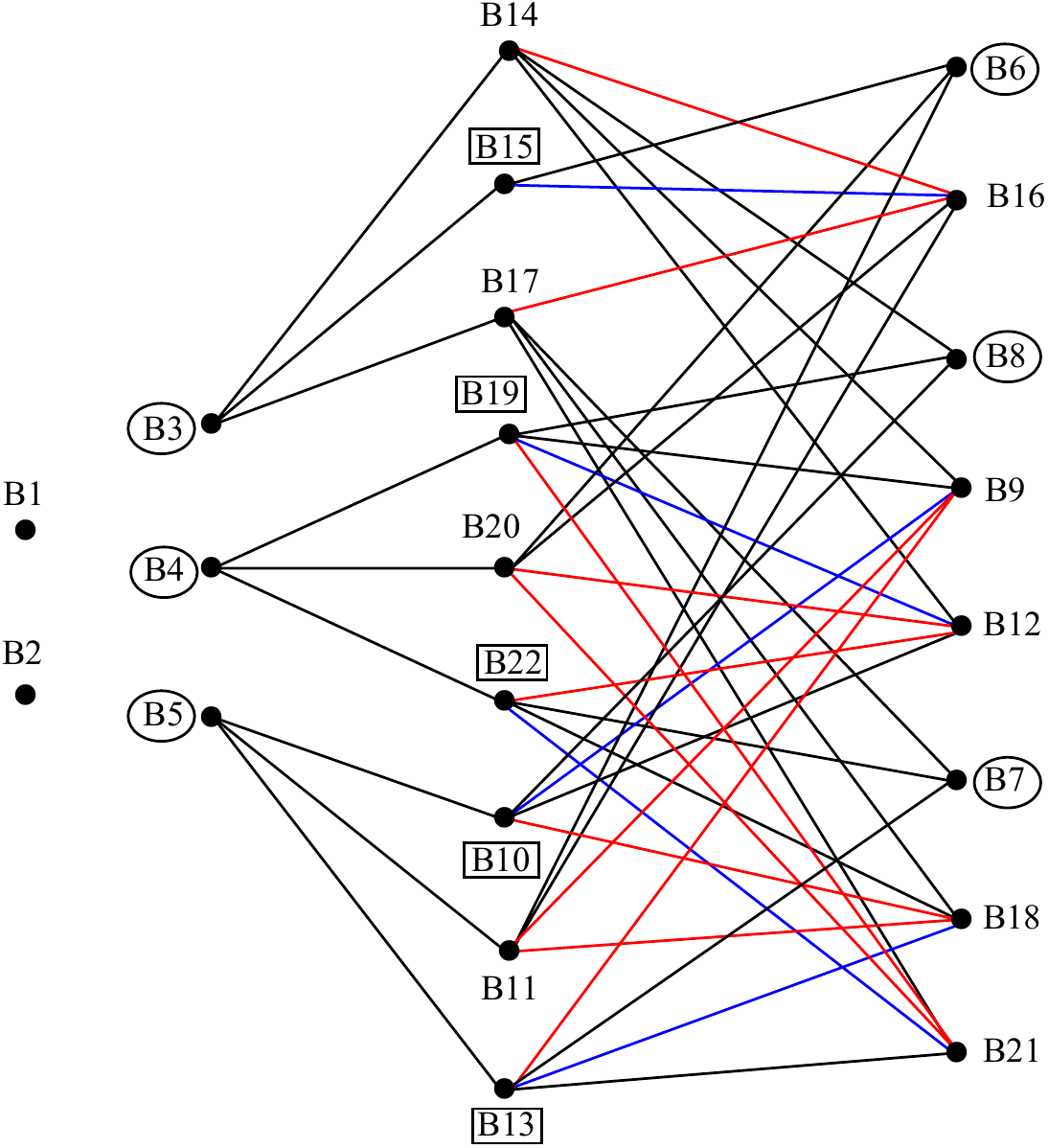}
	\end{center}
	\caption{A graph showing relationship between blocks. Boxed vertices are blocks with a clasp. Circled vertices are blocks which could only be the ends. Left vertex of a black edge is a block that can be placed on the left side of the block represented by right vertex of the same black edge. Left vertex of a red edge is a block that can be placed on the right side of the block represented by right vertex of the same red edge. Blue edges mean they are both black and red.}
	\label{graph}
\end{figure}

\newpage
For example, the front diagram in Figure \ref{be1} consists of 3 blocks, which are B4, B19 and B8. The middle block B19 has a clasp as shown in Figure \ref{bec}.

\begin{figure}[h]
	\begin{center}
		\vspace{1.5cm}
		\includegraphics[width=5.5in]{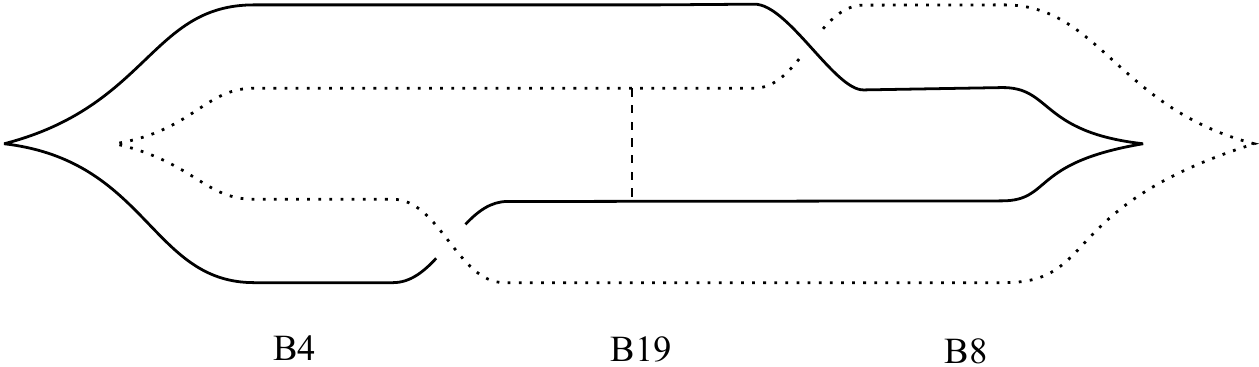}
	\end{center}
	\caption{The middle block has a clasp.}
	\label{bec}
\end{figure}

In a sense, the clasps can represent clasp intersections of disks bounded by eyes of the resolution of a normal ruling. Also, notice that, for those blocks with a clasp, two components must be non-nested\footnote{If two components are disjoint, we say they are nested.} in the middle part. In addition, a block with non-nested components could have no clasp if its two crossings all come from two horizontal strands intersecting with a single strand, e.g. B11, B14, B17 and B20.
\vspace{0.4cm}
\noindent
\\\textbf{3.2 Parity of Normal Rulings.}
Suppose we have the resolution of a normal ruling $\rho$. We may count the number of clasps of each pair of eyes. After that, we sum them up and refer to this amount as the \textbf{number of clasps of $\rho$}. We define the \textbf{parity of $\rho$} by saying that $\rho$ is \textbf{odd} (or \textbf{even}) if its number of clasps is odd (or even).

For example, in Figure \ref{c20a}, it is easy to check that the only normal ruling for the top-left front diagram is $\{a\}$. Its resolution is the top-right. The pair of eye 1 and eye 2 is shown in the bottom-left. This pair consists of 3 blocks, which are B3, B15 and B6. The middle block B15 has a clasp. So this pair contributes 1 clasp to the normal ruling. Next, the pair of eye 1 and eye 3 is shown in the bottom-middle. This pair is the block B1 so it gives no clasp. Finally, the pair of eye 2 and eye 3 is shown in the bottom-right. This pair contains B3, B15 and B6. So it gives a clasp. Hence the normal ruling has 2 clasps in total.

\begin{figure}
\begin{center}
\includegraphics{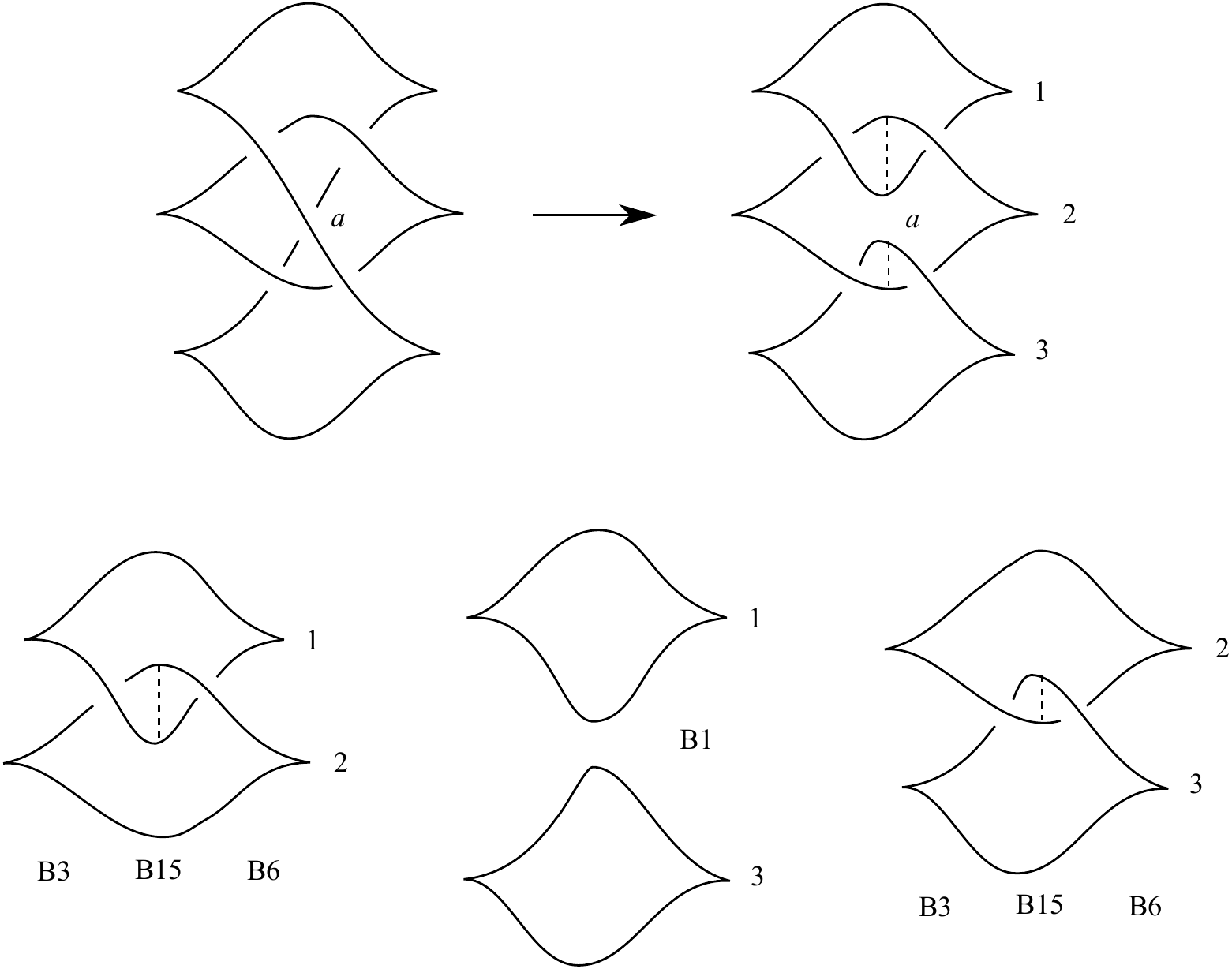}
\end{center}
\caption{A normal ruling with 2 clasps.}
\label{c20a}
\end{figure}

Next, in Figure \ref{c20b}, the only normal ruling is $\{a\}$. Its resolution is the top-right. The pair of eye 1 and eye 2 is shown in the bottom-left. This pair consists of 5 blocks, which are B3, B17, B21, B20 and B6. So this pair contributes 0 clasp to the normal ruling. Next, the pair of eye 1 and eye 3 is shown in the bottom-middle. This pair is the block B1 so it gives no clasp. Finally, the pair of eye 2 and eye 3 is shown in the bottom-right. This pair is also B1. So it gives 0 clasp. Hence, the normal ruling has 0 clasp in this case.

\begin{figure}
\begin{center}
\includegraphics{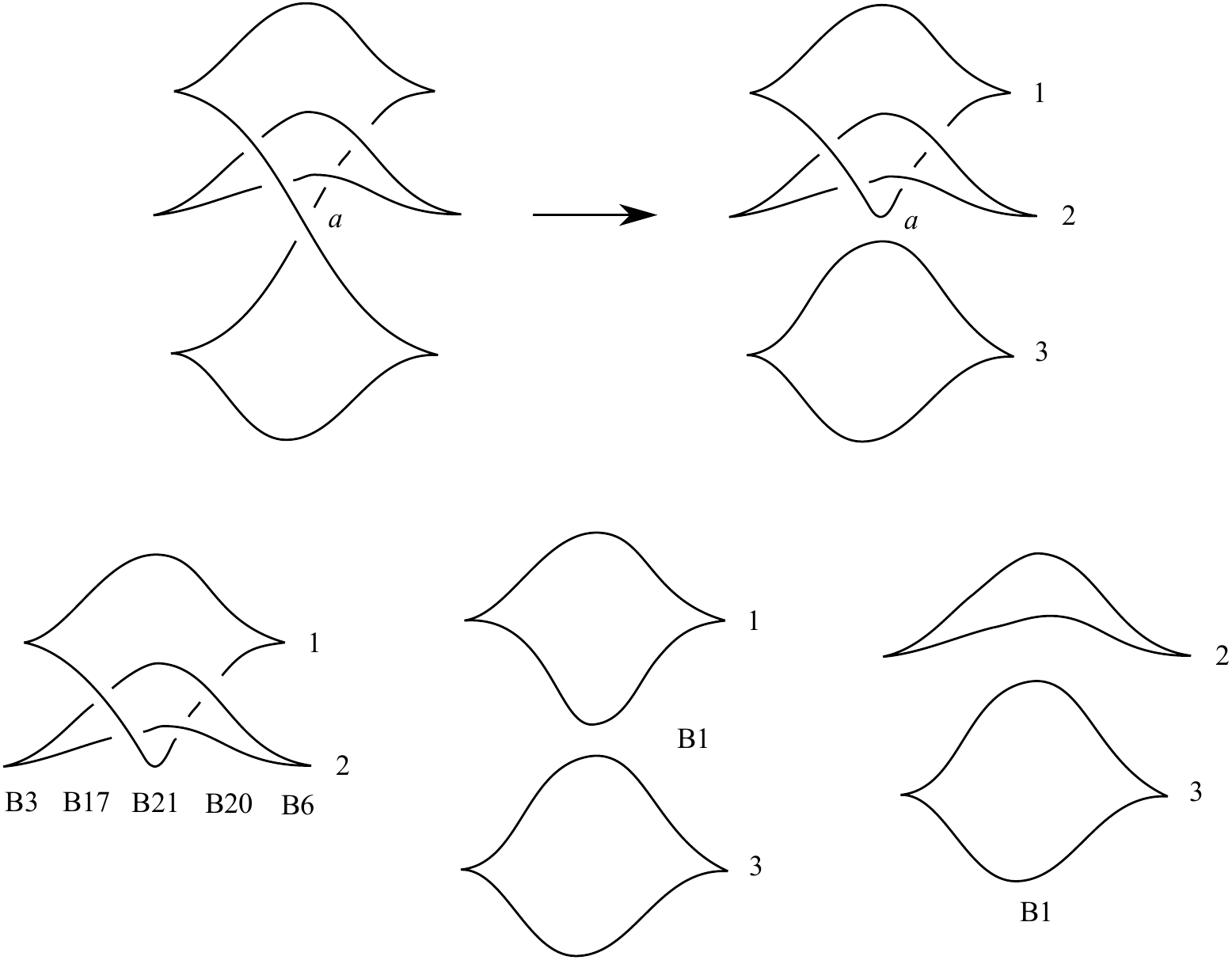}
\end{center}
\caption{A normal ruling with 0 clasp.}
\label{c20b}
\end{figure}

Notice that the top-left front diagram of Figure \ref{c20b} is the result of applying one R3 to the top-left front diagram of Figure \ref{c20a}. Their normal rulings have different number of clasps, but the parity is the same, i.e. both normal rulings are even. Next, we prove that the number of clasps of a resolution is invariant under regular isotopy, R2 and R3 of the resolution.

\begin{Lem}
\label{lres}
Suppose we have the resolution of a normal ruling. Then the number of clasps of the resolution is invariant under regular isotopy, R2 and R3 of the resolution.
\end{Lem}

\begin{proof}
By induction, it is enough to consider only when we apply a single move of regular isotopy, R2 or R3 to the resolution. The only kind of regular isotopy that could change blocks is the one that interchanges $x$-coordinates of two crossings of a block as in Figure \ref{reg}. It is easy to see that the number of clasps is invariant under this move as in Figure \ref{regc}. Next, it is not hard to see that the number of clasps is also the same under an application of a Legendrian Reidemeister move R2 or R3 since none of them will create/destroy a block with clasp.
\end{proof}

\begin{figure}
\begin{center}
\includegraphics[width=6in]{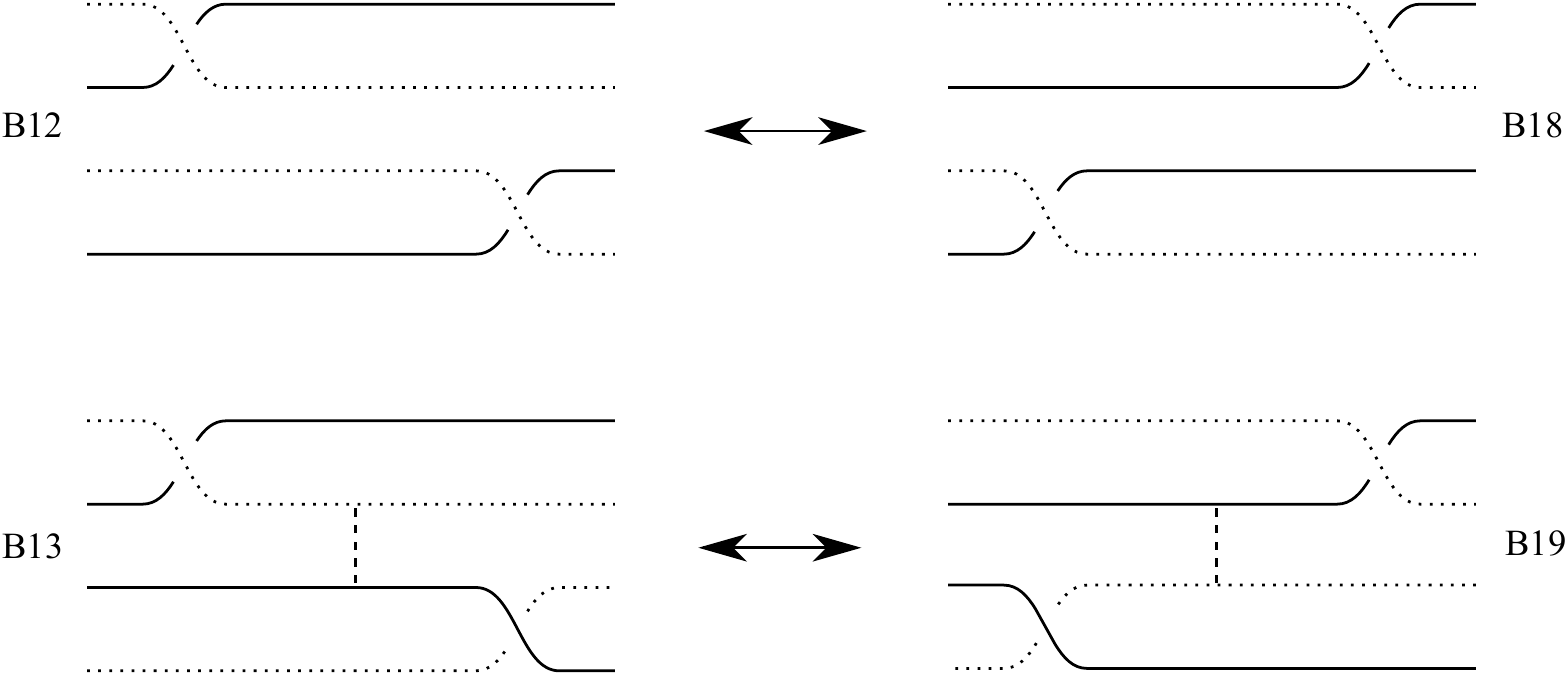}
\end{center}
\caption{The number of clasps is invariant under regular isotopy of resolutions.}
\label{regc}
\end{figure}

\noindent
\textbf{3.3 Enhanced Cuts.}
Let $\rho$ be a normal ruling. An \textbf{enhanced cut} is a modification on an eye of the resolution of $\rho$ that could cut through obstructing horizontal strands, as in Figure \ref{ec1}. Note that under regular isotopy, we may assume that there is no other crossings and cusps in the vertical strips we performing enhanced cut.

\begin{figure}[h]
	\begin{center}
		\vspace{0.6cm}
		\includegraphics[width=5in]{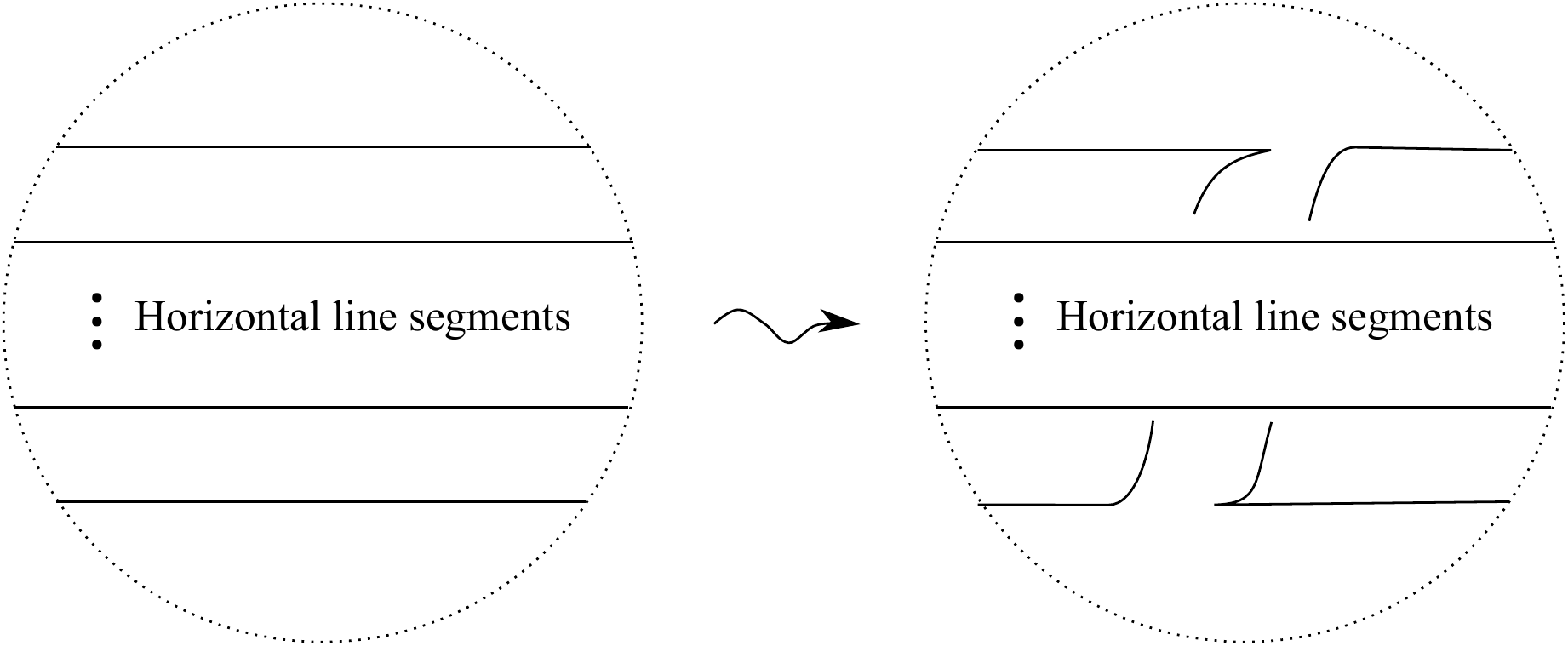}
	\end{center}
	\caption{Enhanced cut.}
	\label{ec1}
\end{figure}

Enhanced cuts have an interesting effect on the number of clasps of a normal ruling.

\begin{Lem}
\label{lc}
Let $\rho$ be a normal ruling. Applying an enhanced cut to an eye of the resolution of $\rho$ does not change the parity of $\rho$.
\end{Lem}

\begin{proof}
The statement is clearly true when the resolution of the normal ruling has 1 component. So we may assume that the resolution has 2 or more components. Since we count the number of clasps between each pair of eyes, we only need to perform an enhanced cut to the middle part of each block with a rule that the cut must happen between the upper and the lower strands of an eye. Then, we observe changes in the number of clasps. First, applying an enhanced cut (with respect to eyes) to the middle part of blocks with one end having cusps (as in Figure \ref{b1}) will not change the number of clasps.

Next, we prove case by case as in Figure \ref{c1} - \ref{c14}.

Case 1 (see Figure \ref{c1}): The original block has no clasp. After modifications, there are 6 blocks for the top-right diagram. None of them has a clasp. Thus, the top-right diagram has the same number of clasps. The bottom diagram also has the same number of clasps.

Case 2 (see Figure \ref{c2}): The original block has 1 clasp. After modifications, there are 4 blocks for the top-right diagram. One of them has a clasp. Thus, the top-right diagram has the same number of clasps. Also, there are 4 blocks for the bottom. One of them has a clasp. Hence, the number of clasps is preserved.

Case 3 (see Figure \ref{c3}): The original block has no clasp. After modifications, there are 4 blocks for the top-right diagram. None of them has a clasp. Thus, the top-right diagram has the same number of clasps. Also, there are 4 blocks for the bottom. Two of them have a clasp. Hence, the number of clasps is increased by 2.

Case 4 (see Figure \ref{c4}): The original block has no clasp. After modifications, there are 6 blocks for the top-right diagram. None of them has a clasp. Thus, the top-right diagram has the same number of clasps. The bottom diagram also has the same number of clasps.

Case 5 (see Figure \ref{c5}): The original block has 1 clasp. After modifications, there are 4 blocks for the top-right diagram. One of them has a clasp. Thus, the top-right diagram has the same number of clasps. Also, there are 4 blocks for the bottom. One of them has a clasp. Hence, the number of clasps is preserved.

Case 6 (see Figure \ref{c6}): The original block has no clasp. After modifications, there are 4 blocks for the top-right diagram. Two of them have a clasp. Thus, the top-right diagram has the number of clasps increased by 2. Also, there are 4 blocks for the bottom. None of them has a clasp. Hence, the number of clasps is preserved.

Case 7 (see Figure \ref{c7}): The original block has 1 clasp. After modifications, there are 4 blocks for the top-right diagram. One of them has a clasp. Thus, the top-right diagram has the same number of clasps. Also, there are 4 blocks for the bottom. One of them has a clasp. Hence, the number of clasps is preserved.

Case 8 (see Figure \ref{c8}): The original block has no clasp. After modifications, it is clear that both the top-right and the bottom diagrams also have the same number of clasps.

Case 9 (see Figure \ref{c9}): The original block has no clasp. After modifications, there are 4 blocks for the top-right diagram. Two of them have a clasp. Thus, the top-right diagram has the number of clasps increased by 2. Also, there are 4 blocks for the bottom. None of them has a clasp. Hence, the number of clasps is preserved.

Case 10 (see Figure \ref{c10}): The original block has no clasp. After modifications, there are 6 blocks for the top-right diagram. None of them has a clasp. Thus, the top-right diagram has the same number of clasps. The bottom diagram also has the same number of clasps.

Case 11 (see Figure \ref{c11}): The original block has 1 clasp. After modifications, there are 4 blocks for the top-right diagram. One of them has a clasp. Thus, the top-right diagram has the same number of clasps. Also, there are 4 blocks for the bottom. One of them has a clasp. Hence, the number of clasps is preserved.

Case 12 (see Figure \ref{c12}): The original block has no clasp. After modifications, there are 4 blocks for the top-right diagram. None of them has a clasp. Thus, the top-right diagram has the same number of clasps. Also, there are 4 blocks for the bottom. Two of them have a clasp. Hence, the number of clasps is increased by 2.

Case 13 (see Figure \ref{c13}): The original block has no clasp. After modifications, there are 6 blocks for the top-right diagram. None of them has a clasp. Thus, the top-right diagram has the same number of clasps. The bottom diagram also has the same number of clasps.

Case 14 (see Figure \ref{c14}): The original block has 1 clasp. After modifications, there are 4 blocks for the top-right diagram. One of them has a clasp. Thus, the top-right diagram has the same number of clasps. Also, there are 4 blocks for the bottom. One of them has a clasp. Hence, the number of clasps is preserved. 
\end{proof}

\begin{figure}
	\begin{center}
		\includegraphics[width=6in]{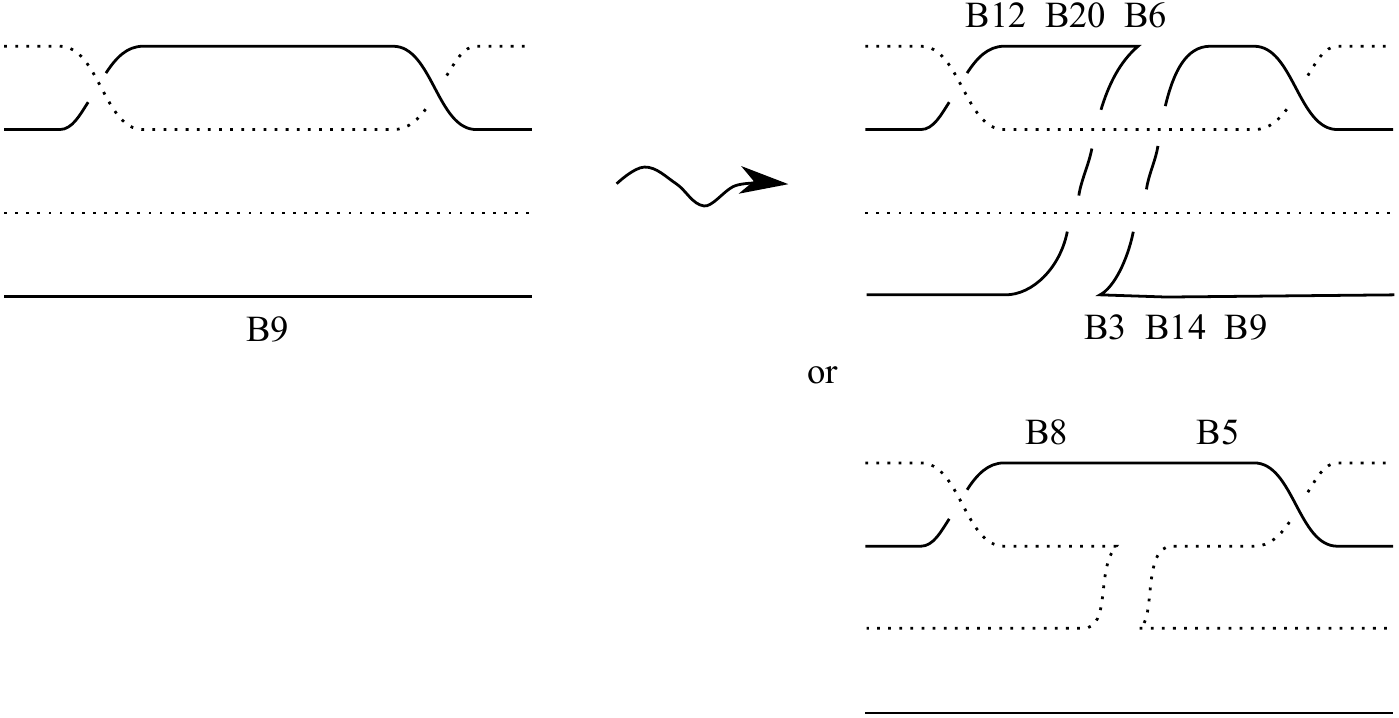}
	\end{center}
	\caption{Case 1.}
	\label{c1}
\end{figure}

\begin{figure}
	\begin{center}
		\includegraphics[width=6in]{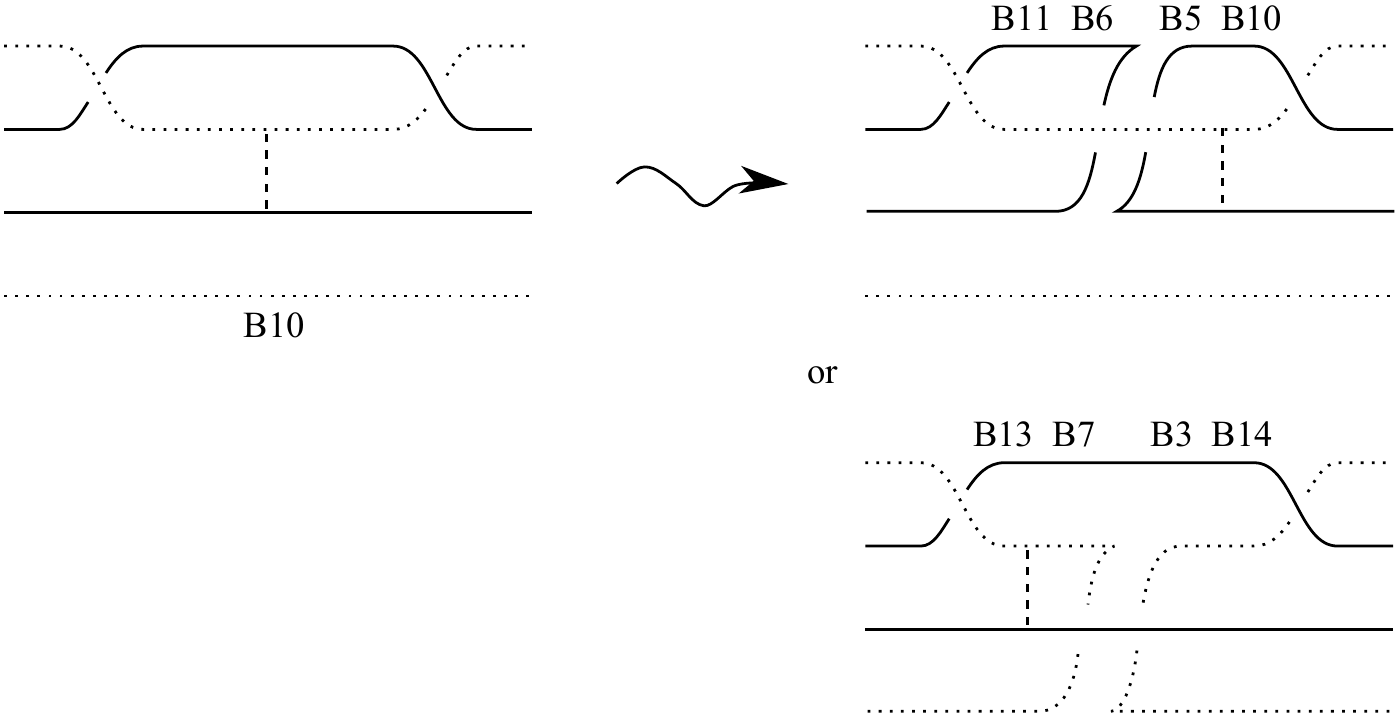}
	\end{center}
	\caption{Case 2.}
	\label{c2}
\end{figure}

\begin{figure}
	\begin{center}
		\includegraphics[width=6in]{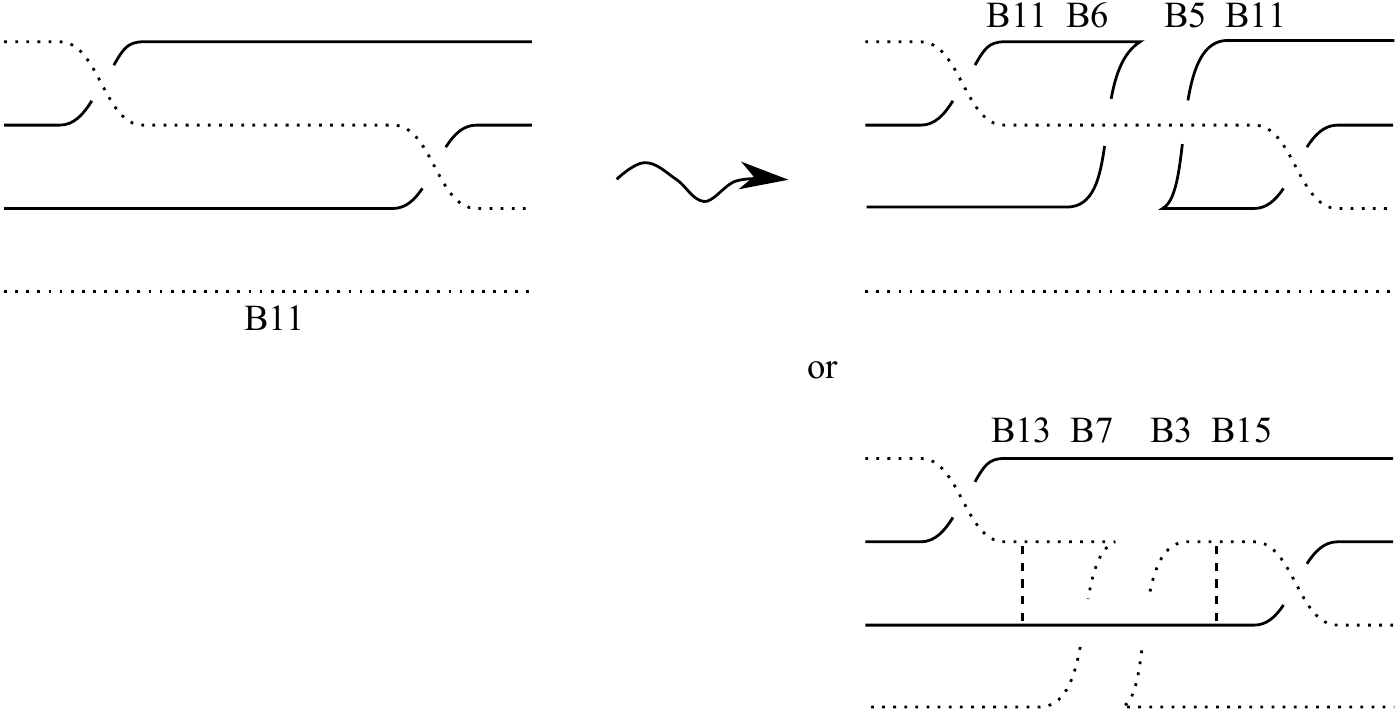}
	\end{center}
	\caption{Case 3.}
	\label{c3}
\end{figure}

\begin{figure}
	\begin{center}
		\includegraphics[width=6in]{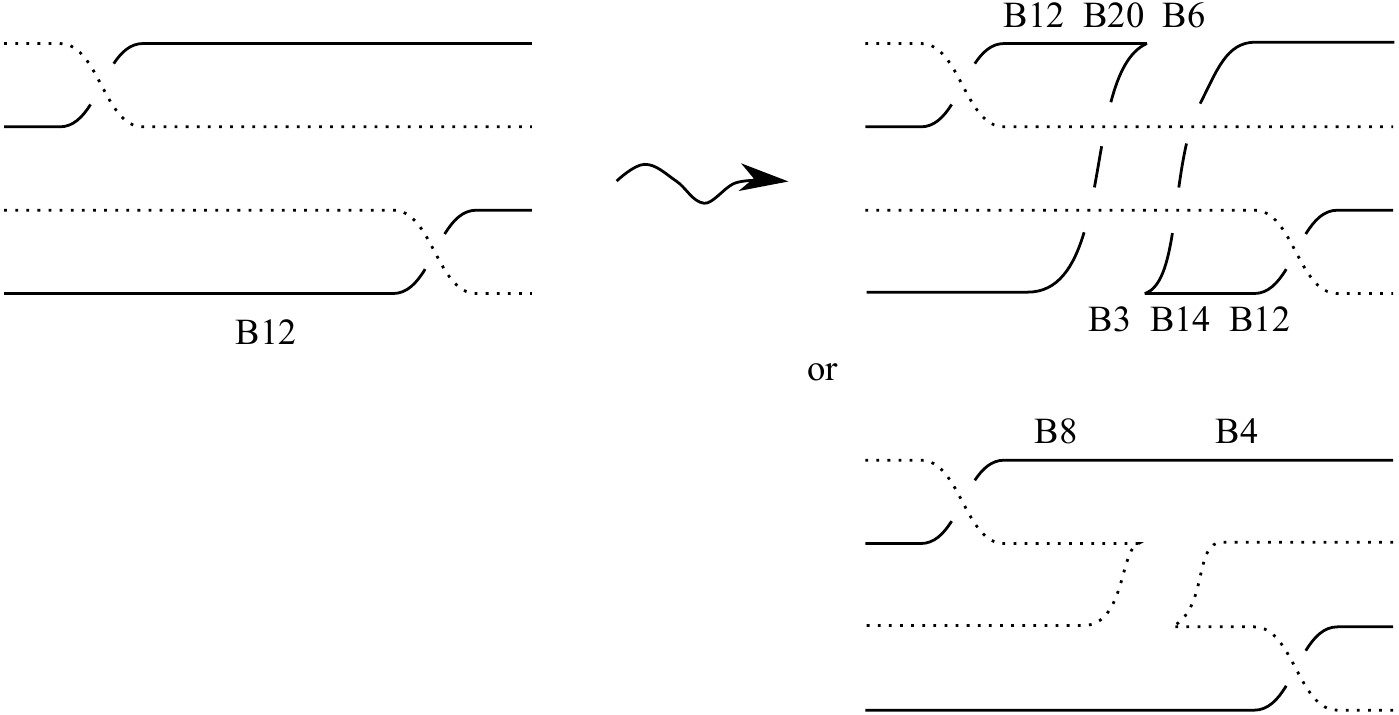}
	\end{center}
	\caption{Case 4.}
	\label{c4}
\end{figure}

\begin{figure}
	\begin{center}
		\includegraphics[width=6in]{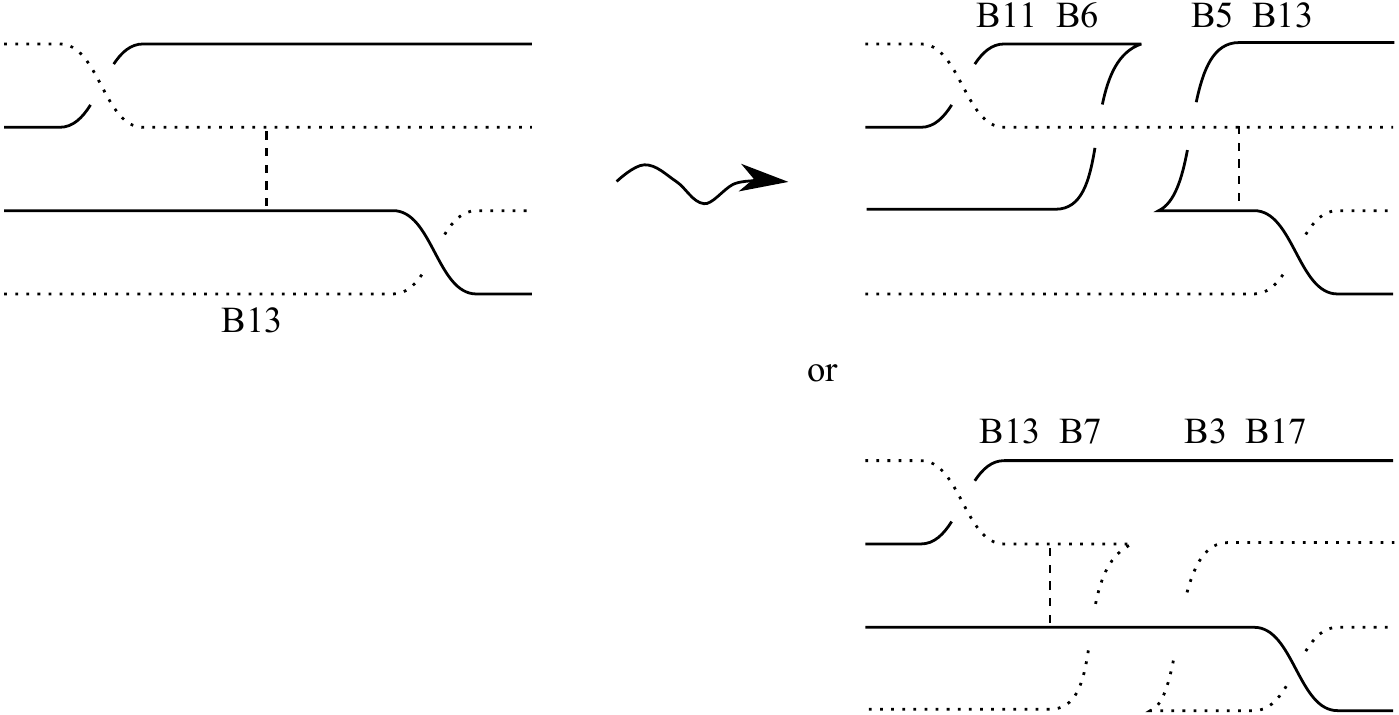}
	\end{center}
	\caption{Case 5.}
	\label{c5}
\end{figure}

\begin{figure}
	\begin{center}
		\includegraphics[width=6in]{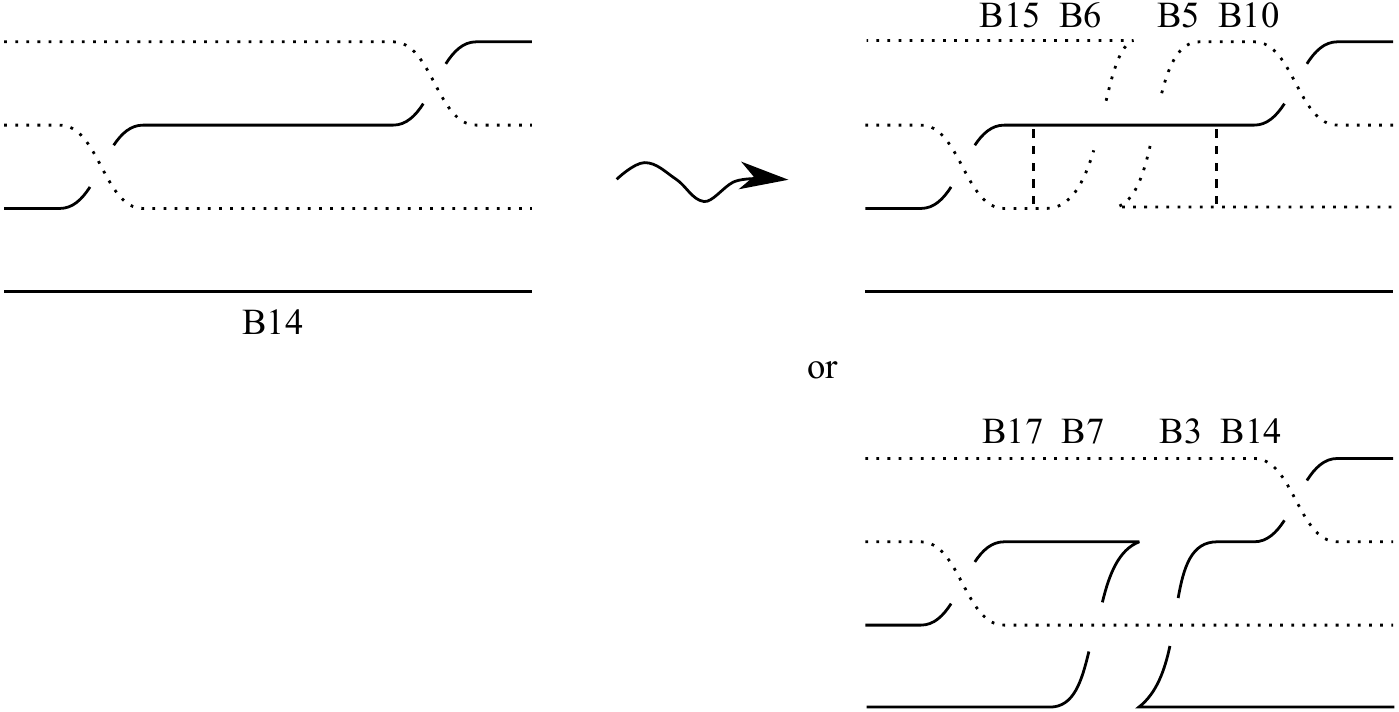}
	\end{center}
	\caption{Case 6.}
	\label{c6}
\end{figure}

\begin{figure}
	\begin{center}
		\includegraphics[width=6in]{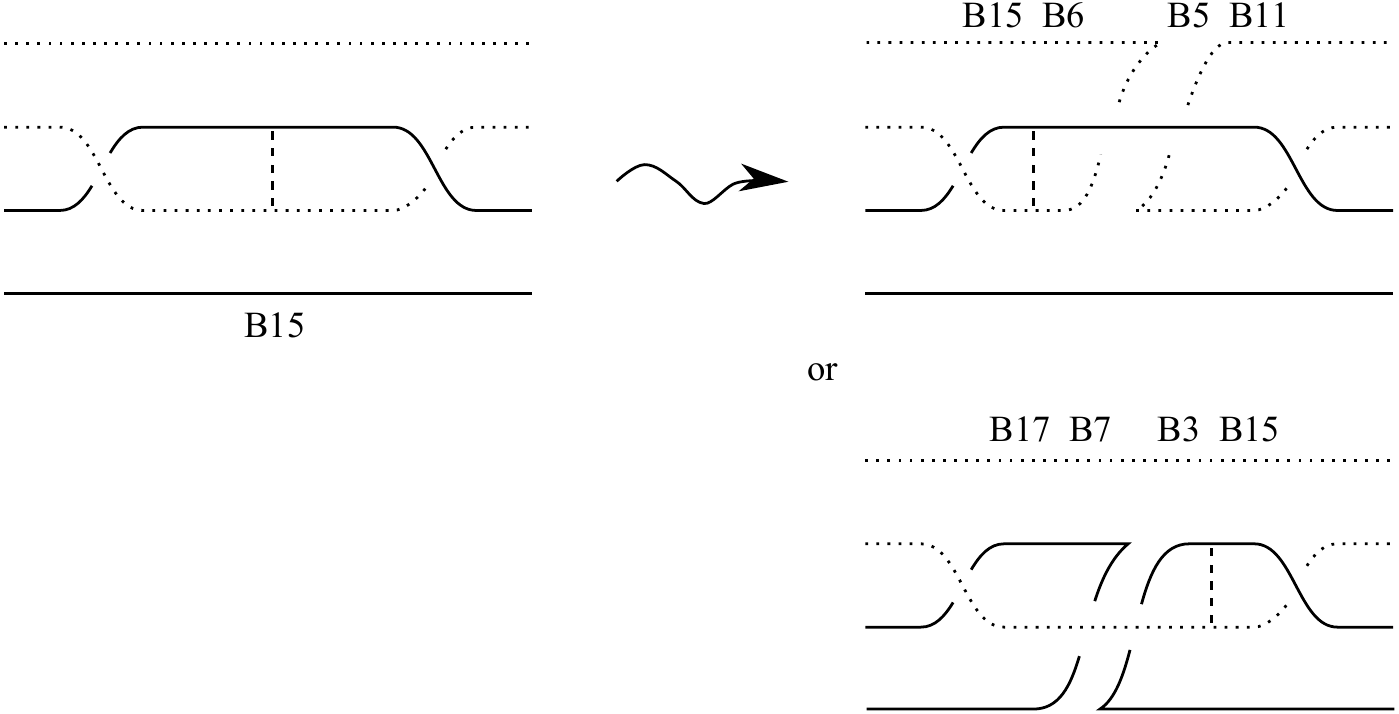}
	\end{center}
	\caption{Case 7.}
	\label{c7}
\end{figure}

\begin{figure}
	\begin{center}
		\includegraphics[width=6in]{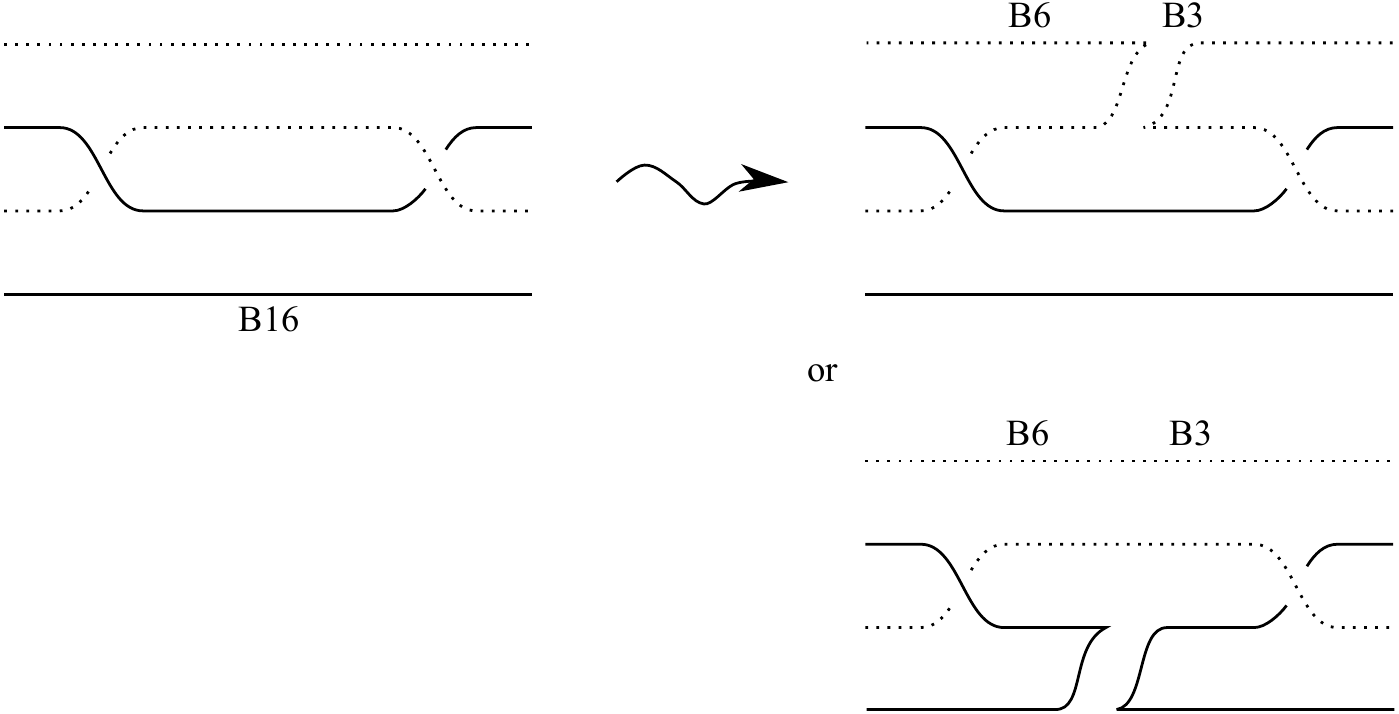}
	\end{center}
	\caption{Case 8.}
	\label{c8}
\end{figure}

\begin{figure}
	\begin{center}
		\includegraphics[width=6in]{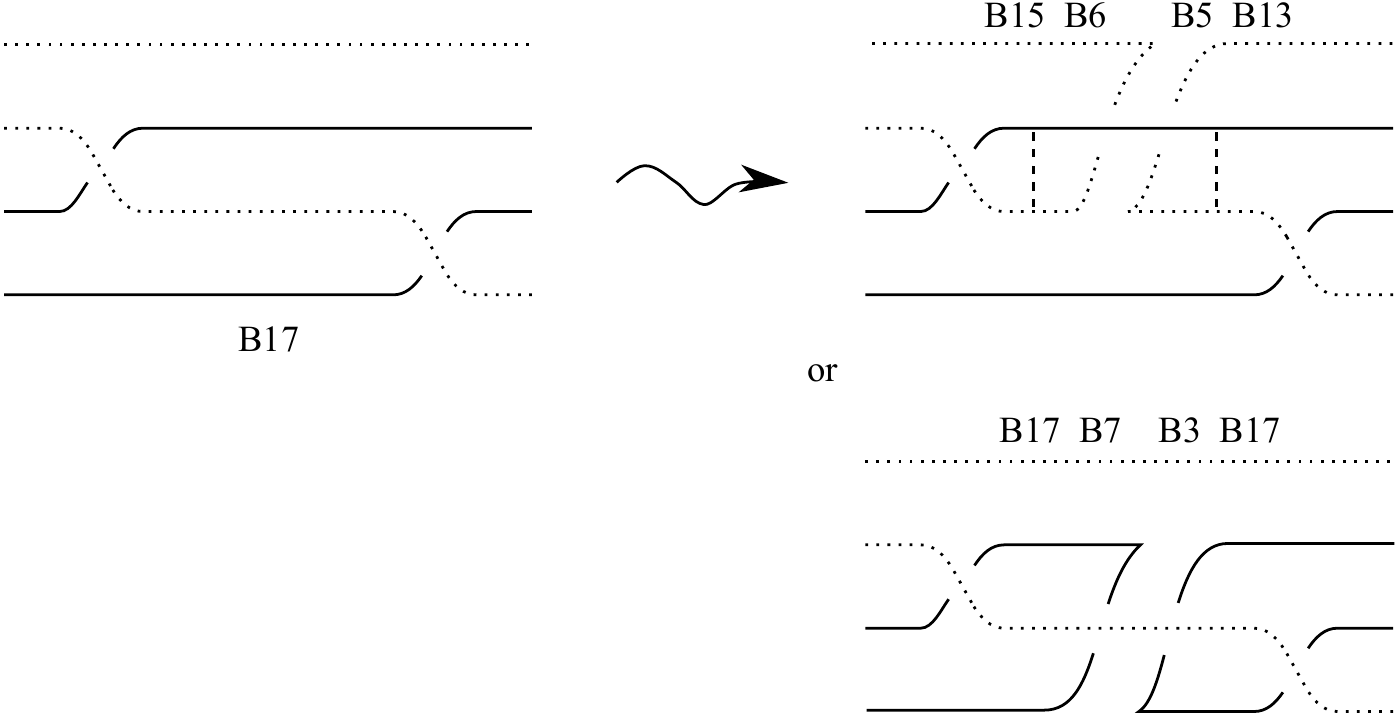}
	\end{center}
	\caption{Case 9.}
	\label{c9}
\end{figure}

\begin{figure}
	\begin{center}
		\includegraphics[width=6in]{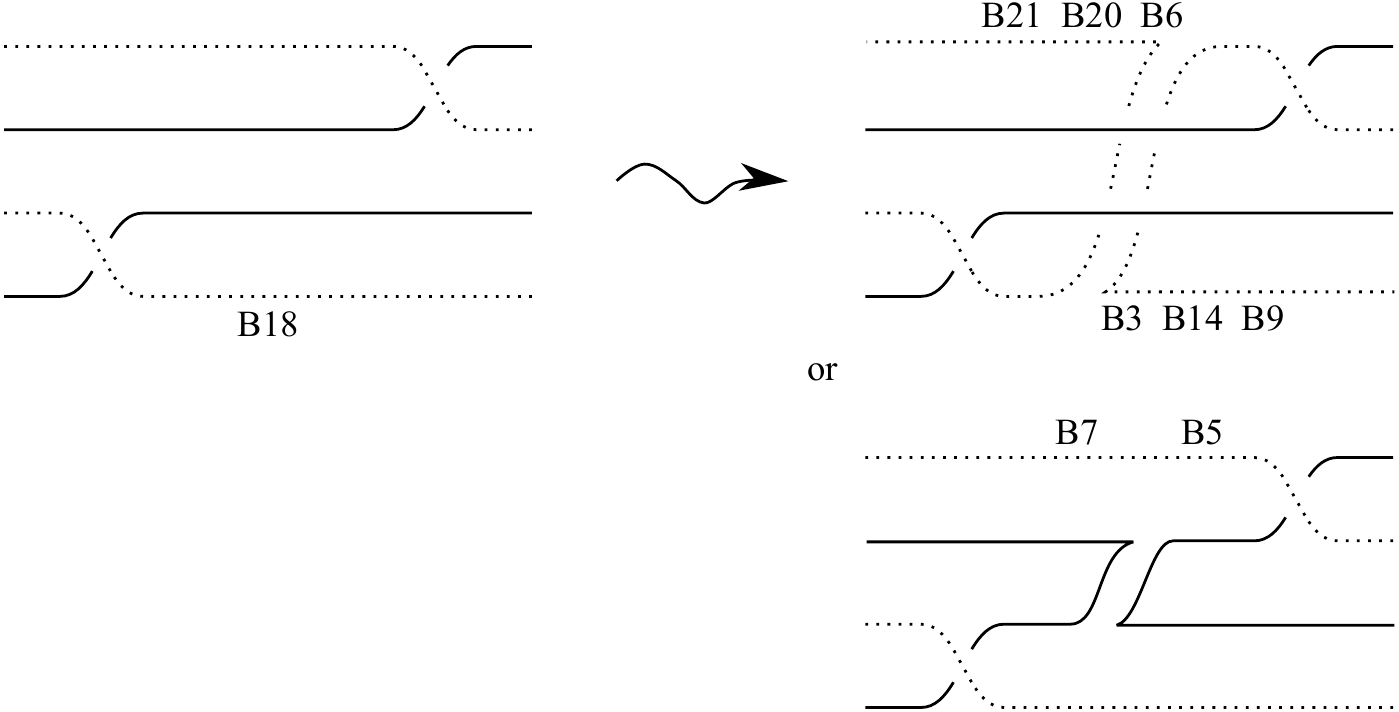}
	\end{center}
	\caption{Case 10.}
	\label{c10}
\end{figure}

\begin{figure}
	\begin{center}
		\includegraphics[width=6in]{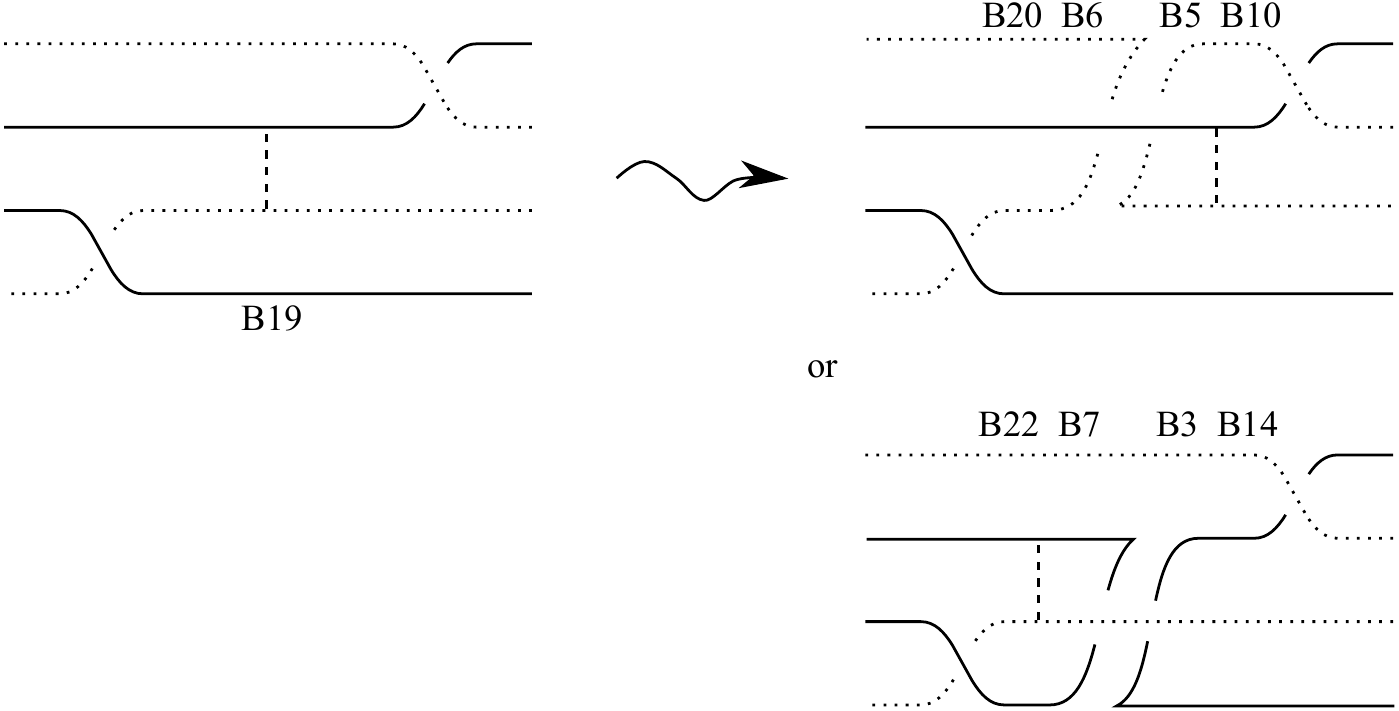}
	\end{center}
	\caption{Case 11.}
	\label{c11}
\end{figure}

\begin{figure}
	\begin{center}
		\includegraphics[width=6in]{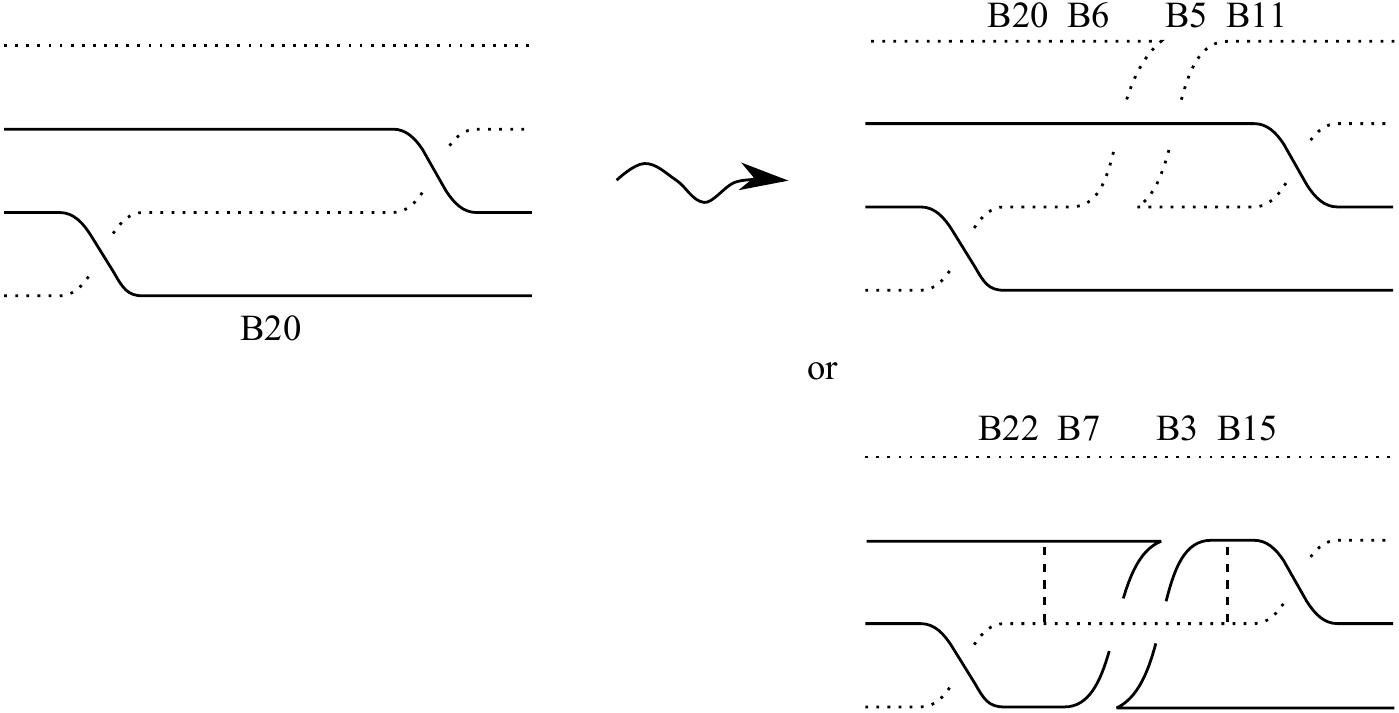}
	\end{center}
	\caption{Case 12.}
	\label{c12}
\end{figure}

\begin{figure}
	\begin{center}
		\includegraphics[width=6in]{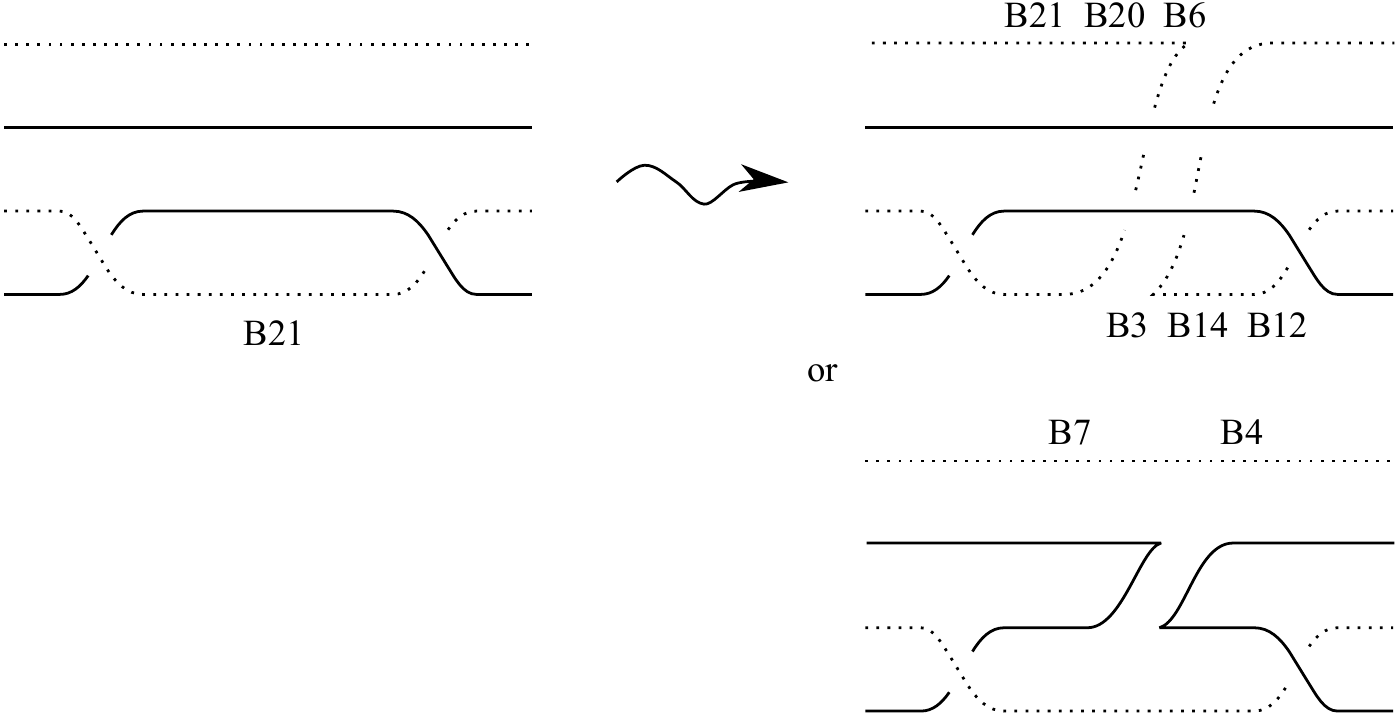}
	\end{center}
	\caption{Case 13.}
	\label{c13}
\end{figure}

\begin{figure}
	\begin{center}
		\includegraphics[width=6in]{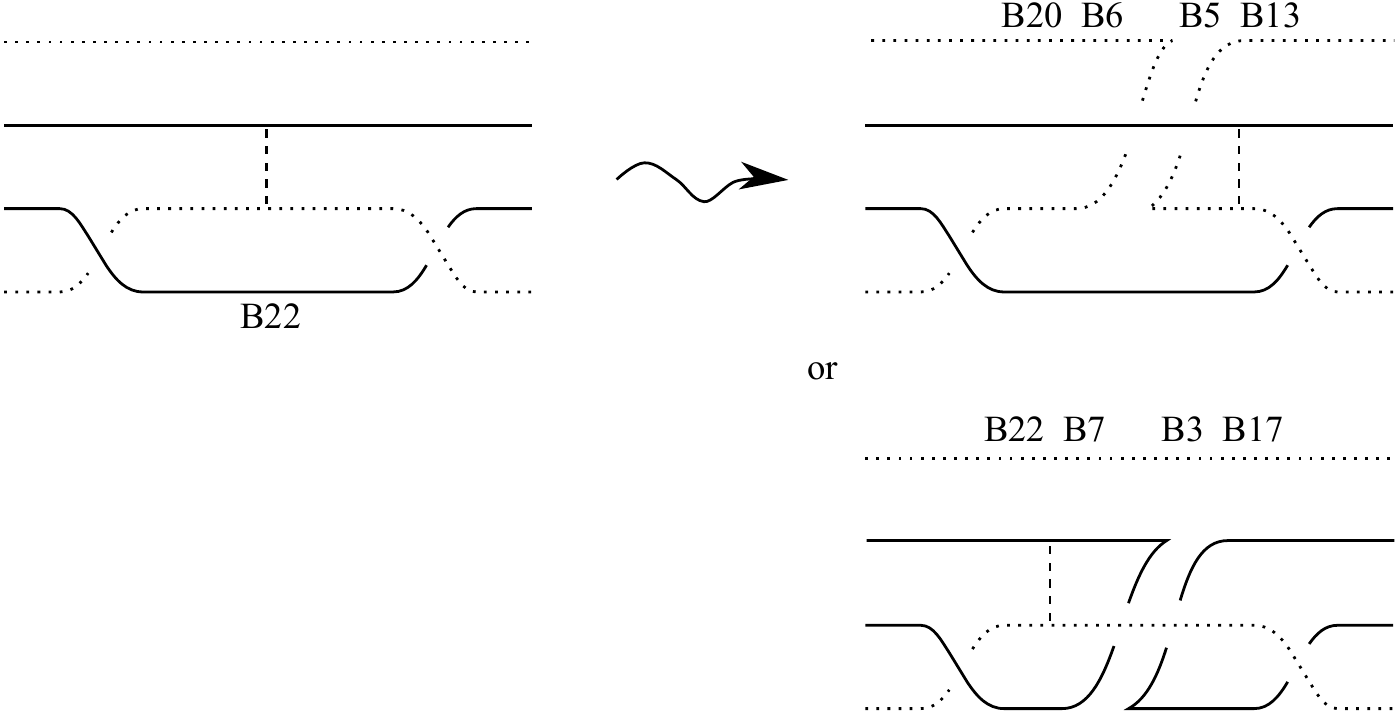}
	\end{center}
	\caption{Case 14.}
	\label{c14}
\end{figure}
\newpage
\noindent
\textbf{3.4 Parity of Normal Rulings under Legendrian isotopy.}
In this section, we show that the parity of normal rulings is invariant under Legendrian isotopy.

\begin{Prop}
\label{p1}
Suppose we have a normal ruling of a Legendrian link $K$. Applying regular isotopy to the front projection of $K$ will not change the parity of the corresponding normal ruling (under the one-to-one correspondence of Theorem \ref{Chekanov}).
\end{Prop}

\begin{proof}
The only case that a regular isotopy might change the number of clasps is when the regular isotopy alter the normal ruling. This is exactly when 2 crossings are interchanged as in Figure \ref{rege}. Note that we assume that there is no other crossings/cusps in the same vertical strips in the Figure \ref{rege}. Moreover, exactly one of the two crossings must be a switch\footnote{For other cases, the number of clasps does not change by Lemma \ref{lres}.}, and there are exactly 2 eyes involved in the regular isotopy. Hence, in this situation, we may obtain their resolutions as in Figure \ref{resedd} (we use color to distinguish 2 eyes). We want to show that their numbers of clasps have the same parity. First, we prove this for case 1. After applying enhanced cuts to both front diagrams of case 1, we have the resulting objects as in the right side of Figure \ref{res1dd}. Furthermore, under regular isotopy, R2 and R3 of resolutions, they are the same, see Figure \ref{rem1dd}. Thus, the resulting diagrams after enhanced cuts must have the same number of clasps by Lemma \ref{lres}. Also, by Lemma \ref{lc}, the two diagrams before applying enhanced cuts must have the same parity for their numbers of clasps.

In Figure \ref{res1dd} - \ref{rem1dd}, local pictures show that our process produces a clasp between red eyes and blue eyes, which can be seen explicitly in Figure \ref{resm1dd}. Notice that the left blocks of local pictures will never have a clasp since they have nested components. On the other hand, depending on which blocks they are at the right side of local pictures, the number of clasps between red eyes and blue eyes is either preserved or increased by 2, see Figure \ref{resm1b1} - \ref{resm1b3}.

For case 2, analogous ideas can be applied as illustrated in Figure \ref{res2dd} - \ref{resm2b3}. In this case, we consider the left blocks of local pictures in stead of the right blocks.
\end{proof}

\begin{figure}
	\begin{center}
		\includegraphics[width=6.3in]{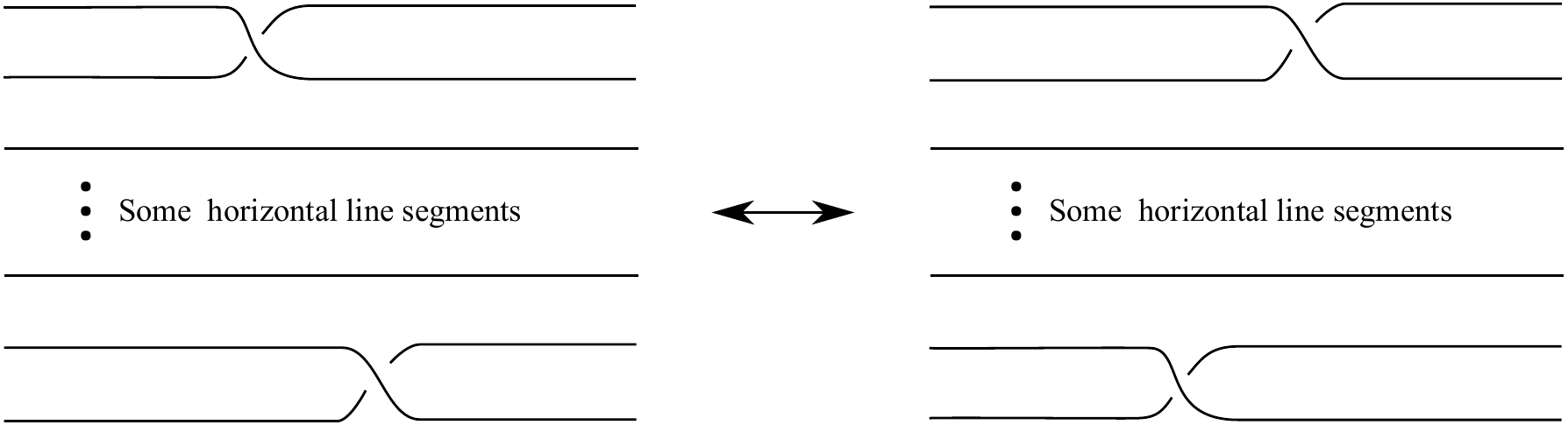}
	\end{center}
	\caption{Regular isotopy that may change the number of clasps.}
	\label{rege}
\end{figure}

\begin{figure}
	\begin{center}
		\vspace{0.9cm}
		\includegraphics[width=6.3in]{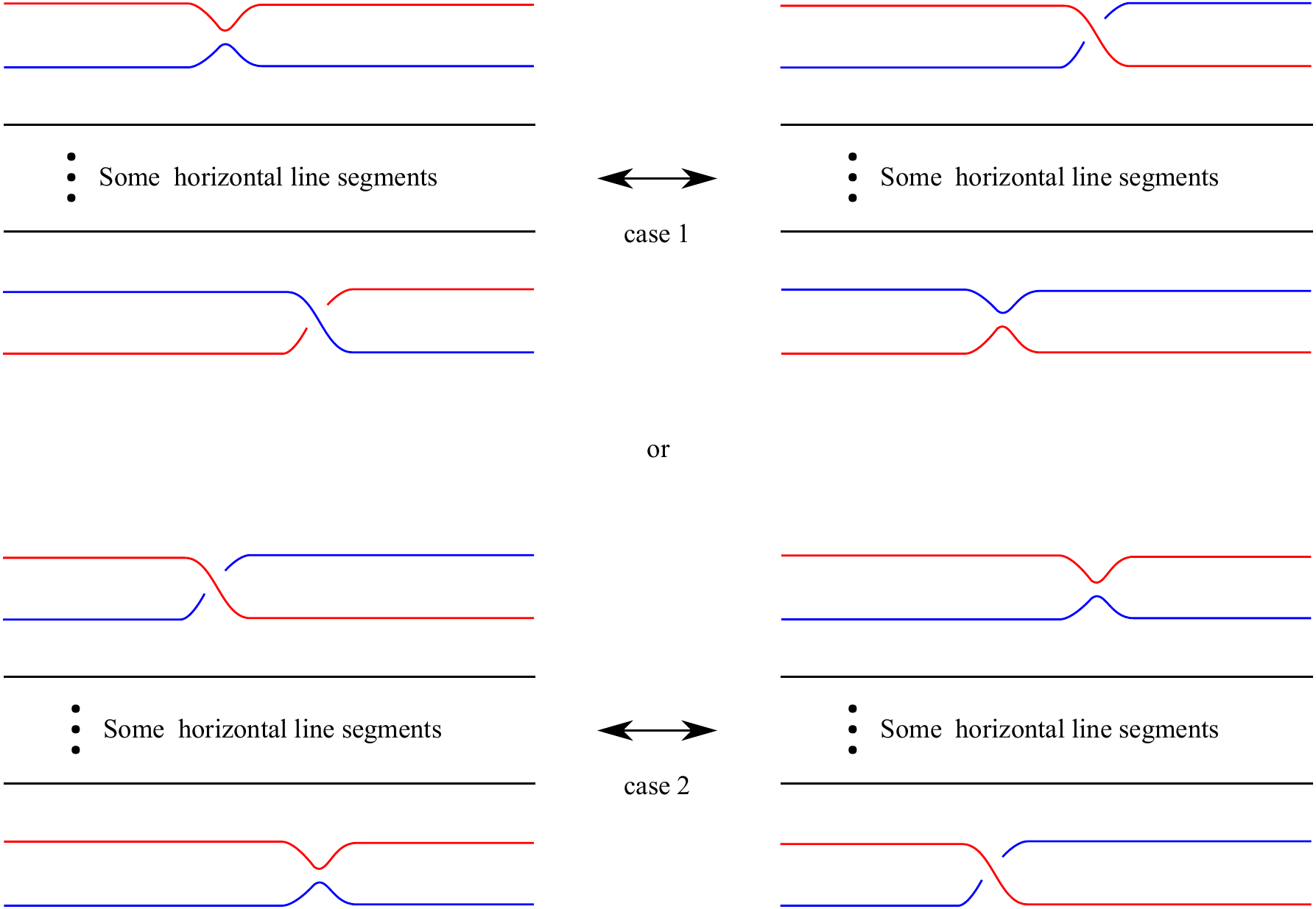}
	\end{center}
	\caption{Possible resolutions of regular isotopy that may change the number of clasps.}
	\label{resedd}
\end{figure}

\begin{figure}
	\begin{center}
		\includegraphics[width=6.3in]{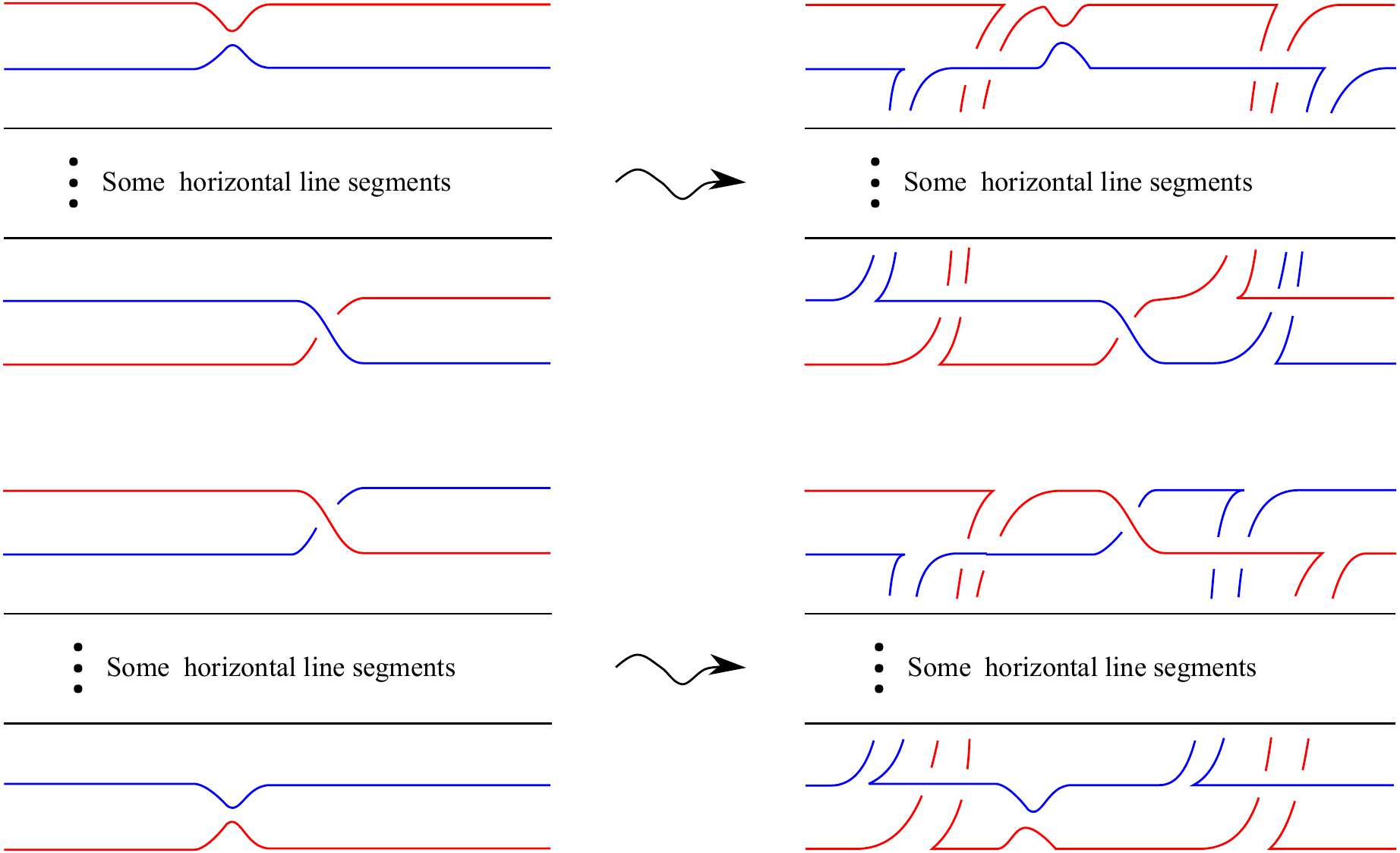}
	\end{center}
	\caption{Applying enhanced cuts to case 1.}
	\label{res1dd}
\end{figure}

\begin{figure}
	\begin{center}
		\includegraphics[width=6.3in]{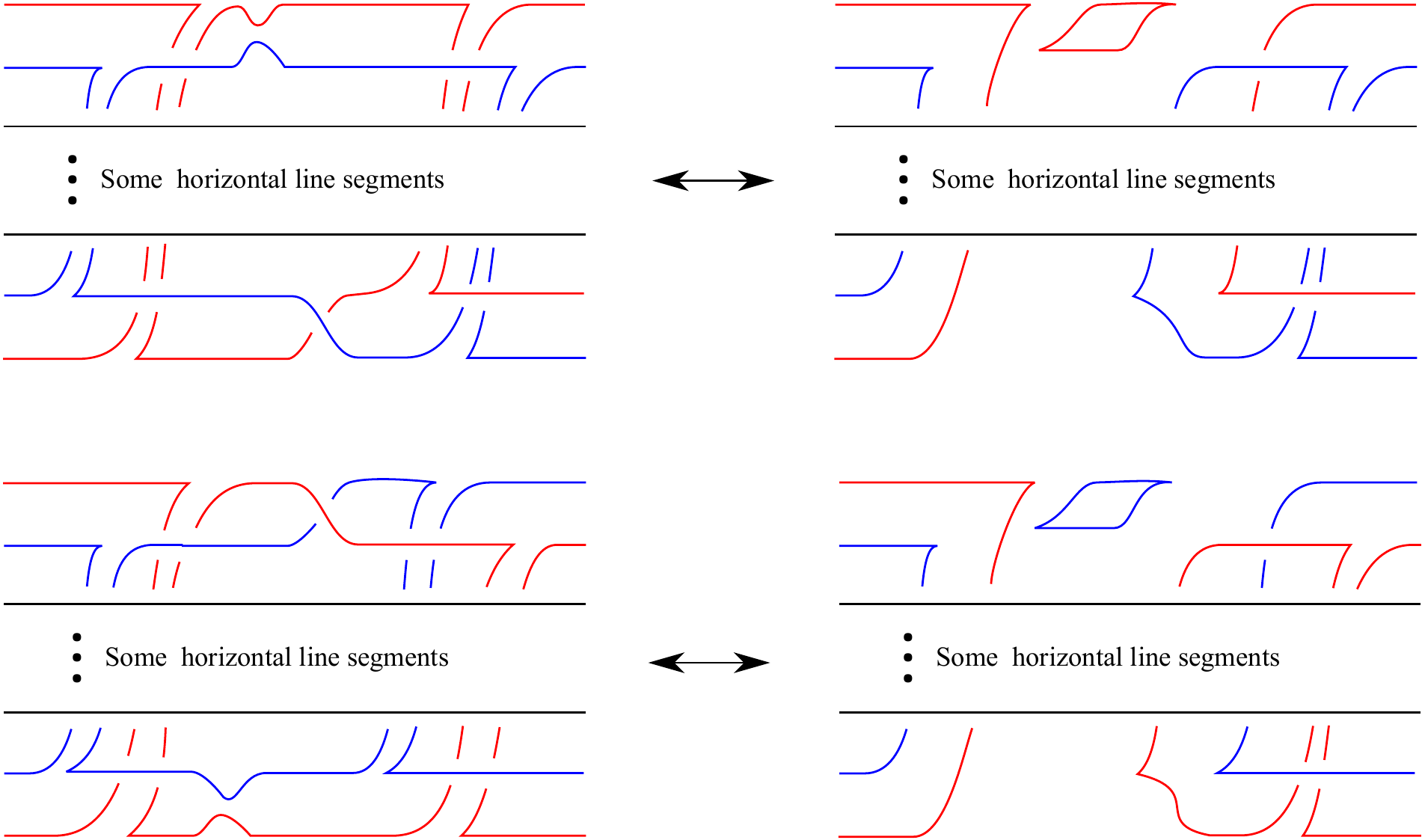}
	\end{center}
	\caption{Applying enhanced cuts to case 1 gives Legendrian isotopic front diagrams.}
	\label{rem1dd}
\end{figure}

\begin{figure}
	\begin{center}
		\includegraphics[width=6.3in]{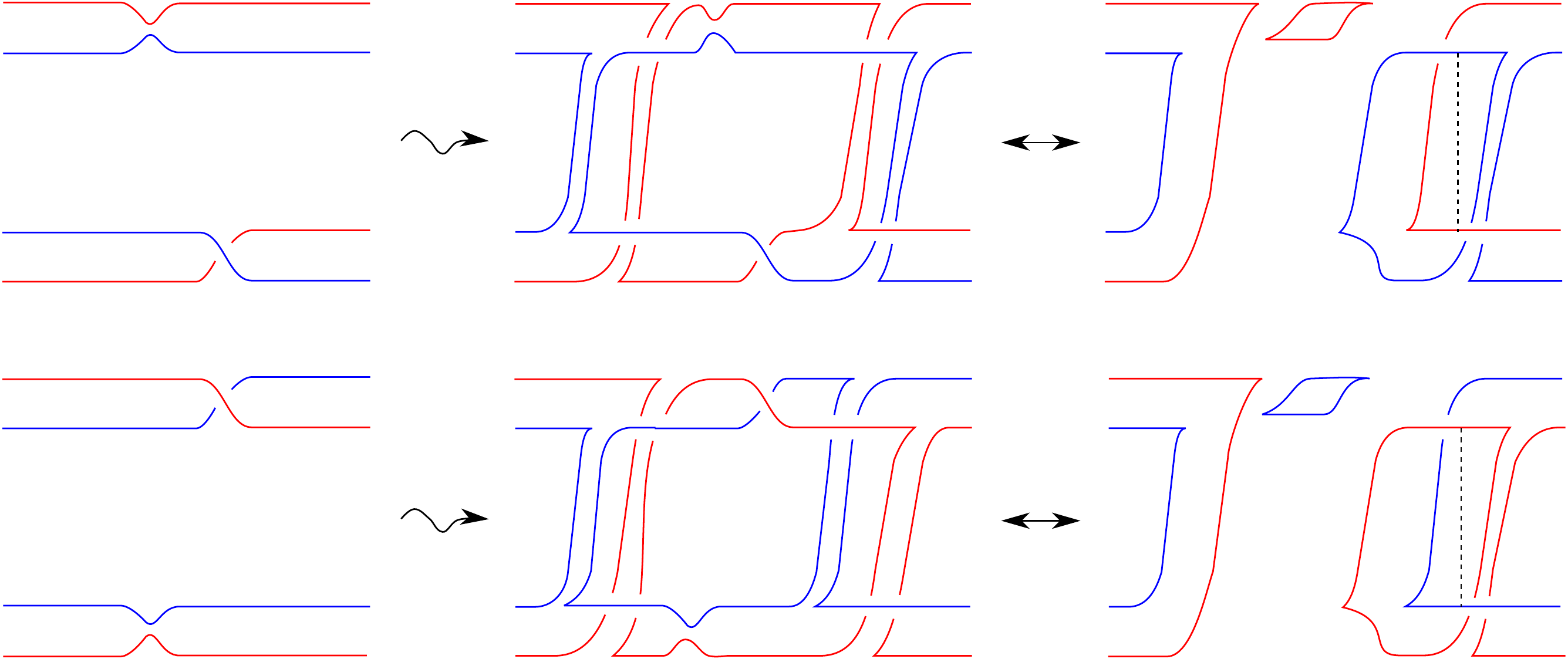}
	\end{center}
	\caption{Local pictures for red eyes and blue eyes in case 1.}
	\label{resm1dd}
	\vspace{1.5cm}
	\begin{center}
		\includegraphics[width=6.3in]{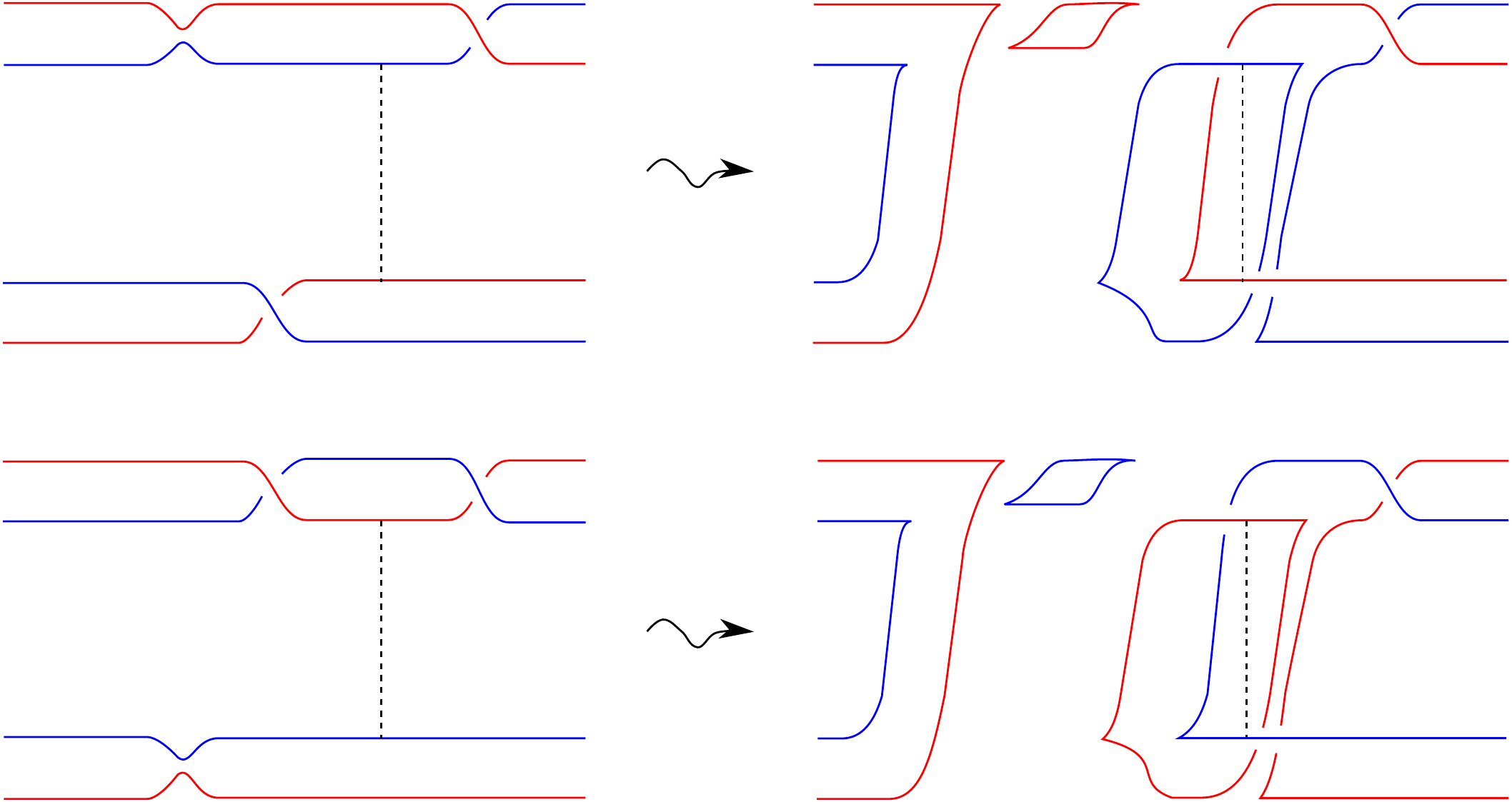}
	\end{center}
	\caption{Subcase 1 of local pictures for red eyes and blue eyes in case 1.}
	\label{resm1b1}
\end{figure}

\begin{figure}
	\begin{center}
		\includegraphics[width=6.3in]{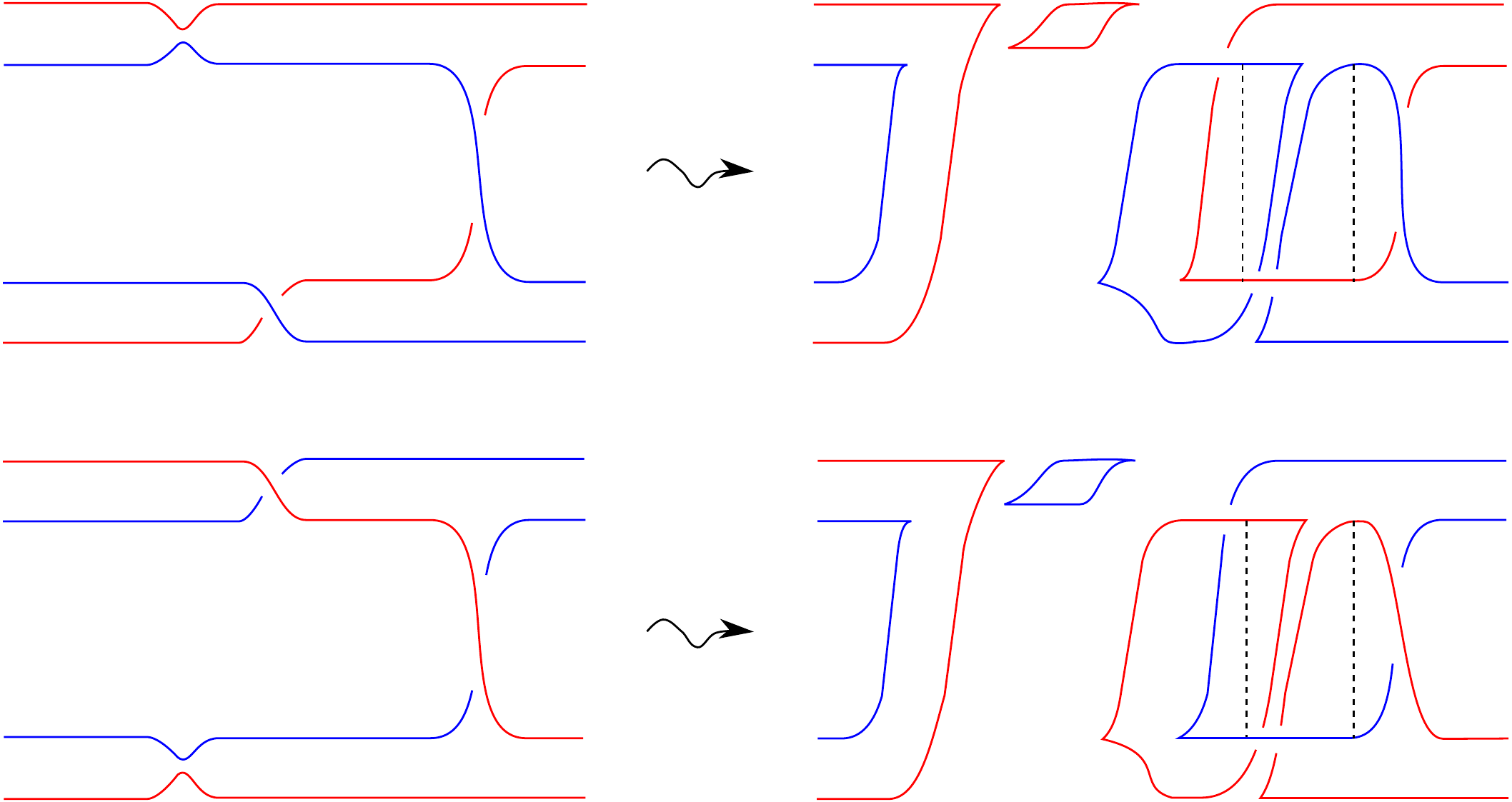}
	\end{center}
	\caption{Subcase 2 of local pictures for red eyes and blue eyes in case 1.}
	\label{resm1b2}
\end{figure}

\begin{figure}
	\begin{center}
		\vspace{.8cm}
		\includegraphics[width=6.3in]{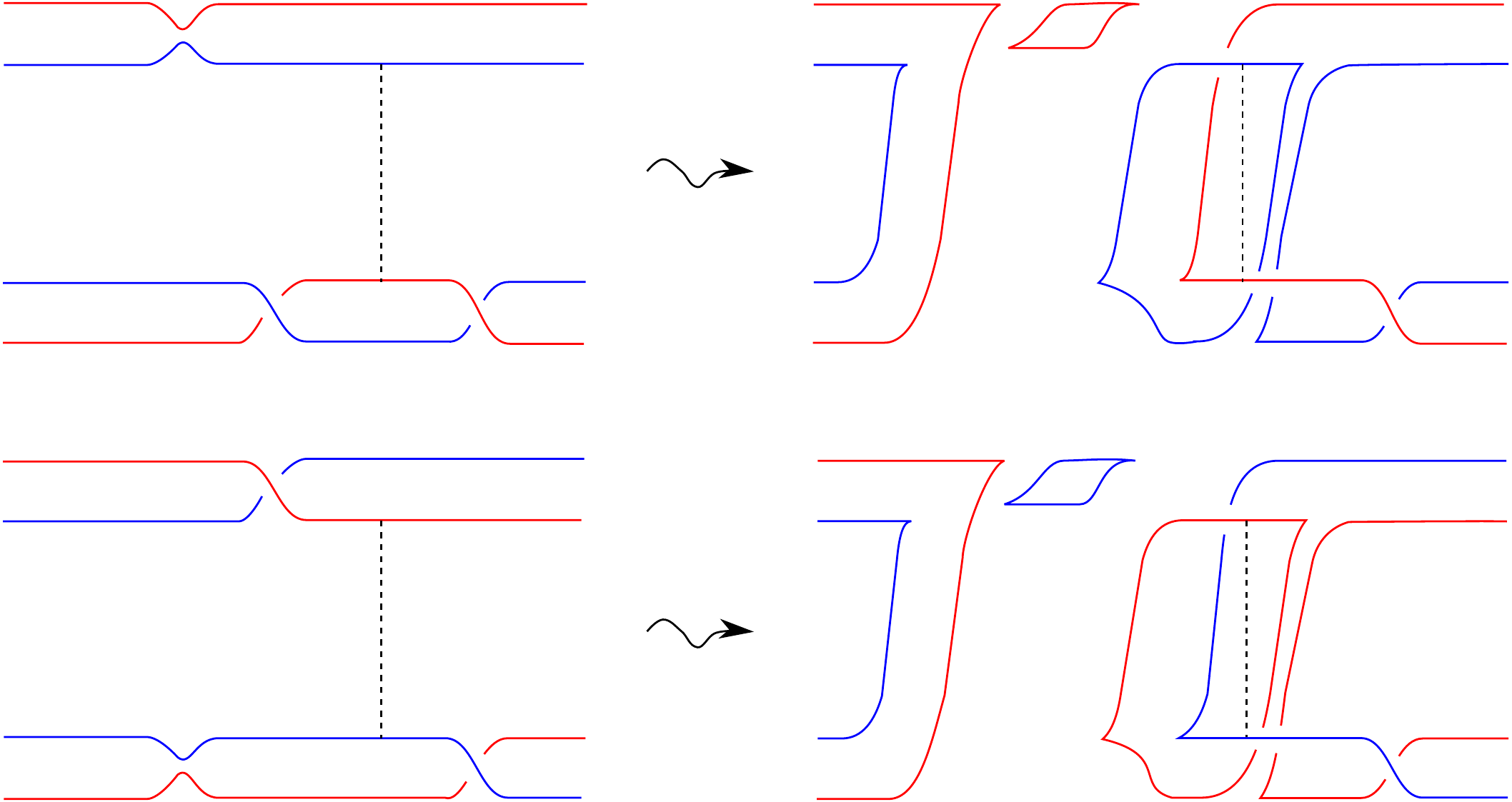}
	\end{center}
	\caption{Subcase 3 of local pictures for red eyes and blue eyes in case 1.}
	\label{resm1b3}
\end{figure}

\begin{figure}
	\begin{center}
		\includegraphics[width=6.3in]{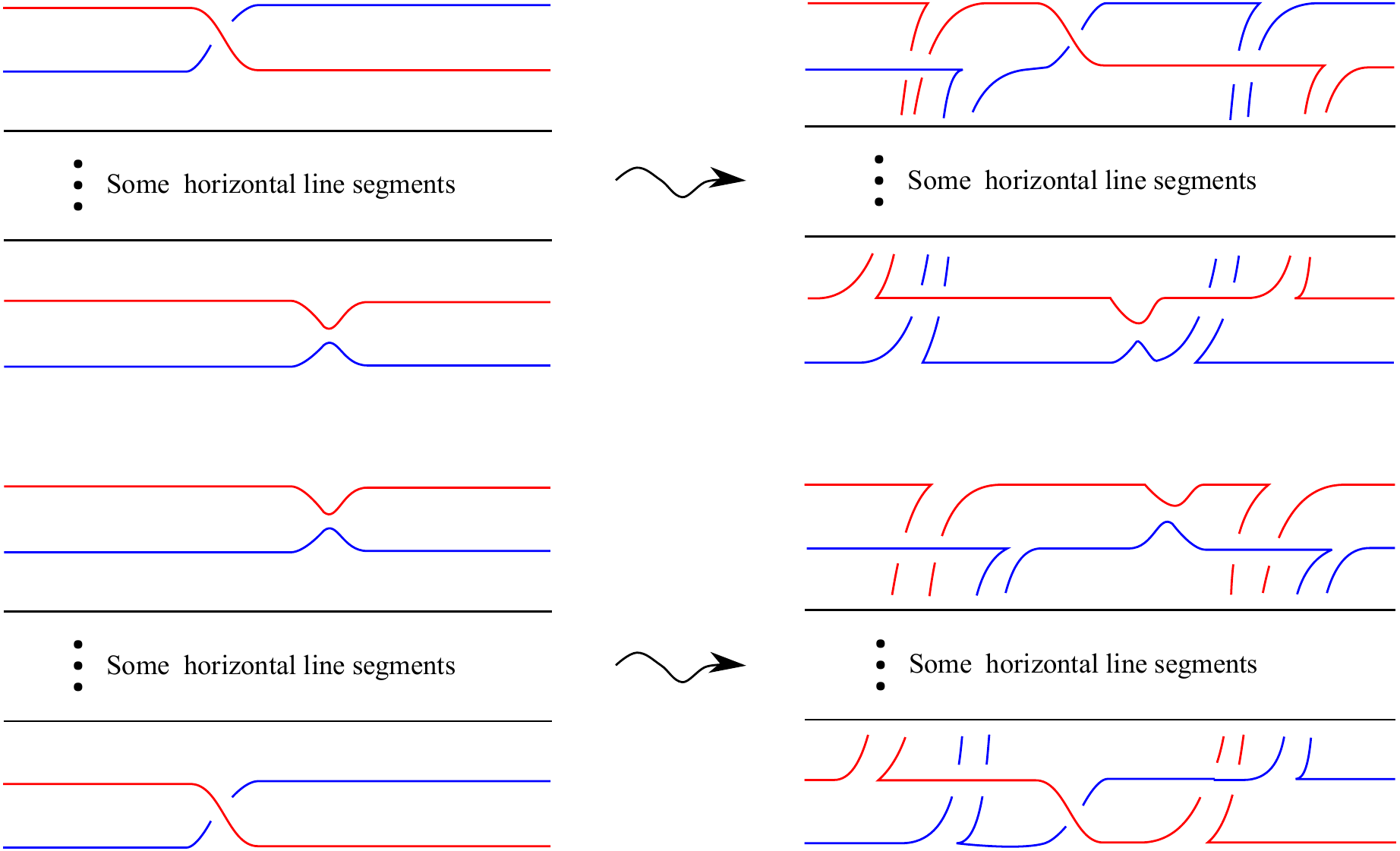}
	\end{center}
	\caption{Applying enhanced cuts to case 2.}
	\label{res2dd}
\end{figure}

\begin{figure}
	\begin{center}
		\includegraphics[width=6.3in]{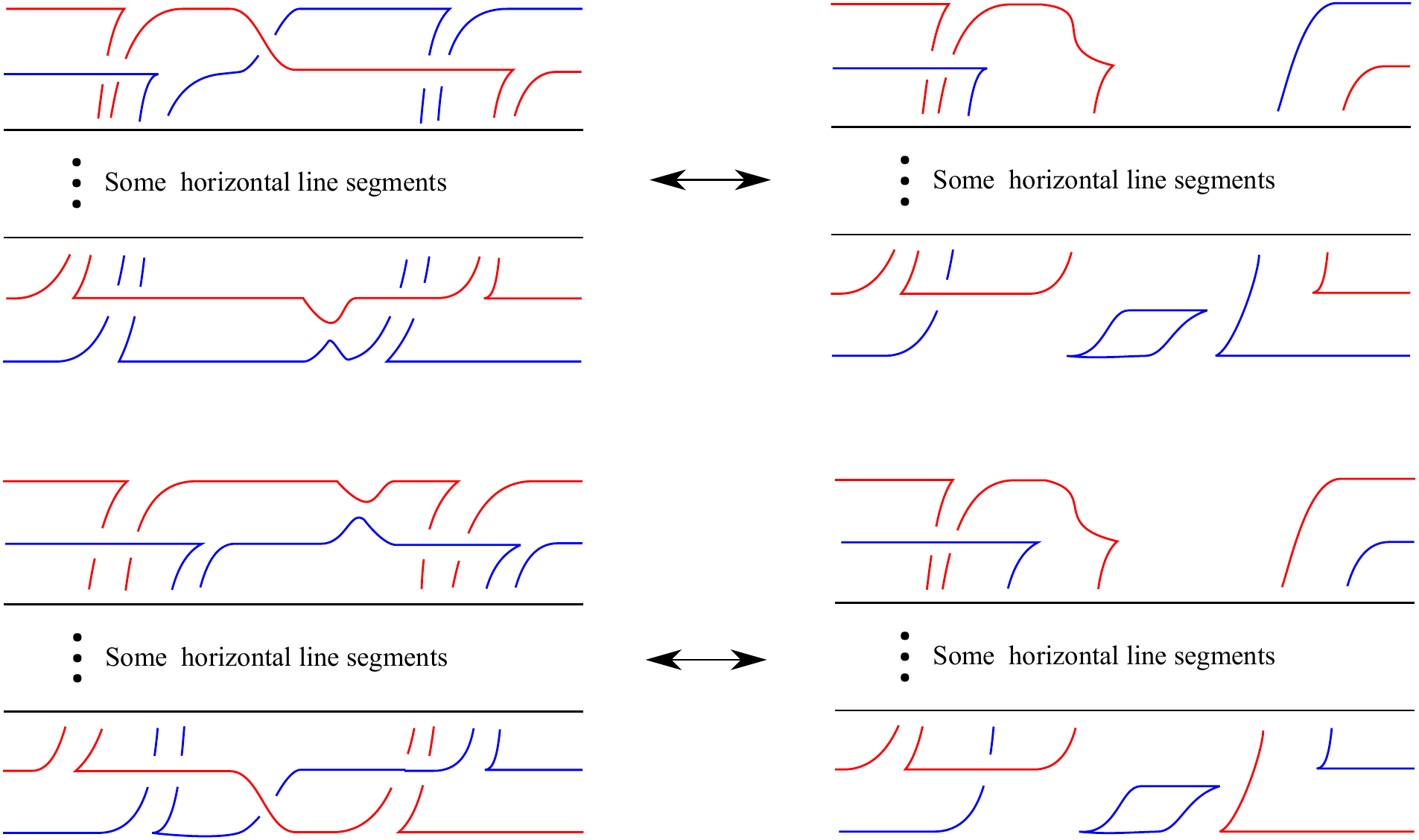}
	\end{center}
	\caption{Applying enhanced cuts to case 2 gives Legendrian isotopic front diagrams.}
	\label{rem2dd}
\end{figure}

\begin{figure}
	\begin{center}
		\vspace{.6cm}
		\includegraphics[width=6.3in]{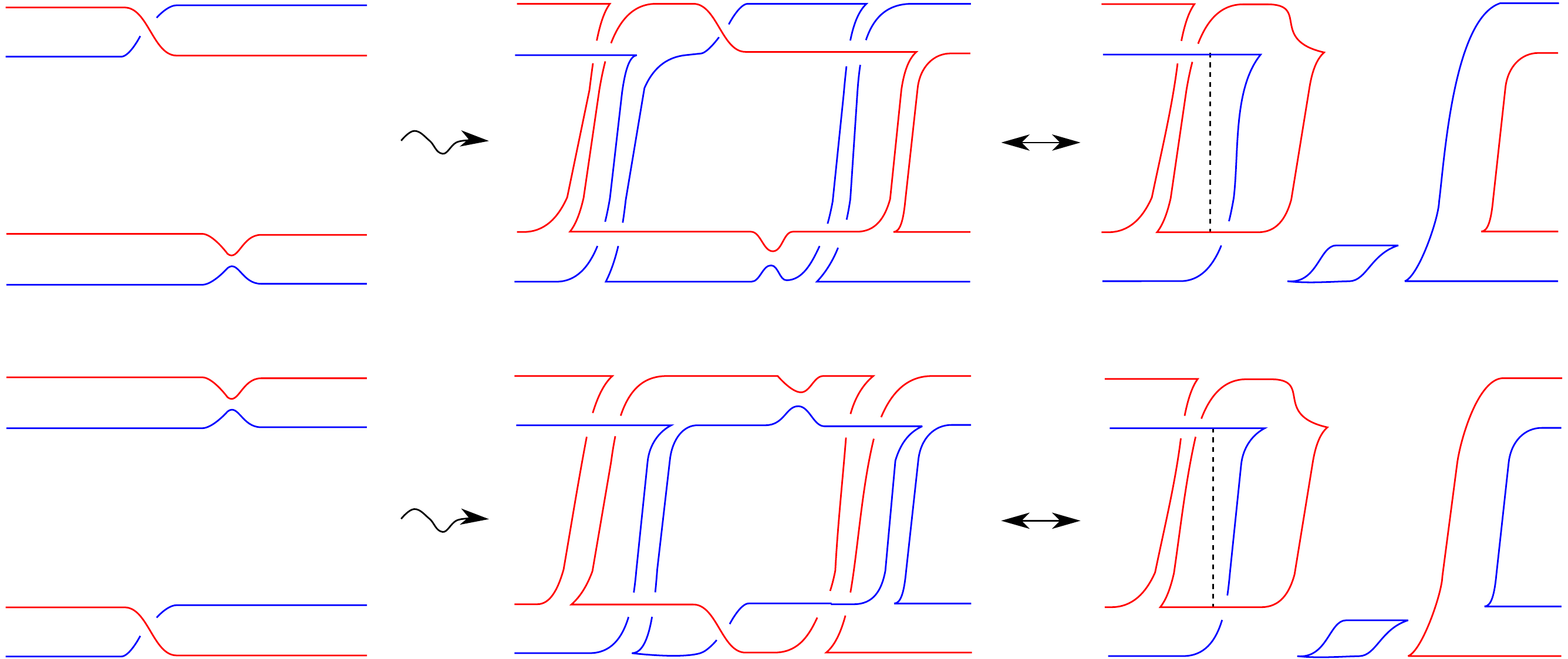}
	\end{center}
	\caption{Local pictures for red eyes and blue eyes in case 2.}
	\label{resm2dd}
	\begin{center}
		\vspace{1.5cm}
		\includegraphics[width=6.3in]{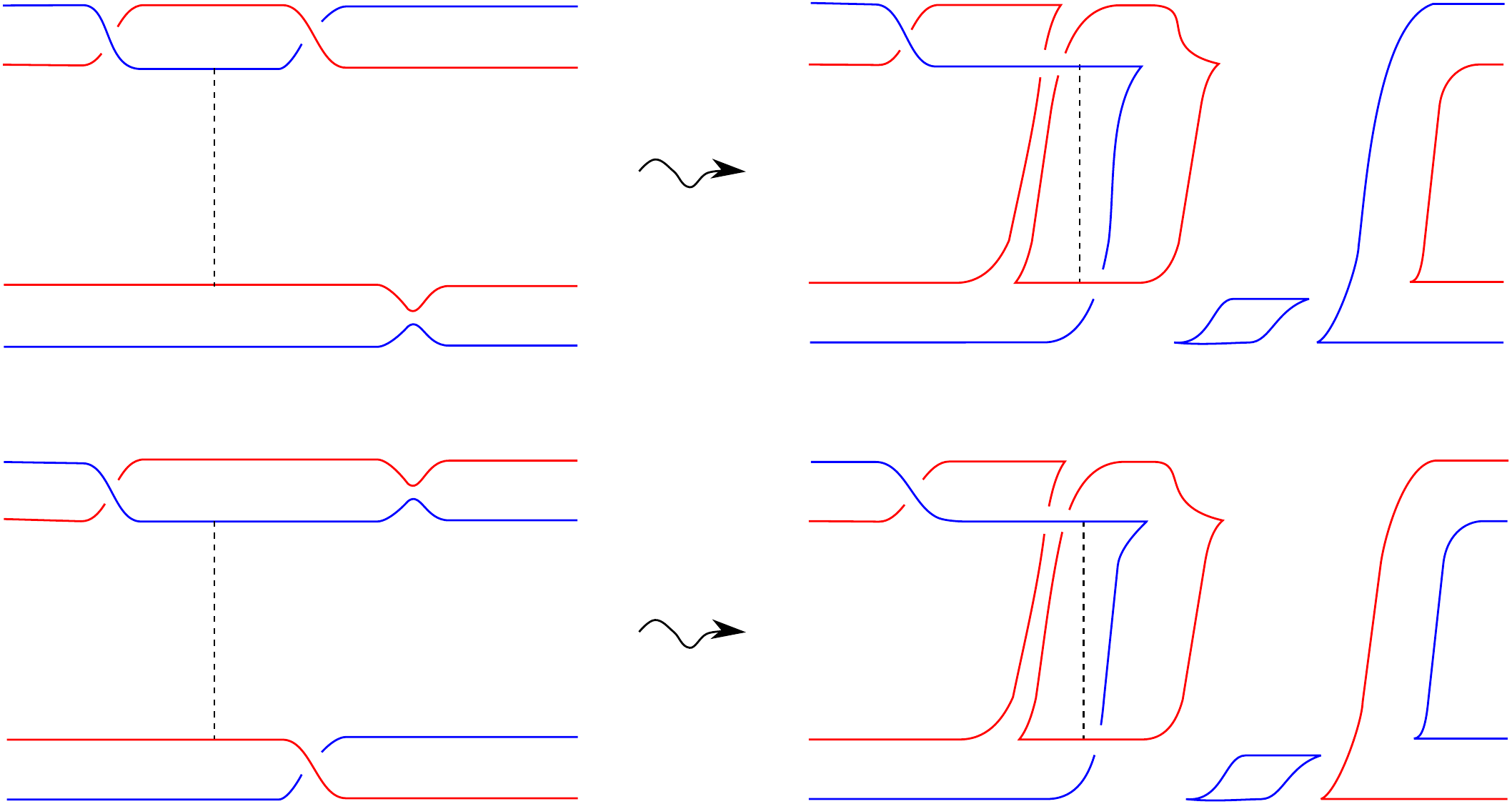}
	\end{center}
	\caption{Subcase 1 of local pictures for red eyes and blue eyes in case 2.}
	\label{resm2b1}
\end{figure}

\begin{figure}
	\begin{center}
		\includegraphics[width=6.3in]{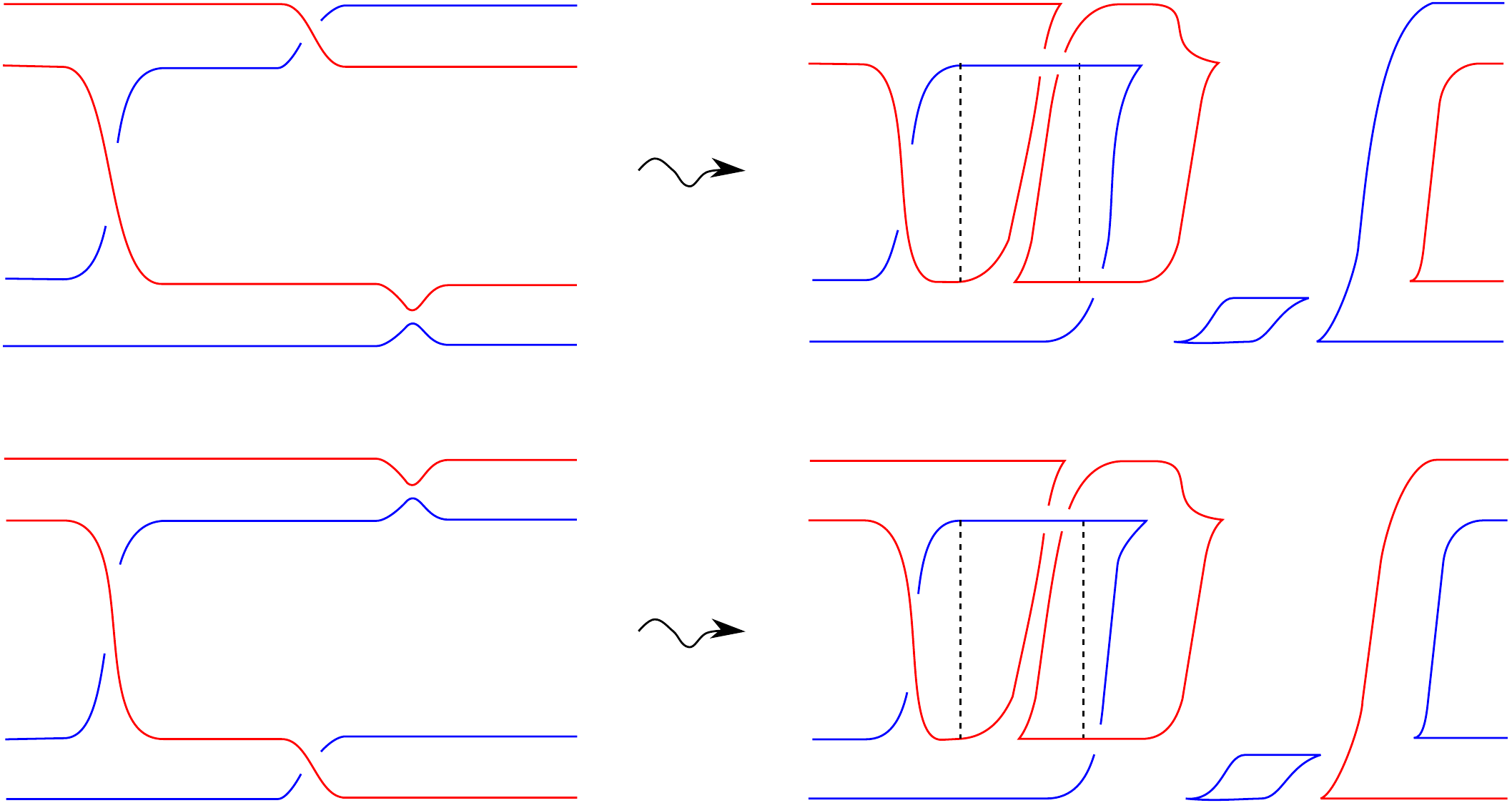}
	\end{center}
	\caption{Subcase 2 of local pictures for red eyes and blue eyes in case 2.}
	\label{resm2b2}
\end{figure}

\begin{figure}
	\begin{center}
		\vspace{.8cm}
		\includegraphics[width=6.3in]{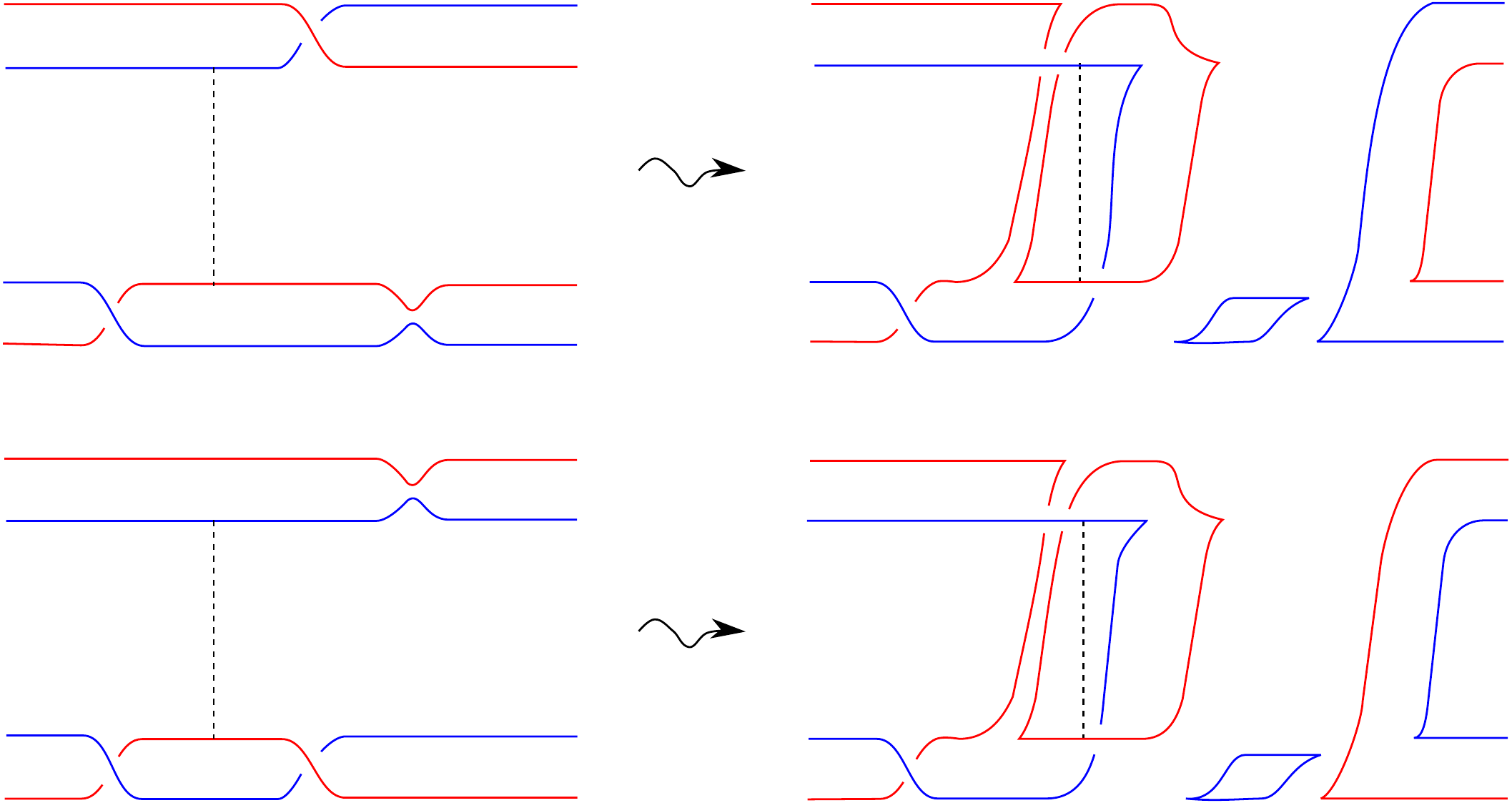}
	\end{center}
	\caption{Subcase 3 of local pictures for red eyes and blue eyes in case 2.}
	\label{resm2b3}
\end{figure}
\newpage
Next, we prove akin statement for Legendrian Reidemeister moves.

\begin{Lem}
\label{l6}
Suppose we have a normal ruling of a Legendrian link $K$. Applying a Legendrian Reidemeister move to $\pi(K)$ will not change the parity of the corresponding normal ruling (under the one-to-one correspondence of Theorem \ref{Chekanov}).
\end{Lem}

\begin{proof}
It is easy to see that applying R1 or R2 to $\pi(K)$ will not change the number of clasps of the corresponding normal ruling under the one-to-one correspondence in Theorem \ref{Chekanov}. For R3, under the correspondence, we have that the only case that might change the number of clasps is when there is exactly 1 switch, i.e. $\{a\} \leftrightarrow \{a'\}$ as in the first row of Figure \ref{rc}. This is because other cases of R3 have the same resolutions by regular isotopy or R3 (see Lemma \ref{lres}). We want to show that the correspondence of resolutions as in the second row of Figure \ref{rc} preserves the parity of the number of clasps. In order to count this amount, we need to specify the environment of the resolutions. First, recall that 3 strands in the resolutions must come from 3 different eyes by the definition of normal rulings. So it is enough to consider clasps for only these 3 eyes. Moreover, since there were switches at $a$ and $a'$, components of eyes must follow the normality conditions. We list all possibilities in Figure \ref{rc1} - \ref{rc3}. Because there is a presence/absence of middle block between solid eye and dotted eye, we need to know the other ends of adjacent blocks, i.e. block that has a middle block emerged or blocks that stay next to the middle block, between the same pair of eyes in order to count the number of clasps coming from this pair of eyes (the presence/absence of middle block can change the number of clasps). Similar idea applies for the pair of dash eye and solid eye.

Case 1: Notice that all adjacent blocks have nested components. So they never provide a clasp no matter what their other ends are. Thus, all blocks affected give 1 clasp in total as in Figure \ref{rcc1}.

Case 2: Since adjacent blocks between dotted eye and solid eye have nested components, they give no clasp. So, we only need to consider adjacent blocks between dash eye and solid eye, and we will not look further into adjacent blocks between dotted eye and solid eye. All possible subcases are presented in Figure \ref{rcc2a} - \ref{rcc2c}. Each case preserves parity.

Case 3: It is not hard to see that this case is obtained form interchanging the dash and the dotted eyes of case 2. So, in a sense, this case is a symmetric version of case 2.

Case 4: This case is a symmetric version of case 1.

Case 5: Since adjacent blocks between dotted eye and solid eye have nested components, they give no clasp. So, we only need to consider adjacent blocks between dash eye and solid eye, and we will not look further into adjacent blocks between dotted eye and solid eye. All possible subcases are presented in Figure \ref{rcc5a} - \ref{rcc5c}. Each case preserves parity.

Case 6: Since adjacent blocks between dash eye and solid eye have nested components, they give no clasp. So, we only need to consider adjacent blocks between dotted eye and solid eye, and we will not look further into adjacent blocks between dash eye and solid eye. All possible subcases are presented in Figure \ref{rcc6a} - \ref{rcc6c}. Each case preserves parity.

Case 7: Notice that all adjacent blocks have nested components. So they never provide a clasp no matter what their other ends are. Thus, all blocks affected give 1 clasp in total as in Figure \ref{rcc7}.

Case 8: Since adjacent blocks between dotted eye and solid eye have nested components, they give no clasp. So, we only need to consider adjacent blocks between dash eye and solid eye, and we will not look further into adjacent blocks between dotted eye and solid eye. All possible subcases are presented in Figure \ref{rcc8a} - \ref{rcc8c}. Each case preserves parity.

Case 9: This case is a symmetric version of case 8.

Case 10: This case is a symmetric version of case 7.

Case 11: This case is a symmetric version of case 6.

Case 12: This case is a symmetric version of case 5.
\end{proof}

\begin{figure}
	\begin{center}
		\vspace{1.5cm}
		\includegraphics{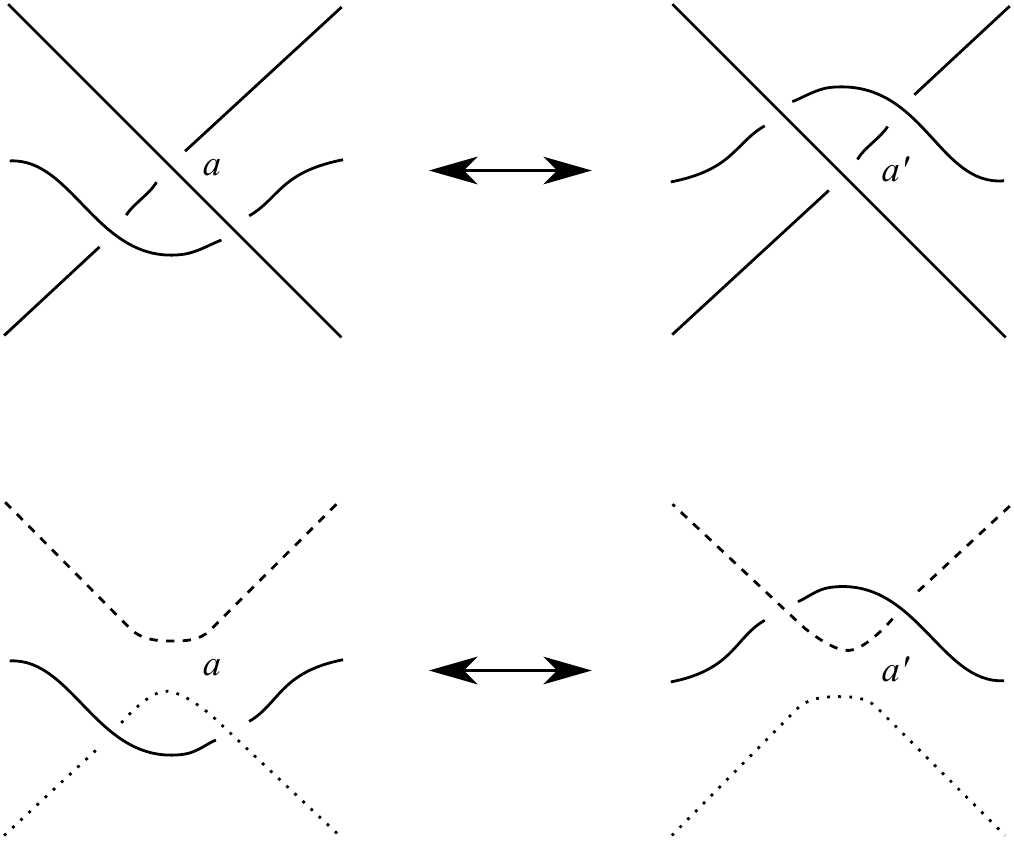}
	\end{center}
	\caption{R3 that might change the number of clasps.}
	\label{rc}
\end{figure}

\begin{figure}
	\begin{center}
		\includegraphics{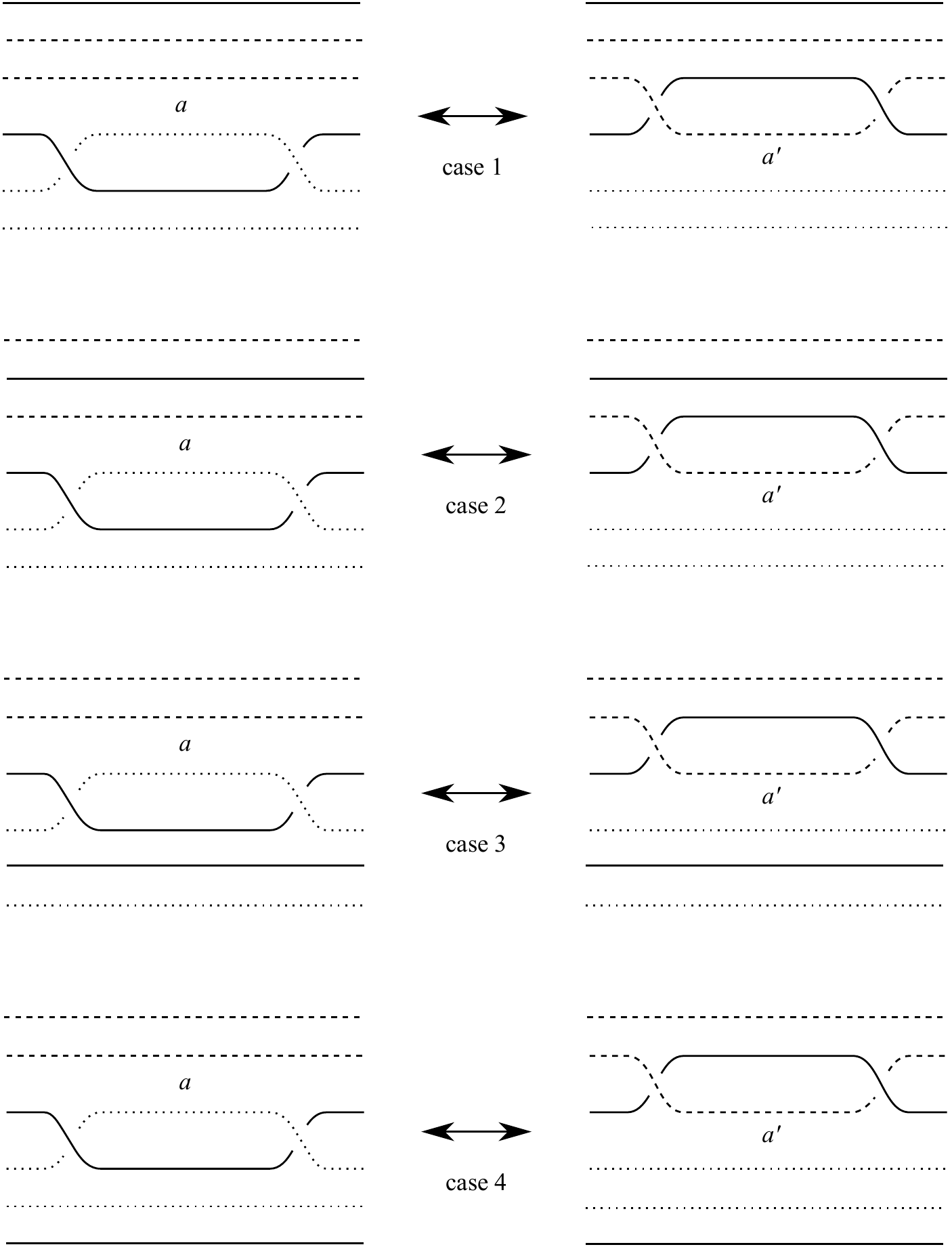}
	\end{center}
	\caption{Case 1 - 4.}
	\label{rc1}
\end{figure}

\begin{figure}
	\begin{center}
		\includegraphics{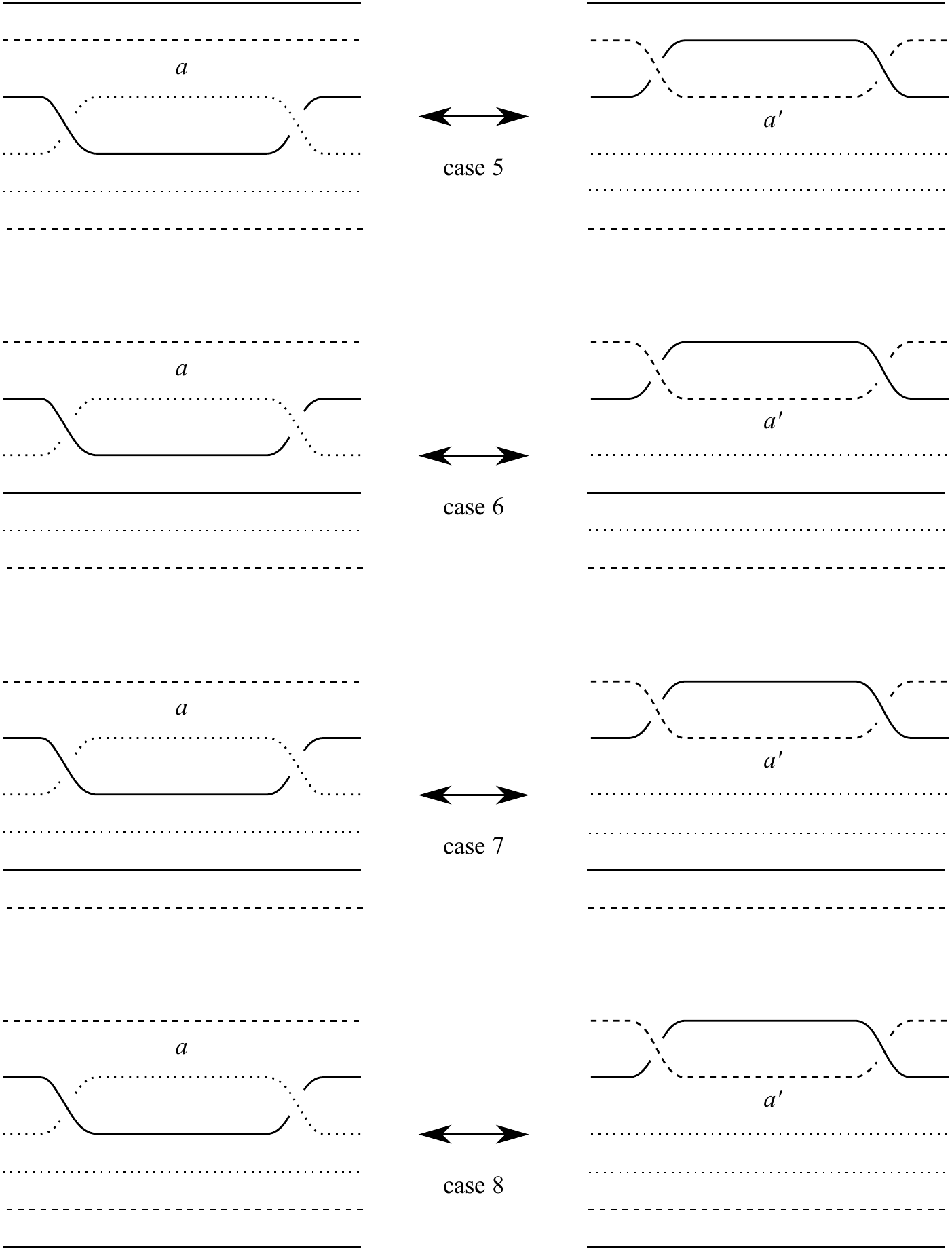}
	\end{center}
	\caption{Case 5 - 8.}
	\label{rc2}
\end{figure}

\begin{figure}
	\begin{center}
		\includegraphics{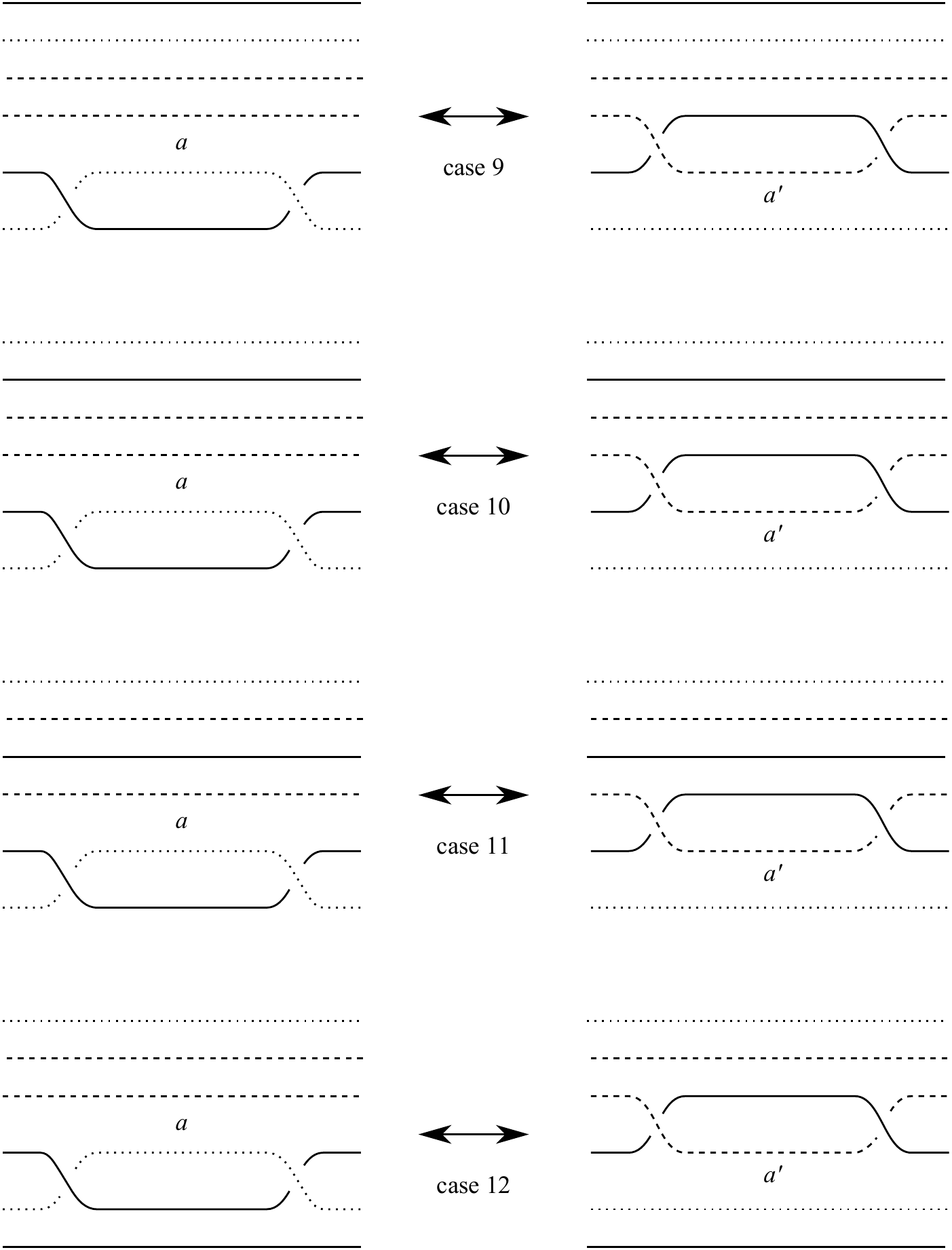}
	\end{center}
	\caption{Case 9 - 12.}
	\label{rc3}
\end{figure}

\begin{figure}
	\begin{center}
		\includegraphics{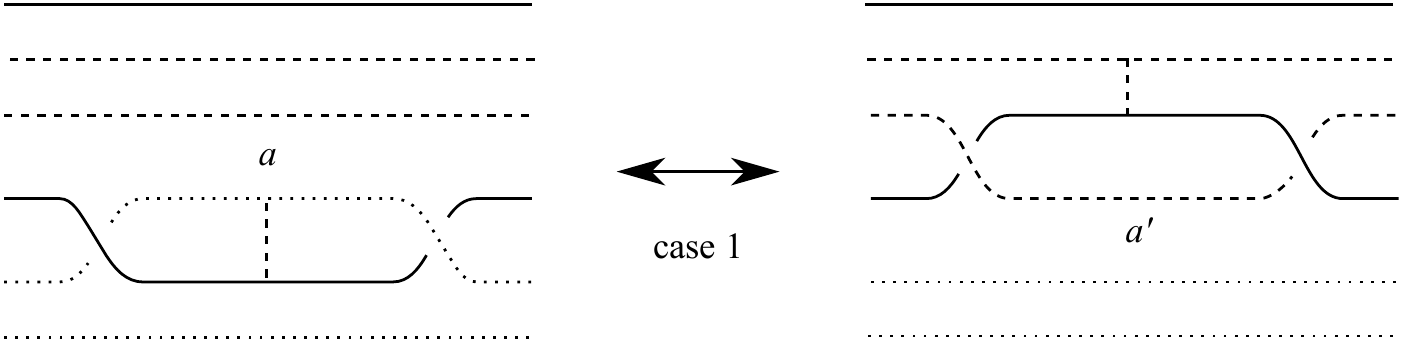}
	\end{center}
	\caption{Case 1 with its clasp exhibited.}
	\label{rcc1}
\end{figure}

\begin{figure}
	\begin{center}
		\includegraphics[width=6.3in]{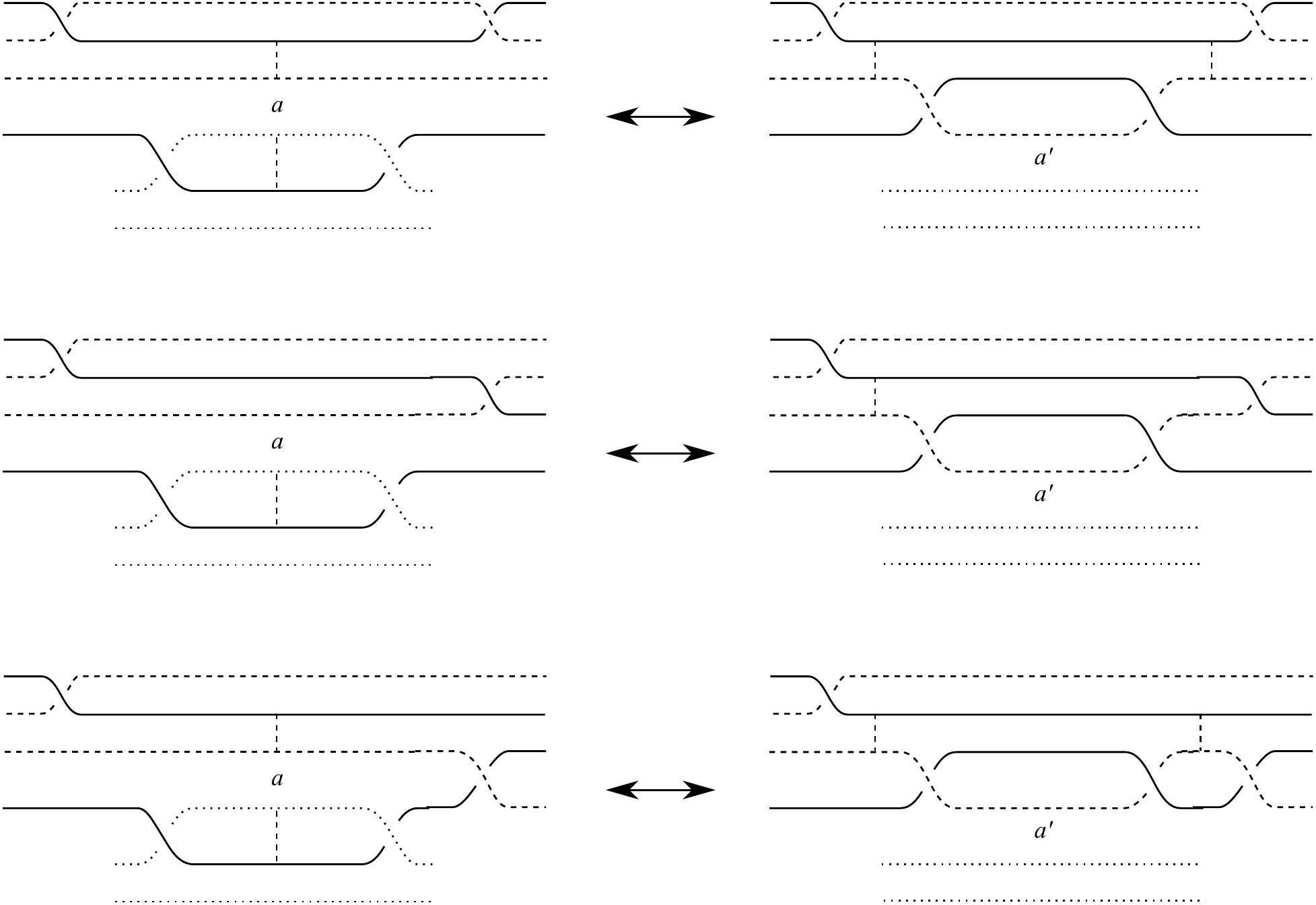}
	\end{center}
	\caption{Case 2 with clasps exhibited for each subcase.}
	\label{rcc2a}
\end{figure}

\begin{figure}
	\begin{center}
		\includegraphics[width=6.3in]{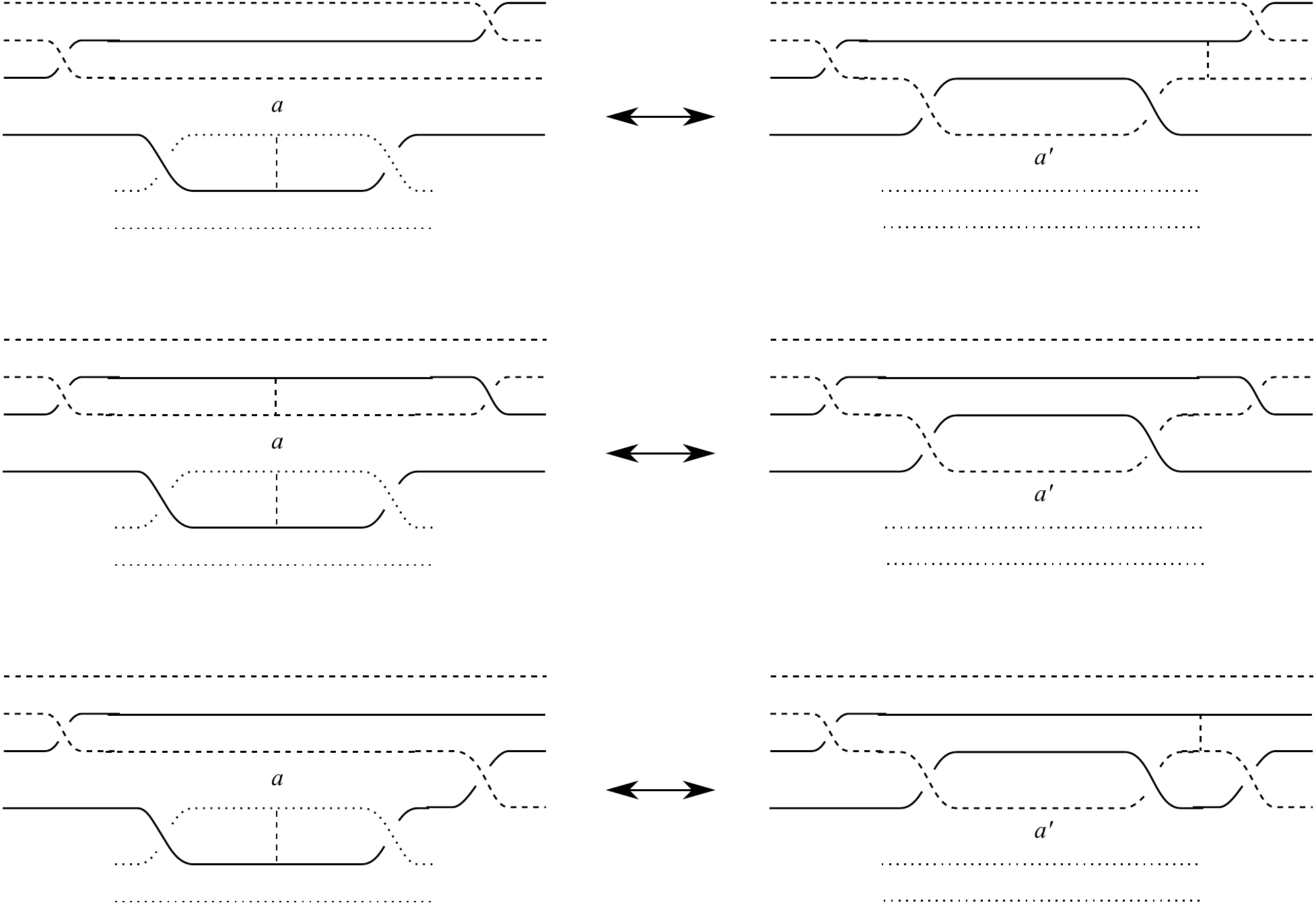}
	\end{center}
	\caption{Case 2 with clasps exhibited for each subcase (continue).}
	\label{rcc2b}
\end{figure}

\begin{figure}
	\begin{center}
		\includegraphics[width=6.3in]{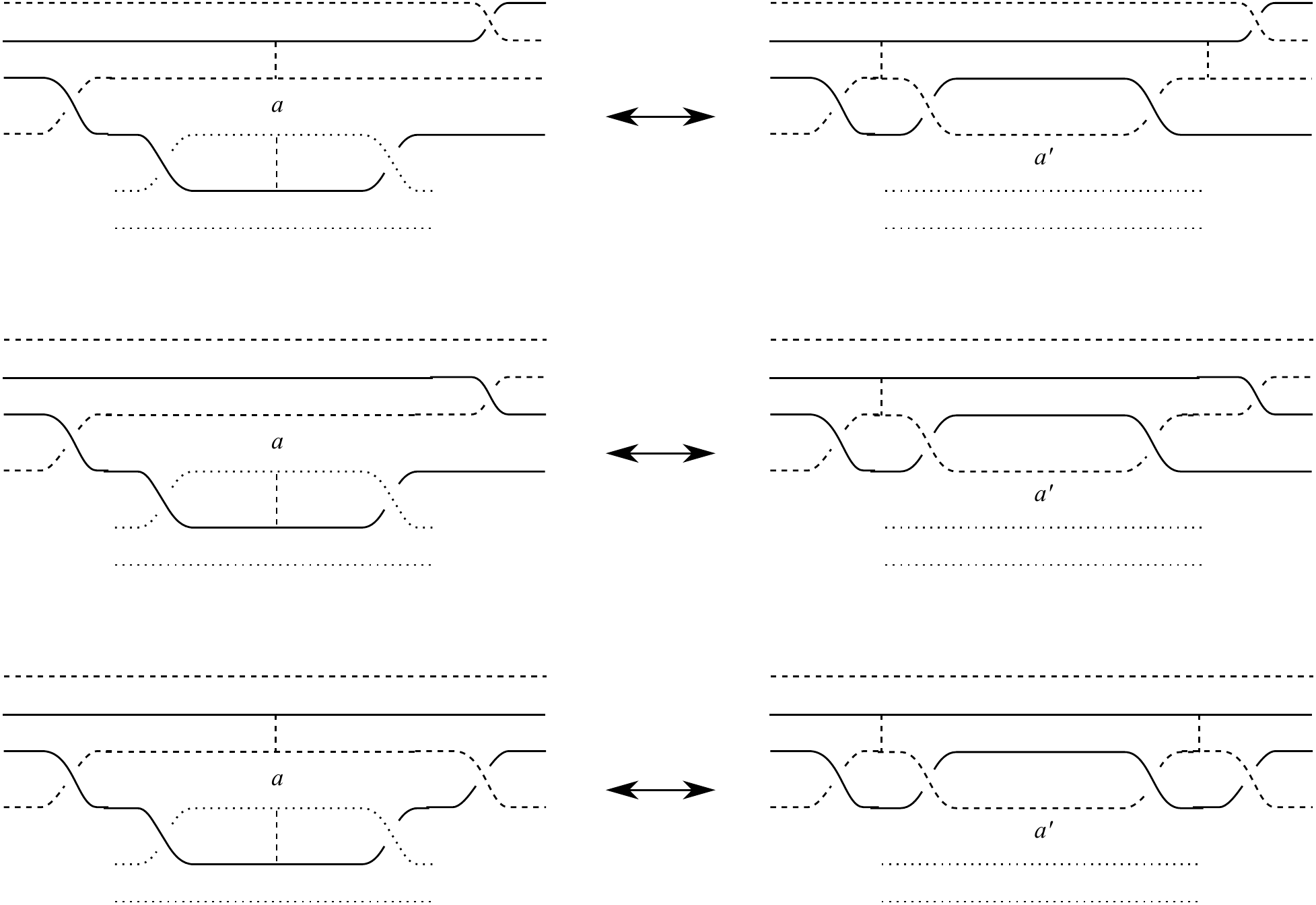}
	\end{center}
	\caption{Case 2 with clasps exhibited for each subcase (continue).}
	\label{rcc2c}
\end{figure}

\begin{figure}
	\begin{center}
		\includegraphics[width=6.3in]{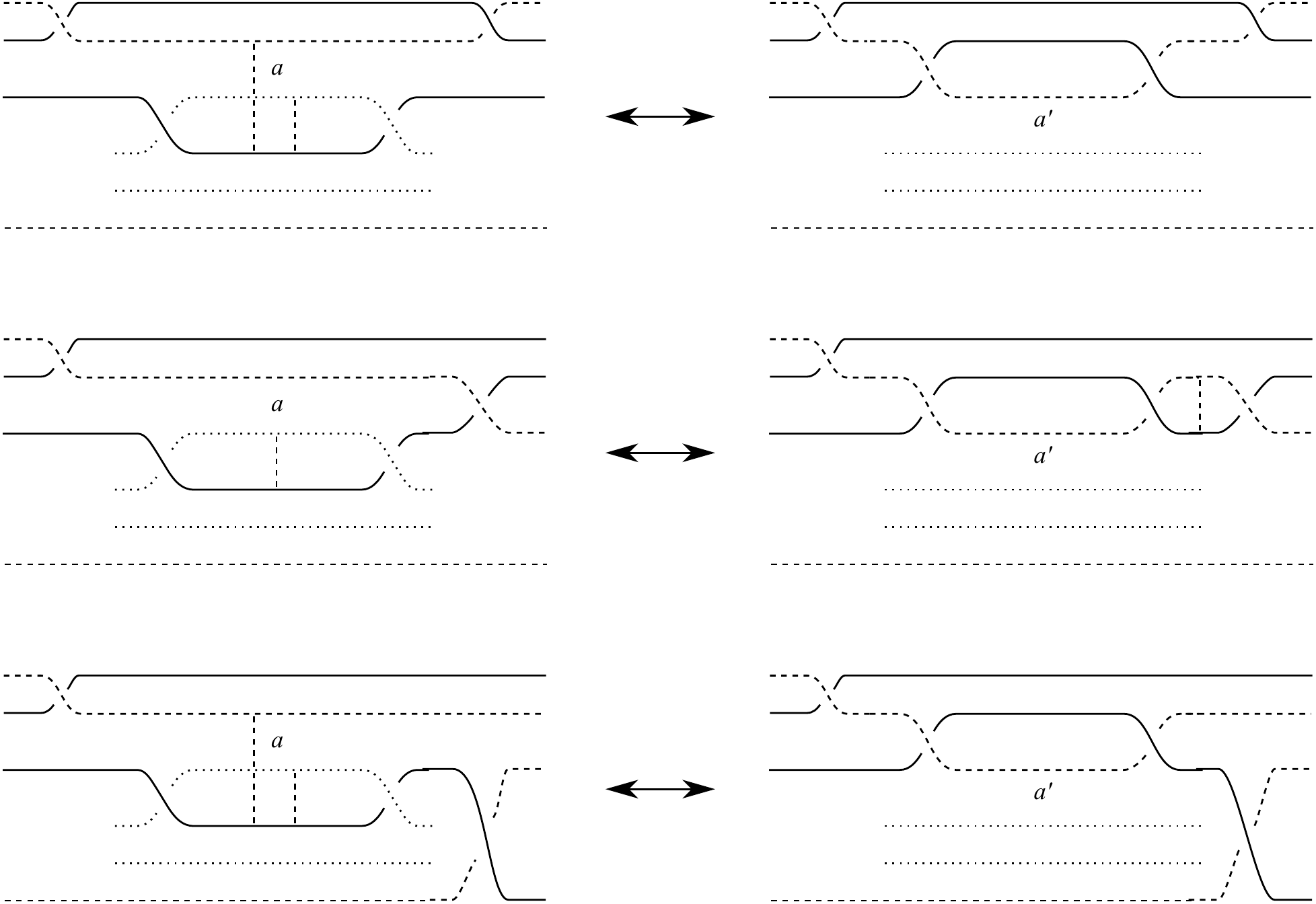}
	\end{center}
	\caption{Case 5 with clasps exhibited for each subcase.}
	\label{rcc5a}
\end{figure}

\begin{figure}
	\begin{center}
		\includegraphics[width=6.3in]{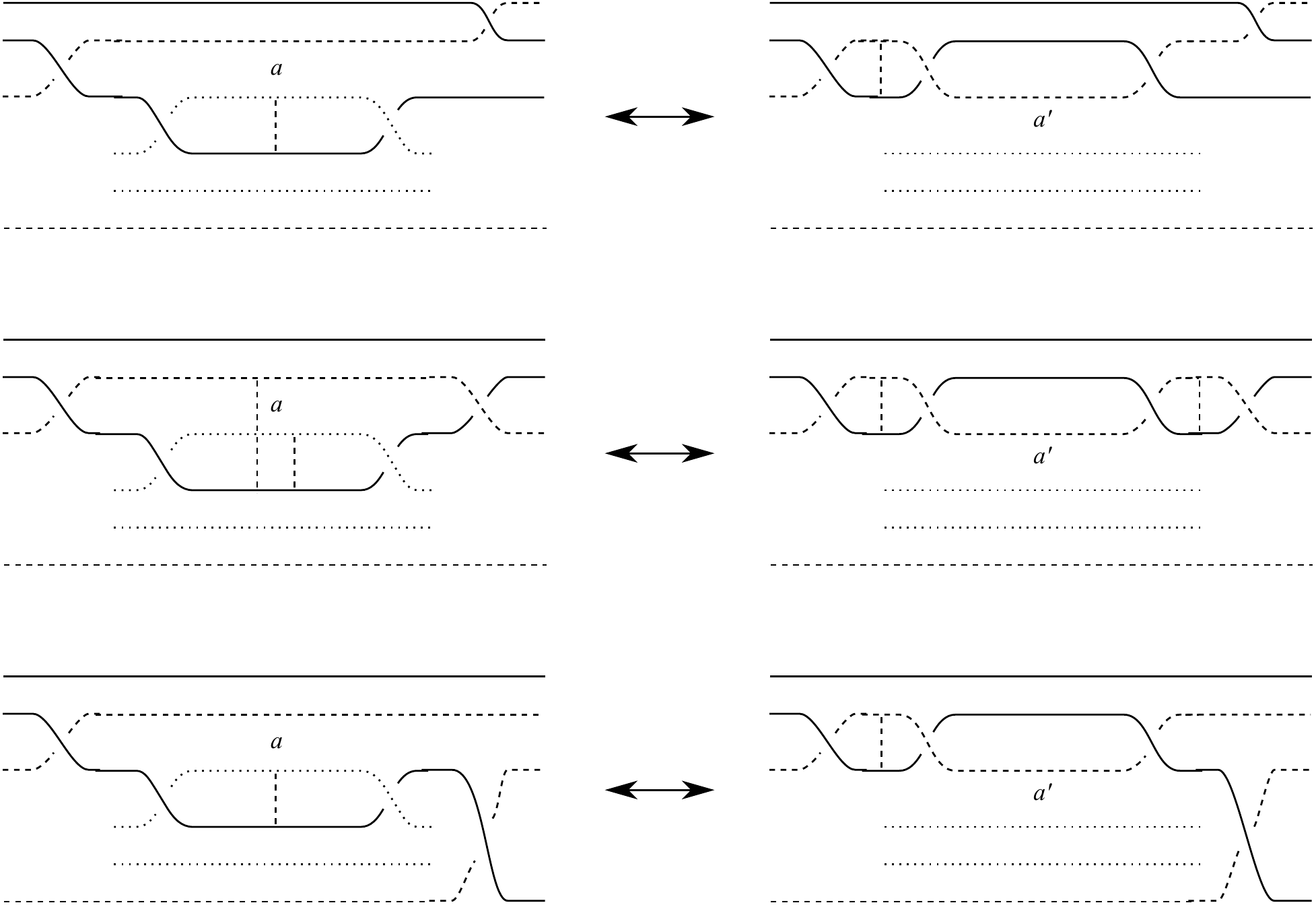}
	\end{center}
	\caption{Case 5 with clasps exhibited for each subcase (continue).}
	\label{rcc5b}
\end{figure}

\begin{figure}
	\begin{center}
		\includegraphics[width=6.3in]{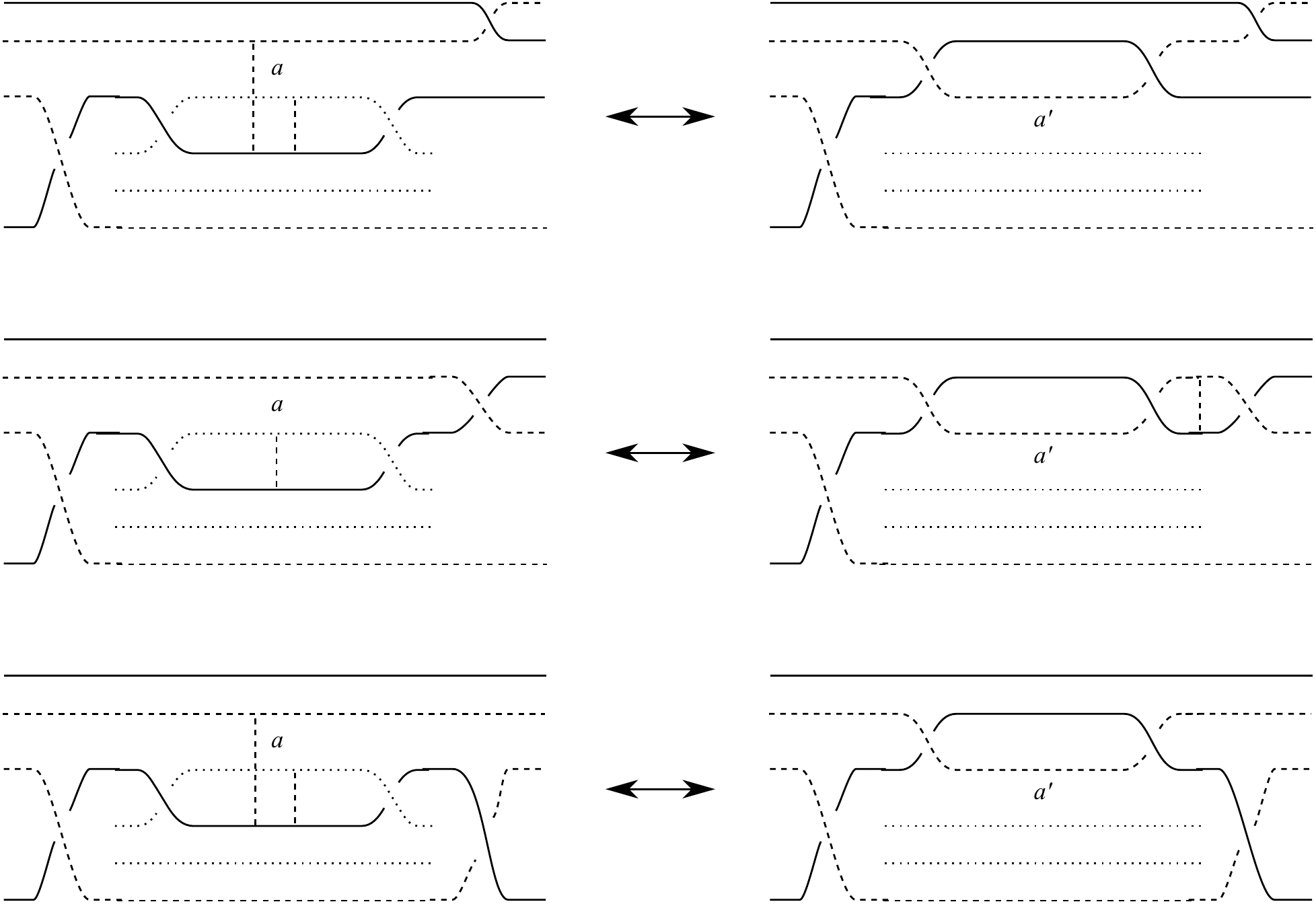}
	\end{center}
	\caption{Case 5 with clasps exhibited for each subcase (continue).}
	\label{rcc5c}
\end{figure}

\begin{figure}
	\begin{center}
		\includegraphics[width=6.3in]{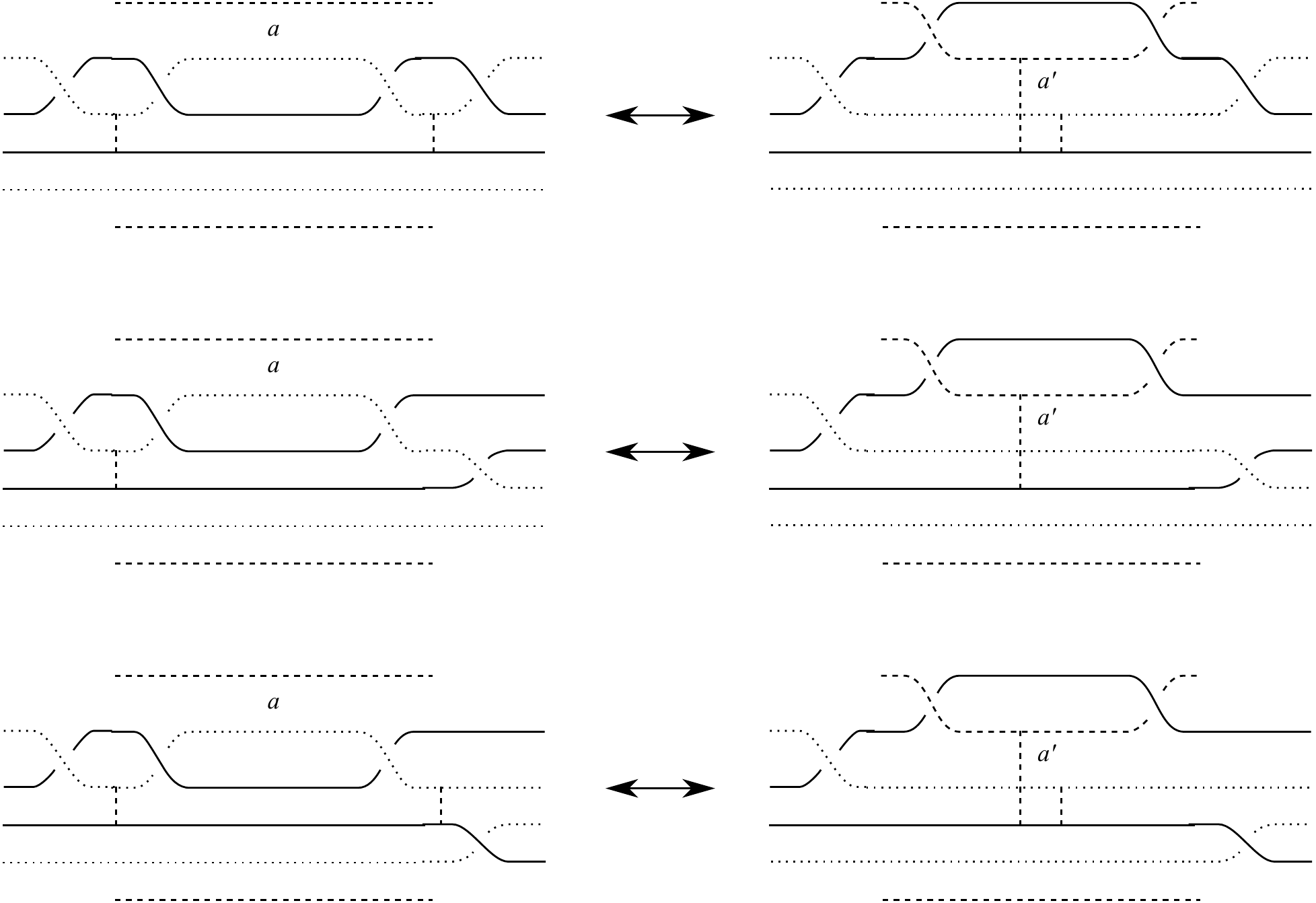}
	\end{center}
	\caption{Case 6 with clasps exhibited for each subcase.}
	\label{rcc6a}
\end{figure}

\begin{figure}
	\begin{center}
		\includegraphics[width=6.3in]{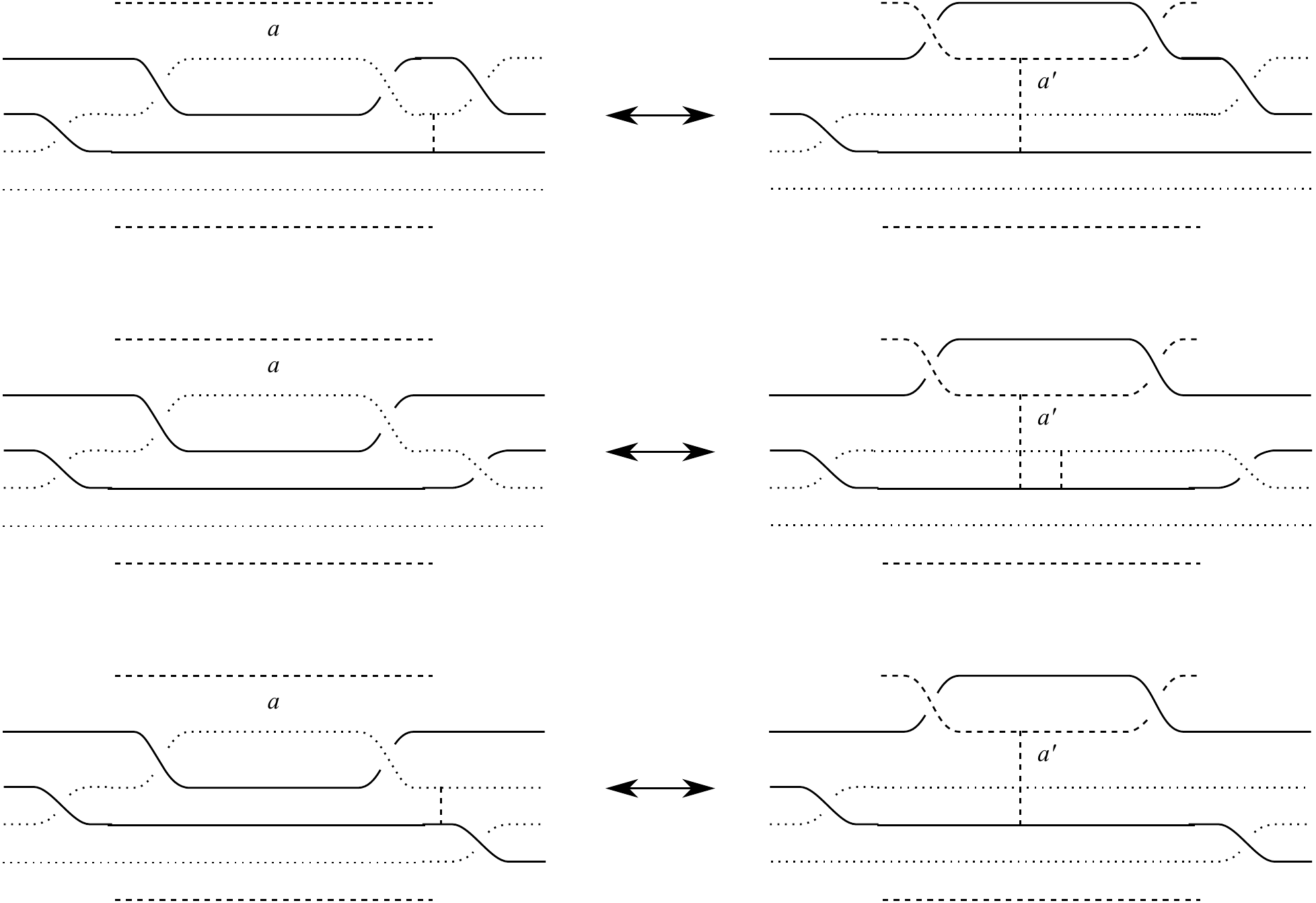}
	\end{center}
	\caption{Case 6 with clasps exhibited for each subcase (continue).}
	\label{rcc6b}
\end{figure}

\begin{figure}
	\begin{center}
		\includegraphics[width=6.3in]{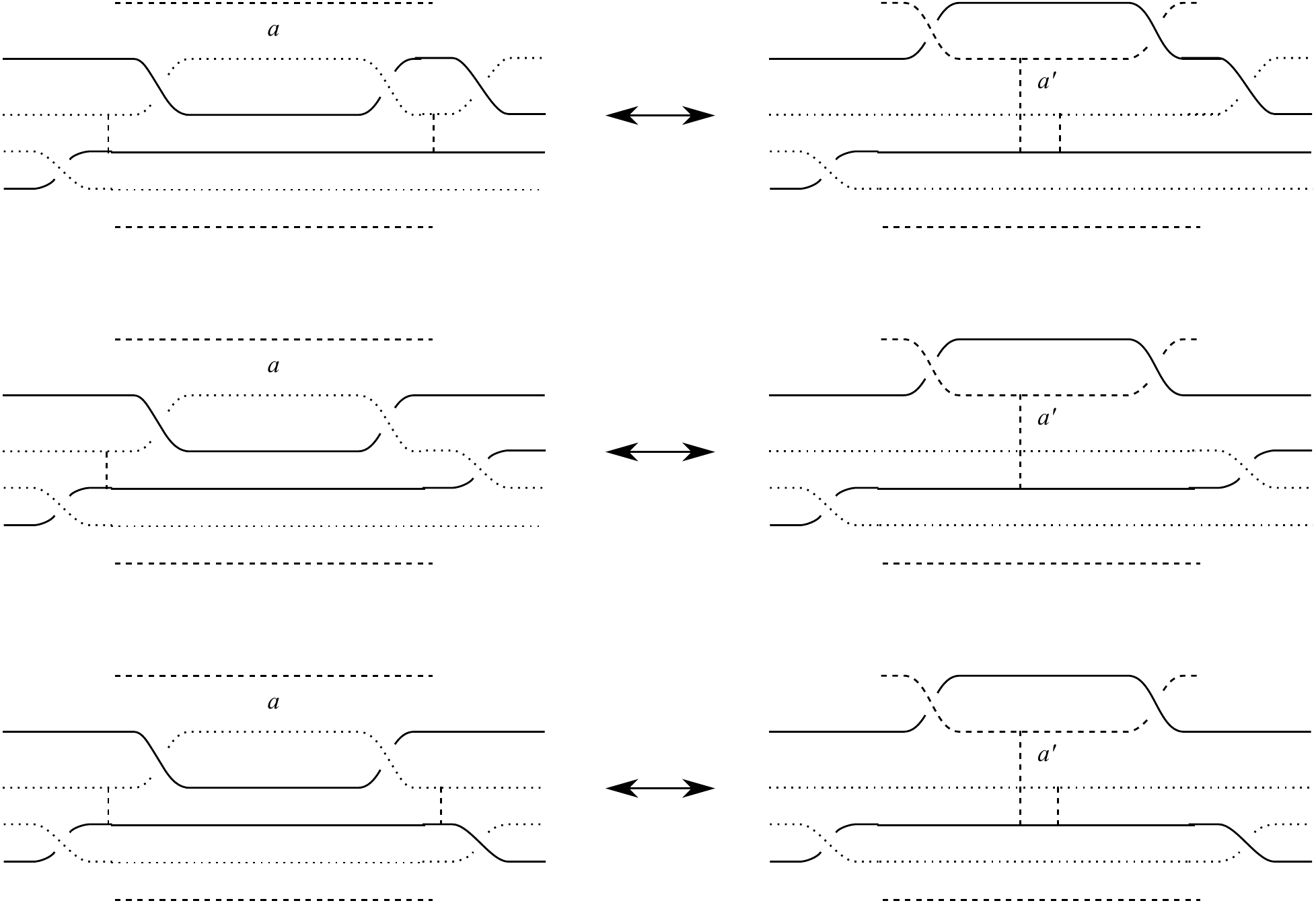}
	\end{center}
	\caption{Case 6 with clasps exhibited for each subcase (continue).}
	\label{rcc6c}
\end{figure}

\begin{figure}
	\begin{center}
		\includegraphics{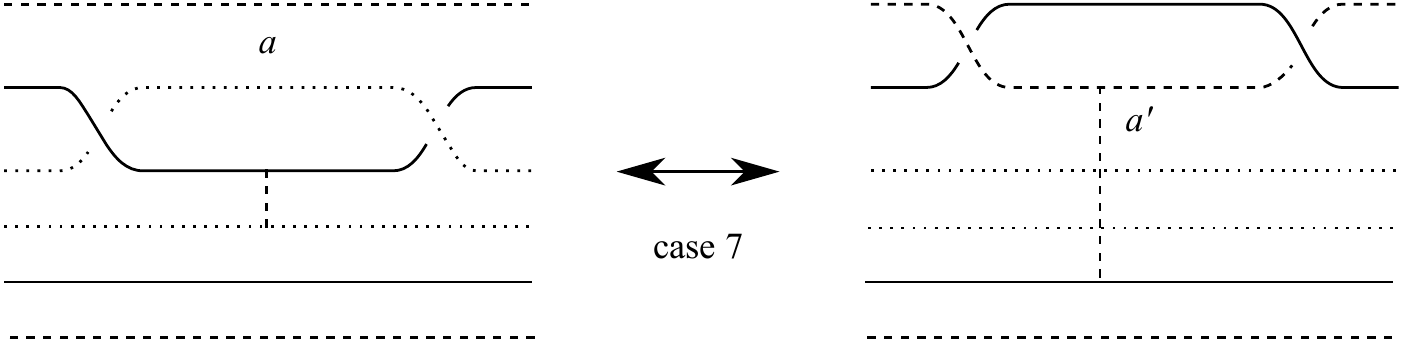}
	\end{center}
	\caption{Case 7 with its clasp exhibited.}
	\label{rcc7}
\end{figure}

\begin{figure}
	\begin{center}
		\includegraphics[width=6.3in]{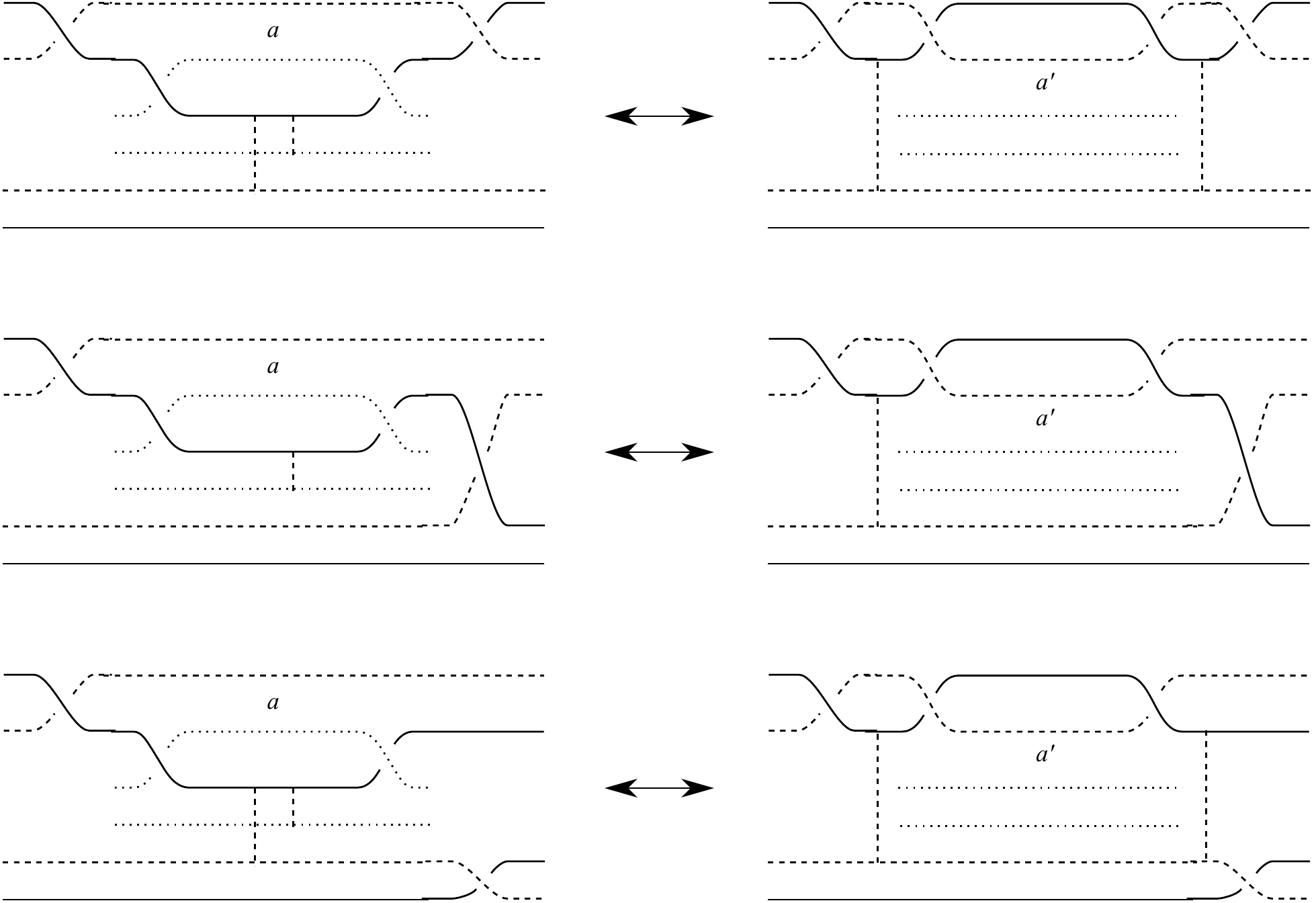}
	\end{center}
	\caption{Case 8 with clasps exhibited for each subcase.}
	\label{rcc8a}
\end{figure}

\begin{figure}
	\begin{center}
		\includegraphics[width=6.3in]{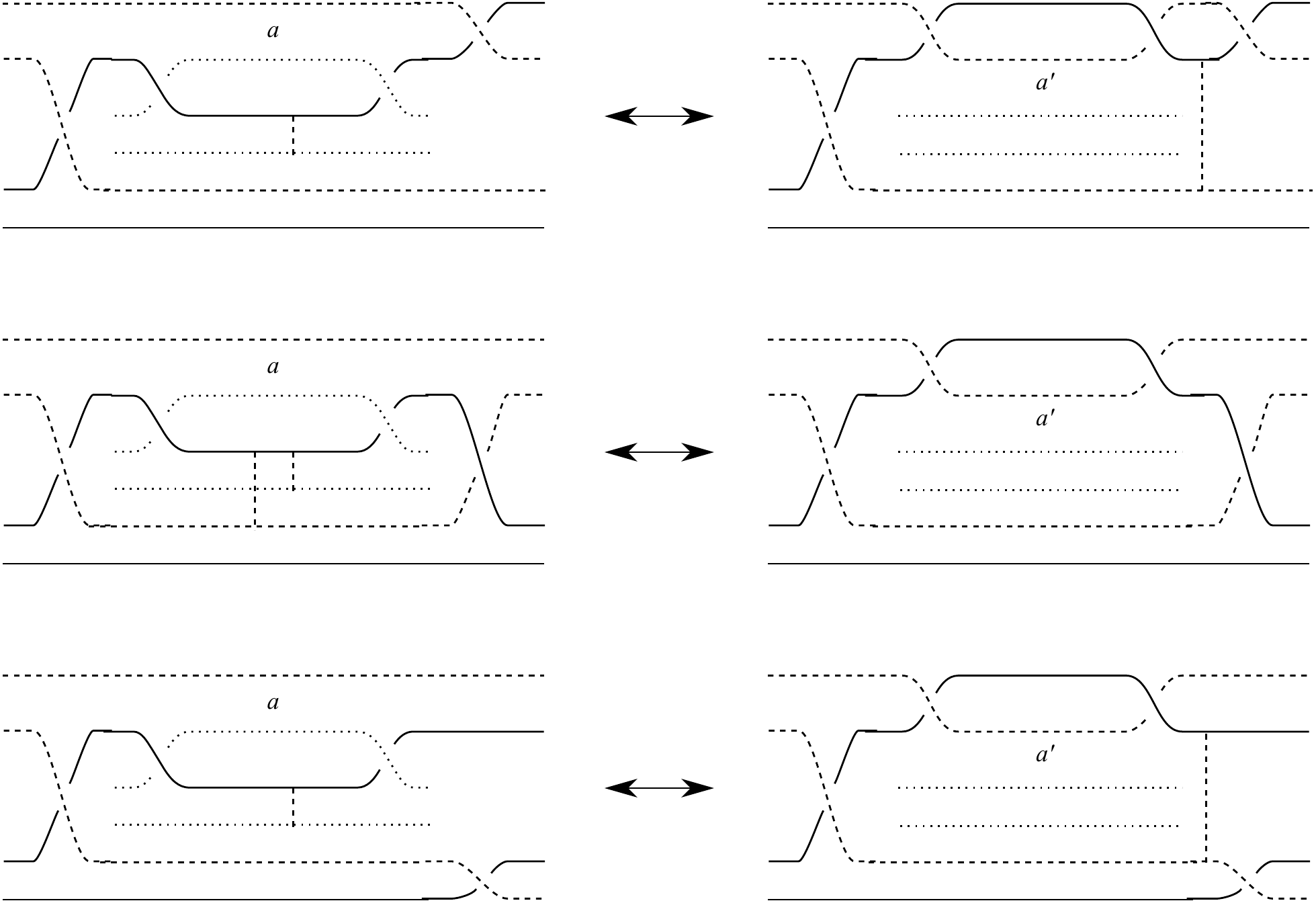}
	\end{center}
	\caption{Case 8 with clasps exhibited for each subcase (continue).}
	\label{rcc8b}
\end{figure}

\begin{figure}
	\begin{center}
		\includegraphics[width=6.3in]{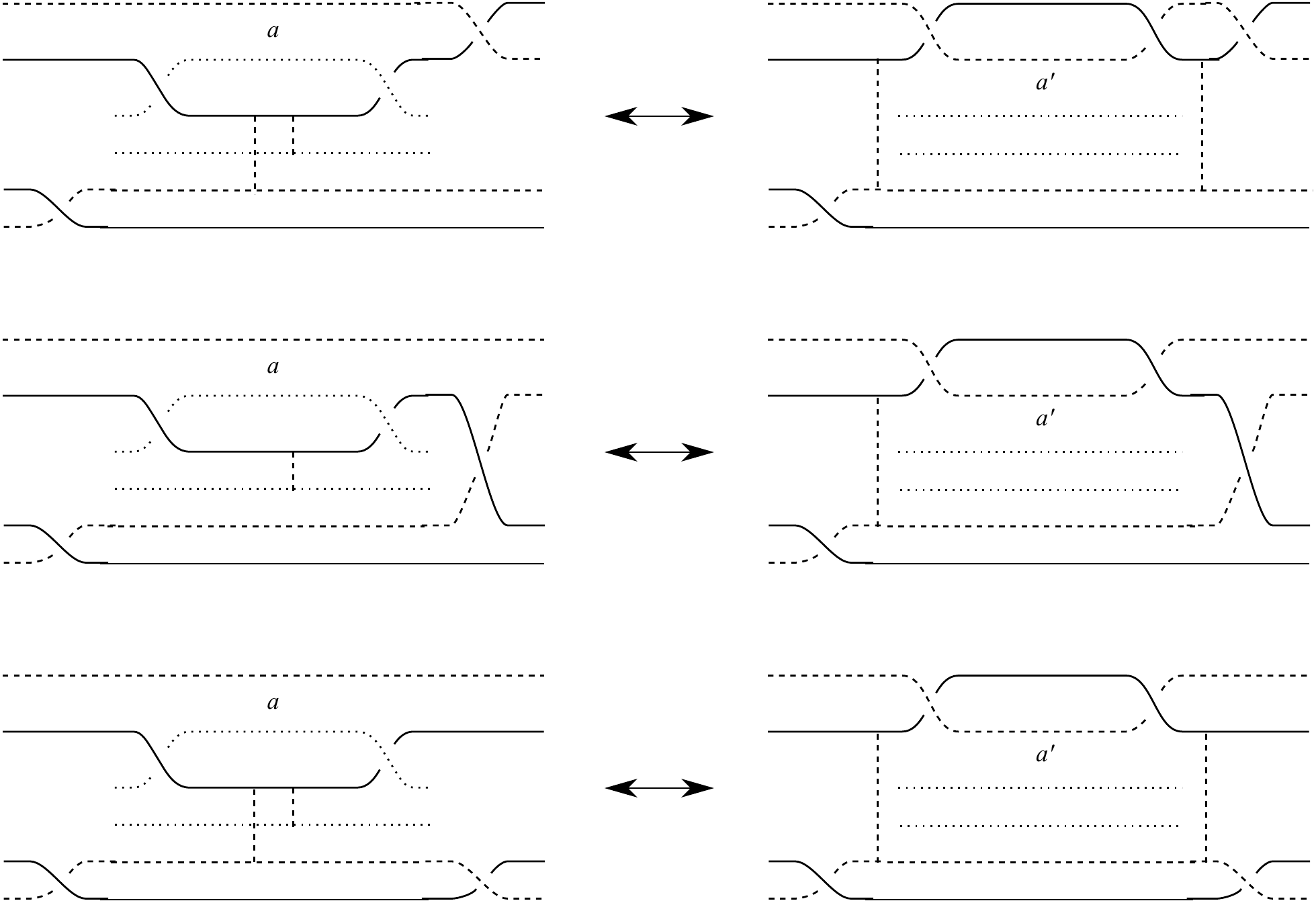}
	\end{center}
	\caption{Case 8 with clasps exhibited for each subcase (continue).}
	\label{rcc8c}
\end{figure}
\newpage
\textbf{3.5 Applications.}
As a consequence of Proposition \ref{p1} and Lemma \ref{l6}, we have our main theorem.

\begin{proof}[Proof of Theorem \ref{key}]
We prove by induction on the number of moves applied to $\emptyset$ to receive $K$. For base case, the first move must be 0-handle. This move clearly preserves the number of clasps, which is 0. Next, suppose that the statement is true for any sequence of moves with length $N$ or less. If the $(N+1)^{th}$ move is Legendrian isotopy, then the statement is true by Proposition \ref{p1} and Lemma \ref{l6}. Finally, if the $(N+1)^{th}$ move is 0-handle or 1-handle, then it preserves the number of clasps since the moves never create/destroy a block with non-nested component.
\end{proof}

In particular, we have the following corollaries.

\begin{Cor}
	\label{ckey}
Suppose a Legendrian link $K$ has a decomposable exact Lagrangian filling. Then $K$ must have at least 1 even normal ruling.
\end{Cor}

\begin{proof}
By Theorem \ref{key}, the associated normal ruling is even.
\end{proof}

Now, we may prove Theorem \ref{counter}.

\begin{proof}[Proof of Theorem \ref{counter}]
It is proved in \cite{Mep} that, for each $n \geq 0$, the $(4,-(2n+5))$-torus knot, as in Figure \ref{tgen}, has exactly 1 normal ruling, and the resolution of this normal ruling is illustrated in Figure \ref{tgenr}. Furthermore, this normal ruling has $2n+5$ clasps, as in Figure \ref{tgenrc}, so the knot does not have a decomposable exact Lagrangian filling by Corollary \ref{ckey}.
\end{proof}

\begin{figure}
	\begin{center}
		\includegraphics{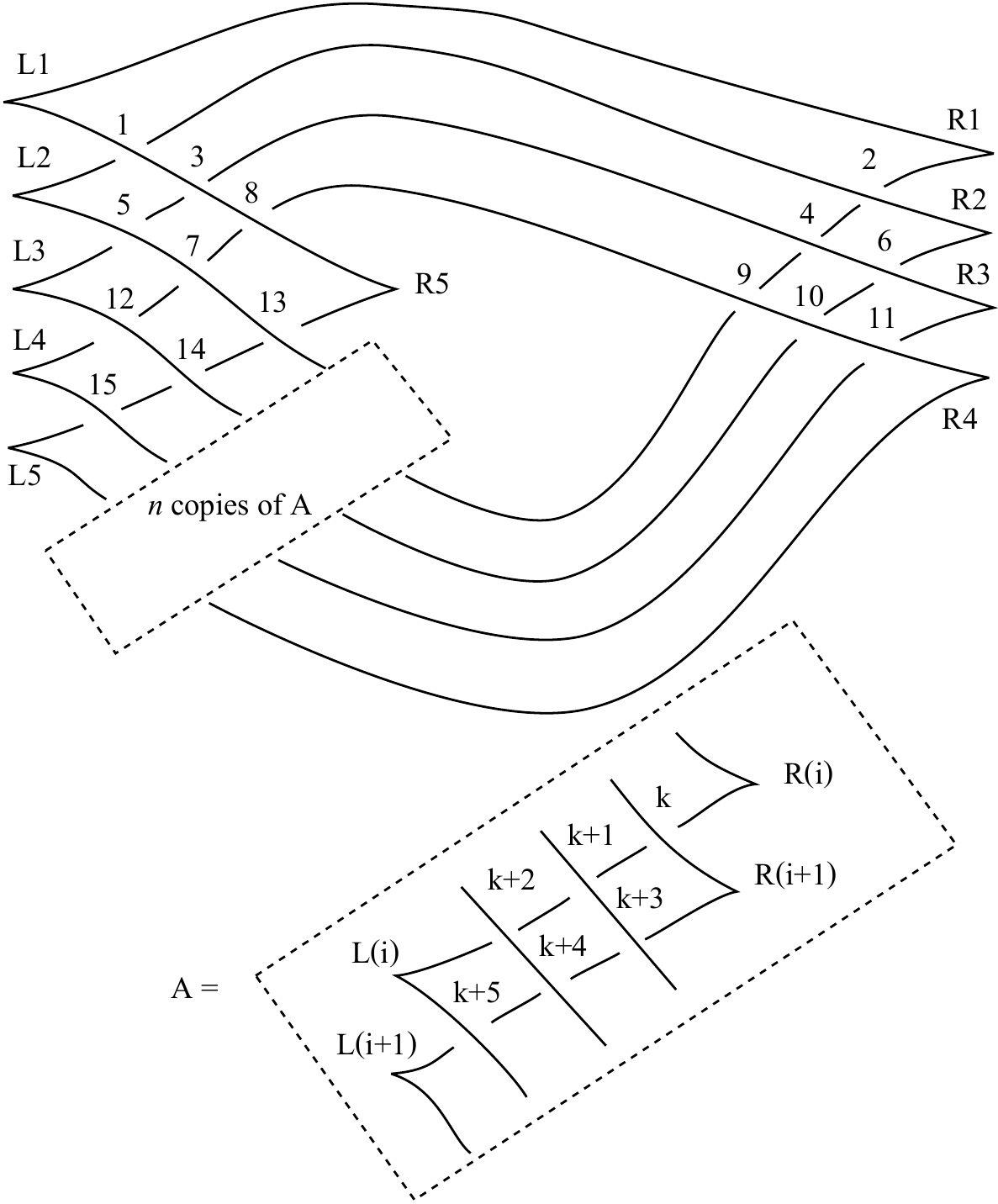}
	\end{center}
	\caption{The Legendrian $(4,-(2n+5))$-torus knot, $n \geq 0$.}
	\label{tgen}
\end{figure}

\begin{figure}
	\begin{center}
		\includegraphics{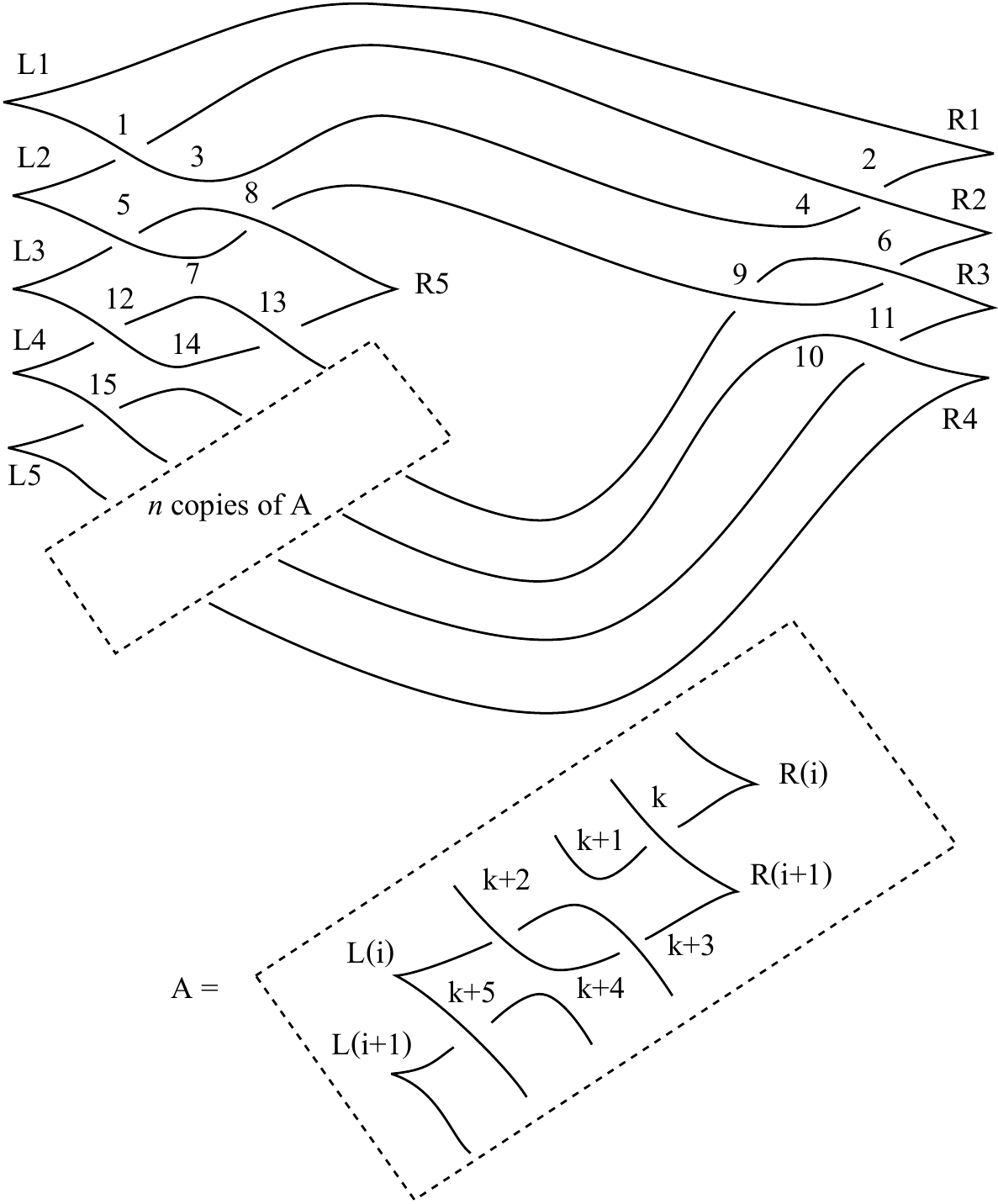}
	\end{center}
	\caption{The resolution of the only normal ruling of the Legendrian $(4,-(2n+5))$-torus knot, $n \geq 0$.}
	\label{tgenr}
\end{figure}

\begin{figure}
	\begin{center}
		\includegraphics{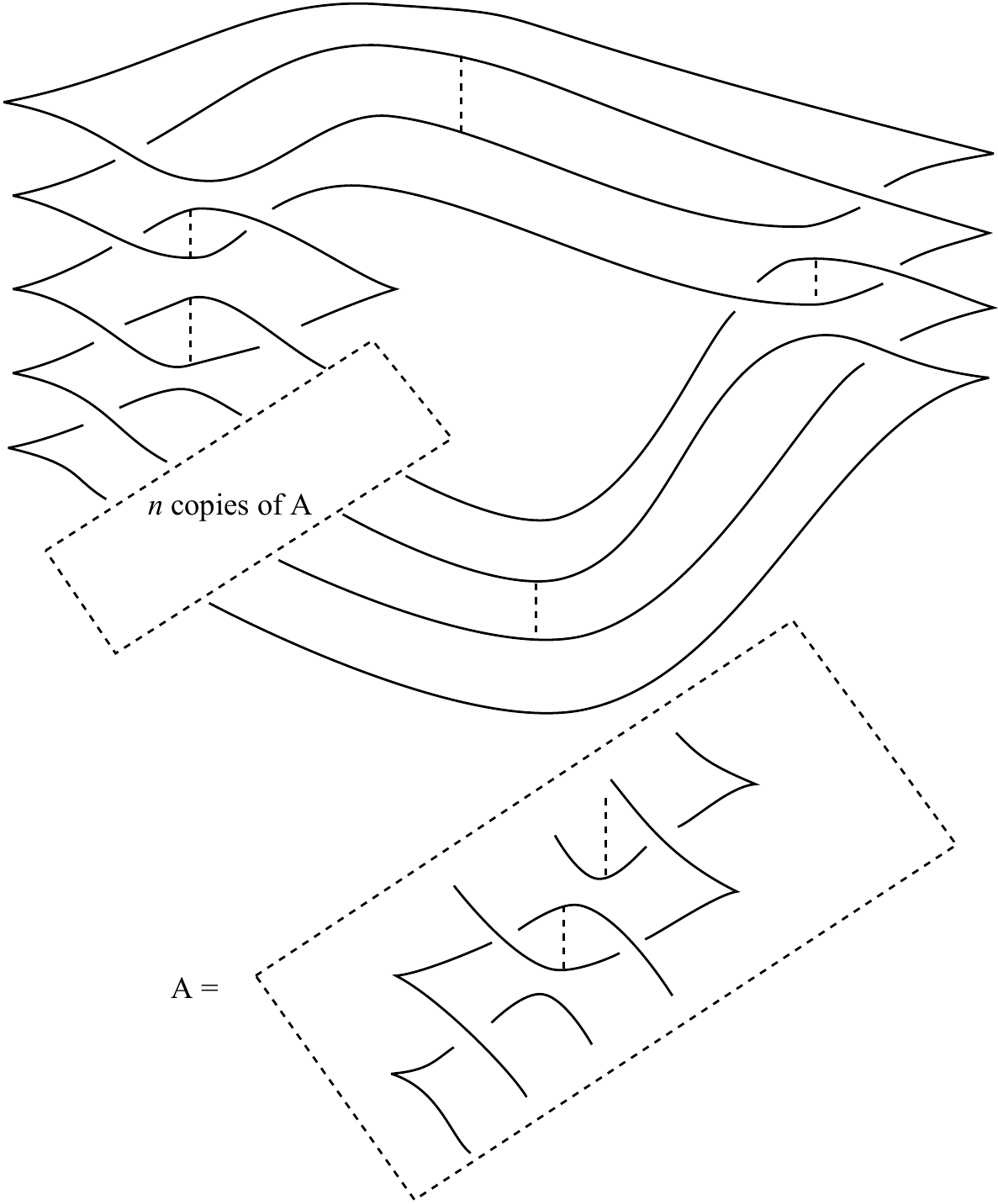}
	\end{center}
	\caption{The resolution of the only normal ruling of the Legendrian $(4,-(2n+5))$-torus knot, $n \geq 0$, and its clasps.}
	\label{tgenrc}
\end{figure}

Finally, we end this section with the following observations.

\begin{Cor}
Suppose we have two Legendrian links $K_+$ and $K_-$, both have exactly 1 normal ruling. If there is a decomposable exact Lagrangian cobordism from $K_-$ to $K_+$, then the parity of two normal rulings must be the same.
\end{Cor}

\begin{proof}
By induction, we may only consider when $K_+$ is obtain from $K_-$ via a single move from Theorem \ref{movie}. By Proposition \ref{p1} and Lemma \ref{l6} and the proof of Theorem \ref{key}, the parity of the associated normal ruling of $K_+$ must be the same as the parity of the only normal ruling of $K_-$. Since $K_+$ has 1 normal ruling, it must be the associated one.
\end{proof}

\begin{Cor}
The number of odd normal rulings is invariant under Legendrian isotopy. The same is also true for even normal rulings.
\end{Cor}

\begin{proof}
This is a result from Theorem \ref{Chekanov}, Proposition \ref{p1} and Lemma \ref{l6}.
\end{proof}

For example, the right-handed trefoil (see Figure \ref{trefoill}) has 3 normal rulings, which are \{1\}, \{3\} and \{1,2,3\}. It is easy to check that both \{1\} and \{3\} have 1 clasp. Also, \{1,2,3\} has 0 clasp. So the knot has 2 odd normal rulings and 1 even normal ruling.

\begin{figure}
\begin{center}
\includegraphics[width=2in]{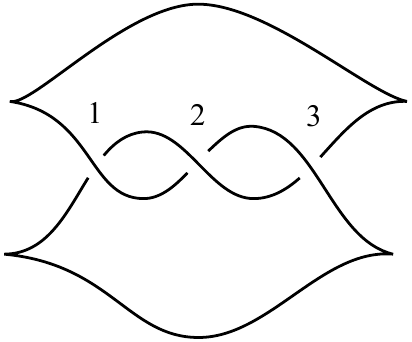}
\end{center}
\caption{The right-handed trefoil.}
\label{trefoill}
\vspace{2.5cm}
\end{figure}

\noindent 
Mathematics Department, Chiang Mai University\\
\textit{E-mail address}: atiponrat@hotmail.com

\end{document}